\documentclass[a4paper, 12pt]{amsart}

\usepackage{amssymb}
\usepackage{amsmath}
\usepackage{amsthm}
\usepackage{amscd}

\usepackage{comment}
\usepackage[all,cmtip]{xy}
\title[Inradius collapsed manifolds]{Inradius collapsed manifolds}
\date{\today}
\author{Takao Yamaguchi, Zhilang Zhang}

\address{Takao Yamaguchi, Department of Mathematics, Kyoto University, Japan}
\thanks{This work was supported by JSPS KAKENHI Grant Numbers 26287010, 15K13436, 15H05739}
\email{takaoy@math.kyoto-u.ac.jp}
\address{Zhilang Zhang, Graduate School of Pure and Applied Sciences, University of Tsukuba, Japan}
\curraddr{School of Mathematics and big data, Foshan University, Foshan city, China}
 \email{zhilangz@fosu.edu.cn}

\begin{document}
\begin{abstract}
In this  paper, we study collapsed manifolds with boundary, 
where we assume a lower  sectional curvature bound,  two sides bounds on the second fundamental forms of 
boundaries and upper diameter bound.  Our main concern is the case when inradii  of manifolds converge to zero. 
This  is a typical case of collapsing manifolds with boundary.
We  determine the limit spaces 
of  inradius collapsed manifolds as Alexandrov spaces with curvature uniformly bounded below. 
When the limit space has co-dimension one, we completely determined the topology of inradius collapsed manifold
in terms of singular $I$-bundles. Genral inradius collapse to almost regular spaces are also characterized.
In the general case of unbounded diameters, we prove that the number of boundary components of inradius collapsed manifolds is at most two,  where the disconnected boundary happens if and only if the manifold has a topological product structure.
%
%Our approach can be applied to the general case of non inradius collapse, where we obtain the structure of 
%limit spaces, stabilities of topological types and volumes. Finally using our structure results, we get an obstruction to 
%the general collapse.
\end{abstract}

\maketitle
%\date{\today}
%

\theoremstyle{plain}
  \newtheorem{thm}{Theorem}[section]
  \newtheorem{lem}[thm]{Lemma}
  \newtheorem{slem}[thm]{Sublemma}
  \newtheorem{cor}[thm]{Corollary}
  \newtheorem{prop}[thm]{Proposition}
  \newtheorem{clm}[thm]{Claim}
  \newtheorem{fact}[thm]{Fact}
  \newtheorem{ass}[thm]{Assertion}
  \newtheorem{obs}[thm]{Observation}

\theoremstyle{definition}
  \newtheorem{defn}[thm]{Definition}
  \newtheorem{conj}[thm]{Conjecture}
  \newtheorem{ex}[thm]{Example}
  \newtheorem{exc}[thm]{Exercise}
  \newtheorem{prob}[thm]{Problem}
  \newtheorem{asmp}[thm]{Assumption}
  \newtheorem{nota}[thm]{Notation}

\theoremstyle{remark}
  \newtheorem{rem}[thm]{Remark}
  \newtheorem{note}[thm]{Note}
  \newtheorem{summ}[thm]{Summary}
  \newtheorem{ack}{Acknowledgment}
 
\makeatletter

\renewcommand{\theequation}{%%%
\thesection.\arabic{equation}}

\newcommand{\wangle}[0]{\tilde{\angle}}
\newcommand{\diam}[0]{\mathrm{diam}\,}
\newcommand{\rad}[0]{\mathrm{rad}\,}
\newcommand{\e}[0]{\epsilon}
\newcommand{\de}[0]{\delta}
\newcommand{\vol}[0]{\mathrm{vol}}
\newcommand{\tr}[0]{\triangle}
\newcommand{\ctr}[0]{\tilde{\triangle}}
\renewcommand{\b}[1]{\tilde{#1}}
\newcommand{\dist}[0]{\mathrm{dist}}
\newcommand{\reg}[0]{\mathrm{reg}}
\newcommand{\supp}[0]{\mathrm{supp}}
\renewcommand{\theack}{\!\!\!}
\newcommand{\ca}{\mathcal}
\newcommand{\ve}{\varepsilon}
\newcommand{\ra}{\rightarrow}
\newcommand{\lra}{\longrightarrow}
\newcommand{\pa}{\partial}
\newcommand{\inrad}[0]{{\rm inrad}}

\newcommand{\Lip}[0]{\mathrm{Lip}}
\newcommand{\relmiddle}[1]{\mathrel{}\middle#1\mathrel{}}

\tableofcontents
\section{Introduction} \label{sec:intro}

%\subsection{Background}
In the present paper, we are concerned with collapsing phenomena of 
Riemannian manifolds with boundary under a lower sectional curvature bound.
The study of collapse of closed manifolds has a long history. 
In the case of two side bounds on sectional curvatures,  a deep general theory was established 
in  \cite{CFG}.  
Then for the case of lower sectional curvature bound, 
in \cite{Ya:pinching},  \cite{FY2}, \cite{KPT}, the structure of the first Betti numbers 
and the fundamental groups with their topological rigidity were determined through 
a fibration theorem.  Later on, those results were partly extended 
to the case of a lower Ricci curvature bound in \cite{ChCo1}, \cite{ChCo2},\cite{CoNa}, \cite{KaWi}.
Especially the general manifold structure results of  lower dimensional collapsed manifolds 
under a lower sectional curvature bound were established in \cite{SY:3mfd}, \cite{Ya:four}, \cite{SY:volume}.

In those results, it is crucial to study Alexandrov spaces with curvature bounded below
which appear as the Gromov-Hausdorff limit spaces. In particular, Perelman's topological stability 
theorem has played significant roles. 
In connection with the study of  Alexandrov spaces,
the collapsing phenomena of three-dimensional closed Alexandrov
spaces  with curvature bounded below has been classified in a recent work \cite{MY:dim3closed}.

For collapsing Riemannian manifolds with boundary, there is a pioneering work by 
J. Wong \cite{wong0}, \cite{wong2} on this subject after the investigation in the non-collapsing
and bounded curvature case due to \cite{Kod}, \cite{AKKLT}. 
In the study of convergence and collapsing Riemannian manifolds with boundary,
it is obvious that  the main problem is to control the boundary behavior in a geometric way.
It is in \cite{wong0} that a nice extension procedure over the boundary was first carried out to 
study collapsed manifolds with boundary under a lower sectional curvature bound.
The study of collapse of three-dimensional Alexandrov spaces with boundary
is now undergoing in the work \cite{MY2:dim3bdy}, where all the details of collapses 
will be made clear.  

In the present paper, partly motivated by \cite{MY2:dim3bdy}, 
we develop and extend results in \cite{wong2} to a great extent.
Let $\ca M(n,\kappa,\lambda, d)$ denote the set of all isometry classes of 
$n$-dimensional compact  Riemannian manifolds $M$ with boudary    
whose sectional curvature, second foundamental form and diameter satisfy
\[
      K_M \ge \kappa, \,\, |\Pi_{\partial M}|\le\lambda, {\rm diam}(M)\le d.
\]
%%%%%%%%%%%%%%%%%%%%%%%%%% The meaning of those bounds esp on \lambda  %%%%%%%%%%%%%%%%%%
Every Riemannian manifold in $\ca M(n,\kappa,\lambda,d)$ can be glued with a warped cylinder 
along their boundaries in such a way that the resulting space becomes an Alexandrov space with 
curvature bounded below having $C^0$-Riemannian sturcture and 
that its  boundary is totally geodesic (\cite{wong0}). 
%
%%%%%%%%%%%%%%%%%%%%%%%%　　Wong's main results (precptness and weak topological stability 
%%%%%%%%%%%%%%%%%%%%%%%%%%  under a lower volume bound      %%%%%%%%%%%%%%%%%%%%%%%%%%%%
Investigating such a cylindrical  extension,
Wong proved that $\ca M(n,\kappa,\lambda,d)$ is precompact 
with respected to the Gromov-Hausdorff distance.
He also proved that if $\ca M(n,\kappa,\lambda,d,v)$ denote the set of all 
elements $M\in\ca M(n,\kappa,\lambda,d)$ having volume
$\vol(M)\ge v>0$, then it contains only finitely many homeomorphism types.

%%%%%%%%%%%%%%%%%%%%%%%%%%%%%%%%%%%%%%%%%%%%%%%%%%%%%%%%%%%%%%%%%%%%%%%%%
%\subsection{Problem}

Under the situation above, the main problem we are concerned in this paper is 
as follows:

\begin{prob}
Let $M_i$ be a sequence in  $\ca M(n,\kappa,\lambda,d)$ converging to a 
length space $N$  with respected to the Gromov-Hausdorff distance.
 \begin{enumerate} 
  \item Characterize the structure of $N;$
  \item Find geometric and topological relations between $M_i$ and $N$ for large enough $i$.
 \end{enumerate}
\end{prob}

The {\it inradius} of $M$ is defined as the largest radius of metric ball contained in 
the interior of $M$:
\[
                    \inrad(M):=\sup_{x\in M}{d(x,\pa M)}.
\]
In the  present paper, we first consider the case of  ${\rm inrad}(M_i)$ converging to zero.
We  prove in Corollary \ref{cor:dim-collaps} that if  ${\rm inrad}(M_i)$ converges to zero, then 
$M_i$ actually dimension collapses in the sense that any limit space $N$ has dimension 
\[
            \dim N \le n-1.
\]
Therefore in this case , we say that $M_i$ {\it inradius collapses}. 
The inradius collapse is a typical case of collapsing of manifolds with boundary.
Actually in the forthcoming paper \cite{YZ:general}, we show that if  a sequence  $M_i$ in  $\ca M(n,\kappa,\lambda,d)$ converges to 
a topological closed manifold or a closed Alexandrov space, then $M_i$ inradius 
collapses. % (Proposition \ref{prop:inradius}).

%%%%%%%%%%%%%%%%%%%%%%%%%%%%%%%%%%%%%%%%%%%%%%%%%%%%%%%%%%%%%%%%%%%%%%%%%%%%%%%%%%%%%%%%%
%\subsection{Main results}

%%%%%%%%%% inradius collapse with bounded diameter %%%%%%%%%%%%%%%%%%%%%%%%%%%%%%%%%%%%%%

The main results in this paper are stated as follows. 
The first one is about the limit spaces of inradius collapse.

\begin{thm} \label{thm:limit-alex}
Let $M_i\in\ca M(n,\kappa,\lambda,d)$ inradius collapses to a length space  $N$ with respect to
the Gromov-Hausdorff distance.  Then 
$N$ is an Alexandrov space with curvature $\ge  c(\kappa,\lambda)$, where
$c(\kappa,\lambda)$ is a constant depending only on $\kappa$ and $\lambda$.
\end{thm}

It should be noted that $M_i$ are not Alexandrov spaces unless ${\rm II}_{\partial M_i}\ge 0$,
and that the constant  $c(\kappa,\lambda)$ really depend on both $\kappa$ and $\lambda$.
Moreover, if one assume only ${\rm II}_{\pa M_i}\ge -\lambda^2$ or ${\rm II}_{\pa M_i}\le \lambda^2$
instead of $|{\rm II}_{\pa M_i}|\le \lambda^2$, there is a counterexamples to Theorem \ref{thm:limit-alex}
(see Examples \ref{ex:counter2},  \ref{ex:counter} and \ref{ex:counter3}).

Let $\ca M(n,\kappa,\lambda)$ denote the set of all isometry classes of 
$n$-dimensional complete Riemannian manifolds $M$ satisfying
\[
      K_M \ge \kappa, \,\, |{\rm II}_{\partial M}|\le\lambda.
\]
This family is also precompact with respect to the pointed Gromov-Hausdorff convergence.
Theorem \ref{thm:limit-alex}  actually holds true for the limit of manifolds in 
$\ca M(n,\kappa,\lambda)$ with respect to the pointed Gromov-Hausdorff convergence
(see Theorem \ref{thm:limit-alex'}).

Next we discuss the topological structure of inradius collapsed manifolds.
First consider  the case of inradius collapse of codimension one.
%In this case we define two types of models of the singularities around boundary component of 
%the limit space, {\em the product or the twisted singular $I$-fiber bundle} (see Definition \ref{def:model}).
We can give a complete characterization of codimension one inradius collapsed manifolds as follows.
Let $D^2_+$ be the upper half disk on $xy$-plane, and 
$J:=D^2_{+}\cap \{ y=0\}$.

\begin{thm}\label{thm:codim1}
Let $M_i\in\ca M(n,\kappa,\lambda, d)$ inradius collapse to an $(n-1)$-dimensional 
Alexandrov space $N$.  Then there is a  singular $I$-fiber bundle:
\[
      I  \rightarrow  M_i  \overset{\pi} \rightarrow   N
\]
whose singular locus coincides with $\partial N$, 
and $M_i$ is a gluing of $I$-bundle $N\tilde\times I$ over $N$ and $D^2_{+}$-bundle
$\pa N \tilde\times D^2_+$ over $\pa N$ , 
\[
      M_i = N\tilde\times I\cup \pa N \tilde\times D^2_+,
\]
where the gluing is done via $\pa N \tilde\times I = \pa N \tilde\times J$,
and $\tilde\times$ denotes either the product  or a twisted product.

In particular $M_i$ has the same homotopy type as $N$.
%
%More precisely, 
%\begin{enumerate}
% \item If  $N$ has no boundary, then  $M_i$ is homeomorphic to a product $N\times I$ 
%          or a twisted product $N  \mathbin{\stackrel{\sim}{\times}} I;$
% \item If  $N$ has non-empty boundary, each component $\partial_{\alpha}N$ of $\partial N$
%         has a neighborhood $V$ such that  $\pi^{-1}(V)$ is isomorphic to either 
%         the product or the twisted singular $I$-fiber bundle around  $\partial_{\alpha}N$;
% \item If  $\pi^{-1}(V)$ is isomorphic to  the product singular $I$-fiber bundle for some component  
%         $\partial_{\alpha}N$,
%         then $M_i$ is homeomorphic to $D(N)\times [-1,1]/(x,t)\sim (r(x), -t)$, where $r$ is the canonical
%         reflection of the double $D(N)$.
%\end{enumerate}
\end{thm}

%Here $\Sigma_x(N)$ denotes the space of directions of $N$ at $x$,
%$\mathbb S^n$ stands for the unit $n$-sphere.
Next, we consider inradius collapse to almost regular spaces.
We say that an Alexandrov space $N$ is {\em almost regular} 
if any point of  $N$  has the space of directions 
whose volume is close to $\vol\,\mathbb{S}^{\dim N-1}$,
where $\mathbb S^m$ denotes the unit $m$-sphere.

\begin{thm}\label{thm:RMBcap}
Let a sequence $M_i$ in $\ca M(n,\kappa,\lambda, d)$ inradius collapse
to an Alexandrov space $N$, and suppose that 
the limit of $\pa M_i$ is almost regular and 
%\[
%         \vol(\Sigma_x(N)) >
% 
%\]
\[
   \vol(\Sigma_x(N)) >\frac{1}{2}\vol  \mathbb S^{m-1}
\]
for all $x\in N$.
% which is almost regular. 
Then the topology of $M_i$ can be classified 
into the following two types:
\begin{itemize}
\item[(a)] There exists a locally trivial fiber bundle
                       $$F_i\times I\to M_i\to N,$$
  where $F_i$ is a closed almost nonnegatively curved manifold  in a generalized sense as in \cite{Ya:pinching};
\item[(b)] There exists a locally trivial fiber bundle
                       $${\rm Cap}_i \to M_i\to N,$$
  where ${\rm Cap}_i$ (resp. $\partial\, {\rm Cap}_i$) is an  almost nonnegatively curved manifold with boundary
( resp.  a closed almost nonnegatively curved connected manifold) in a generalized sense as in \cite{Ya:pinching}.
\end{itemize}
\end{thm}

In general, the almost regularity of $N$ implies  that of the limit of $\pa M_i$ (Proposition \ref{prop:quotient}).
However the converse if not true (see Example \ref{ex:D}). 

It should also be pointed out that several fibration theorem were obtained in  \cite{wong2}
in some cases, where the nonnegativity of the second fundamental form ${\rm II}_{\pa M_i}\ge 0$, 
or the upper bound $K_{M_i}\le \kappa^2$ and the lower bound for the injectivity radius
${\rm inj}(M_i)\ge i_0>0$ were assumed.
%
%Theorem \ref{thm:RMBcap} solves a problem raised in,p.297, without assuming
%the upper sectional curvature bound. 

%%%%%%%%%% inradius collapse with unbounded diameter %%%%%%%%%%%%%%%%%%%%%%%%%%%%%%%%%%%%%%

Next we discuss the number of boundary components of inradius collapsed
manifolds,
where we do not assume the diameter bound.

\begin{thm} \label{thm:two}
There exists a positive number $\epsilon=\epsilon_n(\kappa, \lambda)$
such that if $M$ in  $\ca M(n,\kappa,\lambda)$  satisfies   ${\rm inrad}(M)<\epsilon$,
then
\begin{enumerate}
 \item the number $k$ of connected components of $\partial M$ is at most two;
 \item if $k=2$, then $M$ is diffeomorphic to $W\times [0,1]$, where $W$ is a component of 
         $\partial M$.
\end{enumerate}
\end{thm}

Theorem \ref{thm:two} (1) was stated in \cite[Theorem 5]{wong2}. 
However it seems to the authors that the argument there is unclear  (see Remark \ref{rem:unclear}).
Theorem \ref{thm:two} may be considered as a generalization of  a result  
in Gromov\cite{G:synthetic} and Alexander and Bishop\cite{AB}, where an $I$-bundle structure was found for an inradius collapsed manifold
under the two-sides bound on sectional curvature.
It should be pointed out that the constants $\epsilon(\kappa, \lambda)$  in \cite{G:synthetic} and \cite{AB} are 
explicit and independent of $n$ while our constant  $\epsilon_n(\kappa, \lambda)$ is neither. This is because our 
argument is by contradiction.

The organization and the outline of the proofs are as follows.

In section \ref{sec:prelim},  we first recall basic notions and facts on the Gromov-Hausdorff convergence
and Alexandrov spaces with curvature bounded below. Then we focus on Wong's extension procedure of 
a Riemannian manifold with boundary by gluing a warped cylinder along their boundaries.
By \cite{Kos}, 
the result of the gluing is a $C^{1,\alpha}$-manifold with $C^0$-Riemannian metric, and becomes an 
Alexandrov spaces with curvature uniformly bounded below. 
This construction is quite effective and used in an essential way in the present paper.

In section \ref{sec:str-limit}, we describe limit spaces of glued Riemannian 
manifolds with boundary. The limit spaces also have gluing structure.
In this section we focus on the estimate of multiplicities of gluing, 
the intrinsic metric structure of the limit space and a general description 
of the limit spaces of extensions.

In Section \ref{sec:metric}, we determine the metric structure of limit spaces.
First we study the spaces of directions of the limit space at gluing points, and 
prove that the gluing map preserves the length of curves. This implies that 
the gluing in the limit space is done metrically  in a natural manner, 
and yields significant structure results (see Theorem \ref{thm:singIbund}) on the limits 
including Theorem \ref{thm:limit-alex}.

Those structure results are applied in Section \ref{sec:fib} to obtain 
the fiber structures of inradius collapsed manifolds. Theorems \ref{thm:codim1} and \ref{thm:RMBcap}
are proved there. To prove Theorem \ref{thm:codim1}, we need to analyze the singularities of
the singular $I$-fiber bundle in details.
To prove Theorem \ref{thm:RMBcap}, we apply an equivariant fibration-capping 
theorem in \cite{Ya:four}.

To prove Theorem \ref{thm:two}, we consider the case of unbounded diameters in Section \ref{sec:unbounded}.
Applying the results in Section \ref{sec:metric}, we obtain basically three types 
on local connectedness of the boundary of an inradius collapsed complete manifold, accdording to the 
types of  the local limit spaces.
After such local obsrvation, Theorem \ref{thm:two} follows from a
monodromy argument.

Our approach can be applied to the general case of non inradius collapse of 
Riemannian manifolds with boundary. As a continuation of the present paper,
in \cite{YZ:general}, we obtain  the structure of 
limit spaces, stabilities of topological types and volumes,  in the general framework 
of non inradius collapse/convergence, and get an obstruction to 
the general collapse.

%%%%%%%%%%%%%%%%%%%%%%%%%%%%%%
%In  Section \ref{sec:non-inradius}, we consider the  convergence where the inradii have a positive lower 
%bound. In this case,  we investigate the  "boundary" of the limit space which has the structure of 
%an Alexandrov space with curvature uniformly bounded below outside the "boundary singular point set".
%% if it has no singular points.
%In case there are boundary singular points, we classify them into type $1$ and $2$ singular points.
%We prove that the type $1$ singular set has measure zero in the boundary. 
%The boundary of the limit space is the limit of boundaryies of Riemannian manifolds
%in the sequence under consideration. 
%In the case when the boundary singular points consists of only type $2$, we obtain 
%a fiber bundle theorem like Theorem \ref{thm:gen-fib}.   
%
%In  Section \ref{sec:general}, we discuss the general collapsing/convergence in $\mathcal M(n,\kappa,\lambda, d)$.
%Theorem \ref{thm:w-stability} and Corollary \ref{cor:simplicial} are proved there.
%
%%
%%We show the uniform boundedness of ratios of volumes of boundary components 
%%of Riemannian manifolds in $\ca M(n,\kappa,\lambda, d)$. 
%% 

%%%%%%%%%%%%%%%%%%%%%%%%%%%%   Preliminaries   %%%%%%%%%%%%%%%%%%%%%%%%%%%%%%%%
\section{Preliminaries}\label{sec:prelim}

%%%%%%%%%%%%%%%%%%%Notations and conventions  　%%%%%%%
%\subsection{Notations and conventions}
In order to make the presented paper more accessible, we fix some basic definition, notations and conventions.
\begin{itemize}
\item $\tau(\delta)$ is a function which depends on $\delta$ such that $\lim_{\delta\to0}\tau(\delta)=0$.
\item For topological spaces $X$ and $Y$, $X\approx Y$ means $X$ is homeomorphic to $Y$.
\item The distance between two points $x,y$ in a metric space is denoted by $d(x,y)$, $|x,y|$ or $|xy|$.
\item For a point $x$ and a subset $A$ of a metric space $X$, $B(x, r)=B^X(x,r)$ and $B(A, r)=B^X(A,r)$ 
denote open $r$-balls in $X$ around $x$ and $A$ respectively.
%The diameter of  a metric space $X$ is denoted by $d(X)$ or $diam(X)$.
\item For a metric space $(X,d)$, and $r>0$, the rescaled metric space $(X,rd)$ is denoted by $rX$.

\item  The Euclidean cone $K(\Sigma)$ over a metric space $(\Sigma,\rho)$ is $\Sigma\times[0, \infty)$ 
        equipped with the metric $d$ defined as 
$$
        \hspace{1.3cm}   d((x_1,t_1),(x_2,t_2))=(t_1^2 + t_2^2-2t_1t_2\cos(\min\{\rho(x_1,x_2),\pi\}))^{1/2},
$$
for  any two points $(x_1,t_1), (x_2,t_2) \in \Sigma\times[0,\infty)$. 
\item For a subspace $M$ of a metric space $(\tilde M, d_{\tilde M})$,  $M^{\rm ext}$ denotes $(M,d_{\tilde M})$,
        which is called the exterior metric of $M$.
\item 
The metric $d$ of  a connected metric space $(X,d)$ induces a length metric $d_{\rm int}$ of  $X$
defined as the infimum of the length of all curves joining given two points.
We denote by $X^{\rm int}$ the new metric space  $(X, d_{\rm int})$. 
\item The length of a curve  $\gamma$ is denoted by $L(\gamma)$.% or $|\gamma|$.
\end{itemize}

%%%%%%%%%%%%%%%% Gromov-Hausdroff convergence %%%%%%%%%

\subsection{The Gromov-Hausdroff convergence} \label{subsec:GH}

%\begin{defn}
A (not necessarily continuous) map $f:X\ra Y$ between two metric spaces $X$ and $Y$ is called an
 {\em $\ve$-approximation} if it satisfies
 \begin{enumerate}
\item  $|d(x,y)-d(f(x),f(y))|<\ve$, for all $x,y\in Y$,
\item $f(X)$ is $\ve$-dense in $Y$, i.e.,  $B(f(X),\ve)=Y$.
\end{enumerate}
The Gromov-Hausdorff distance $d_{GH}(X,Y)$ is defined as the infimum of those $\ve$ such that 
there are $\ve$-approximations $f:X\to Y$ and $g:Y\to X$.

%\end{defn}

%\begin{defn}
 A map $f:(X,x)\ra (Y,y)$ between two pointed metric spaces is called a
 {\em pointed $\ve$-approximation} if it satisfies
 \begin{enumerate}
 \item $f(x)=y$,
\item  $|d(x,y)-d(f(x),f(y))|<\ve$, for all $x,y\in B^X(x,1/\ve)$,
\item  $f(B^X(x,1/\ve))$ is $\ve$-dense in  $B^Y(y,1/\ve)$.    
\end{enumerate}
%\end{defn}
The pointed Gromov-Hausdorff distance $d_{pGH}((X,x), (Y,y))$ is defined as the infimum of those $\ve$ such that 
there are pointed $\ve$-approximations $f:(X,x)\to (Y,y)$ and $g:(Y,y)\to (X,x)$.

%\begin{defn} 
Consider a pair $(X,\Lambda)$ of a metric space $X$ and a group $\Lambda$ of isometries of $X$.
For such pairs $(X,\Lambda)$, $(Y,\Gamma)$, a triple $(f, \varphi, \psi)$ of maps $f:X\to Y$, 
$\varphi:\Lambda\to\Gamma$ and $\psi:\Gamma\to\Lambda$ is called an 
\emph{equivariant $\ve$-approximation} from  $(X,\Lambda)$ to $(Y,\Gamma)$  if the following holds
 \begin{enumerate}
\item  $f$ is an $\ve$-approximation,
\item  if $\lambda\in\Lambda$, $x\in X$,
         then $d(f(\lambda x), (\varphi\lambda)(fx))<\ve$,
\item  if $\gamma\in\Gamma$, $x\in X$,
         then $d(f(\psi(\gamma)x),\gamma(fx))<\ve$.
\end{enumerate}
The {\em equivariant Gromov-Hausdorff distance} $d_{eGH}((X, \Lambda), (Y, \Gamma))$ is defined as the infimum 
of those $\ve$ such that there are $\ve$-approximations from  $(X,\Lambda)$ to $(Y,\Gamma)$
and from   $(Y,\Gamma)$ to $(X,\Lambda)$.
%\end{defn}

%%%%%%%%%%%%%%%%%%%%% Alexandrov spaces   %%%%%%%%%%%%%%%%%%%%%%%%%%%%%%%%%%%%%%
\subsection{Alexandrov spaces}\label{ssec:Alex}
Let $X$ be a geodesic metric space,  where every two points of $X$ can be joined by a
shortest geodesic. For a fixed real number $\kappa$ and a geodesic triangle
$\Delta pqr$ in $X$ with vertices $p$, $q$ and $r$,
denote by $\tilde\Delta pqr$ 
 a {\em comparison triangle} in the complete simply connected model surface $M_\kappa^2$ with constant curvature
$\kappa$.
 This means that $\tilde\Delta pqr$ has the same side lengths as the corresponding ones in $\Delta pqr$.
Here we suppose that the perimeter of $\Delta pqr$ is less than $2\pi/\sqrt\kappa$
if $\kappa> 0$.
The metric space $X$ is called an {\em Alexandrov space
with curvature $\geq \kappa$},
sometimes Alexandrov space for short if we do not  emphasis the lower curvature bound,
if each point of $X$
has a neighborhood $U$ satisfying the following:
For any geodesic triangle in $U$ with vertices
$p$, $q$ and $r$ and for any point $x$ on the segment $qr$,
we have $|px| \geq |\tilde{p}\tilde{x}|$,
where $\tilde x$ is the point on $\tilde{q}\tilde{r}$ corresponding to $x$.
From now on we assume that an Alexandrov space is always finite dimensional.

 For an Alexandrov space $X$ with curvature bounded below by $\kappa$, 
let $\alpha:[0,s_0]\to X$ and $\beta:[0,t_0]\to X$ be two geodesics starting from a point $x$.
The {\em angle} between $\alpha$ and $\beta$ is defined by 
$\angle(\alpha,\beta)=\lim_{s, t\to0}\tilde\angle \alpha(s)x\beta(t)$,
where $\tilde\angle \alpha(s)x\beta(t)$ denotes the angle of a 
comparison triangle $\tilde \Delta \alpha(s)x\beta(t)$ at the point 
$\tilde x$.
Two geodesics $\alpha$, $\beta$ from $x\in X$ is called {\it equivalent}  if
 $\angle(\alpha,\beta)=0$. %By \cite{BBI}, $\angle(\gamma,\gamma)=0$ for all geodesic in $M$.
We denoted by $\Sigma_x'(X)$ the set of equivalent classes of geodesics emanating from $x$.
The {\em space of directions} at $x$, denoted by $\Sigma_x=\Sigma_x(X)$, is the completion of
$\Sigma_x'(X)$ with the angle metric.
A direction of minimal geodesic from $p$ to $x$ is also denoted by $\uparrow_p^x$.
Let  $X$ be  $n$-dimensional. Then  $\Sigma_x$ is an $(n-1)$-dimensional compact
Alexandrov space with curvature $\ge 1$.

A points $x\in X$ is called {\em regular} if $\Sigma_x$ is isometric to $\mathbb{S}^{n-1}$.
Otherwise we call $x$ a singular point. We denote by $X^{\rm reg}$ (resp. $X^{\rm sing}$)
the set of all regular points (resp. singular points) of $X$. 

The {\em tangent cone} at $x\in X$, denoted by $T_x(X)$, 
is the Euclidean cone $K(\Sigma_x)$ over $\Sigma_x$.
It is known that  $T_x(M)=\lim_{r\to0}\left( \frac{1}{r}M,x \right)$.

For a closed subset $A$ of $X$ and $p\in A$, the space of directions  $\Sigma_p(A)$
of $A$ at $p$ is defined as the set of 
all $\xi\in \Sigma_x(X)$ which can be written as the limit of 
directions from $p$ to points $p_i$ in $A$ with $|p,p_i|\to 0$:
\[
         \xi =\lim_{i\to\infty} \uparrow_p^{p_i}.
\]
For $x, y\in X\setminus A$, consider a comparison triangle on  $M_{\kappa}^2$
having the side-length $(|A,x|, |x,y|, |y, A|)$ whenever it exists.
Then $\tilde\angle Axy$ denotes the angle of this comparison triangle at the vertex
corresponding to $x$.

For $x,y,z\in X$, we denote by $\angle xyz$ (resp. $\tilde\angle xyz$) the angle
between the geodesics $yx$ and $yz$ at $x$ (resp. the geodesics $\tilde y\tilde x$ and 
$\tilde y\tilde z$ at $\tilde x$ in the comparison triangle 
$\tilde\triangle xyz=\triangle \tilde x \tilde y \tilde z$).

Let $X$ be an $n$-dimensional Alexandrov space with curvature bounded below by $\kappa$.
For  $\delta>0$, a system of $n$ pairs of points, $\{a_i,b_i\}_{i=1}^n$ is called an
$(n,\delta)$-\emph{strainer} at $x\in X$ if it satisfies
\begin{align*}
   \tilde\angle_\kappa a_ixb_i > \pi - \delta, \quad &
                \tilde\angle_\kappa a_ixa_j > \pi/2 - \delta, \\
      \tilde\angle_\kappa b_ixb_j > \pi/2 - \delta,\quad &
                \tilde\angle_\kappa a_ixb_j > \pi/2 - \delta,
\end{align*}
for every $1\le i\neq j\le n$.
If  $x\in X$ has a $(n,\delta)$-strainer, then we say $x$ is $(n,\delta)$-strained.
In this case, we call $x$ {\em $\delta$-regular}.
%he positive number $\underset{1\le i\le k}\min\,\{ |a_ix|, |b_ix|\}$
%s called the {\em length} of the strainer.
We call $X$ {\em almost regular} if every point of $X$ is $\delta_n$-regular 
for some $\delta_n < 1/100n$.
It is known that a small neighborhood of any almost regular point is 
almost isometric to an open subset in $\mathbb R^n$.

Inductively on the dimension, the boundary $\partial X$ is  defined as the set of points $x\in X$ 
such that $\Sigma_x$ has non-empty boundary $\partial\Sigma_x$.
We denote by $D(X)$ the double of $X$, which is also
an Alexandrov space with curvature $\ge \kappa$
(see \cite{Pr:alexII}). By definition, $D(X)=X\amalg_{\partial X} X$, where tow copies of $X$ are glued
along their boundaries.

A boundary point $x\in\partial X$ is called {\em $\delta$-regular} if $x$ is $\delta$-regular
in $D(X)$. We say that $X$ is {\em almost regular with almost regular boundary}
if every point of $X$ is $\delta$-regular in $D(X)$ for $\delta<1/100n$.

In Section \ref{ssec:codim1}, we need the following result on the dimension of the interior singular point sets.
We set ${\rm int}X:= X\setminus \pa X$.

\begin{thm} [\cite{BGP}, cf. \cite{OS}]   \label{thm:dim-sing}
$$
            \dim_H(X^{\rm sing}\cap {\rm int}X) \le n-2,  \,\,\dim_H(\partial X)^{\rm sing} \le n-2,
$$  
where $(\partial X)^{\rm sing}=D(X)^{\rm sing}\cap \pa X$.
\end{thm}

\begin{thm} [\cite{Pr:alexII}, cf.\cite{Kap:stab}]      \label{thm:stability}
If a sequence $X_i$ of $n$-dimensional compact Alexandrov spaces 
with curvature $\ge \kappa$ Gromov-Hausdorff converges to an $n$-dimensional compact Alexandrov space $X$,
then $X_i$  is homeomorphic to $X$ for large enough $i$.
\end{thm}

 A subset $E$ of an Alexandrov space $X$ is called {\em extremal} (\cite{PtPt:extremal})
if every distance function $f={\rm dist}_q$, $q\in M\setminus E$ has the property
that if $f|_E$ has a local minimum at $p\in E$, then $df_p(\xi)\le 0$
for every $\xi\in\Sigma_p(X)$. Extremal subsets posses quite important properties.

\begin{thm} [\cite{PtPt:extremal}] \label{thm:exr-prop}
Let  $E$ be an extremal subset of $X$. 
\begin{enumerate}
 \item For every $p\in E$, $\Sigma_p(E)$ is an extremal subset of $\Sigma_p(X);$
 \item $E$ is totally quasigeodesic in the sense that any nearby two points of $E$ can be joined
         by a quasigeodesic (see \cite{PP QG}).
 \item $E$  has a topological stratification.
\end{enumerate}
\end{thm}

Theorem \ref{thm:exr-prop}$(1)$, $(2)$ implies the following

\begin{cor}  \label{cor:ext-dir}
For an extremal subset  $E$ of $X$ and $p\in E$, $\dim \Sigma_p(E)\le \dim E-1$.
\end{cor}

Suppose that  a compact group $G$ acts on $X$ as isometries. Then the quotient space
$X/G$ is an Alexandrov space (\cite{BGP}). Let $F$ denote the set of $G$-fixed points.

\begin{prop} [\cite{PtPt:extremal}] \label{prop:fixed-ext}
$\pi(F)$ is an extremal subset of $X/G$, where $\pi:X\to X/G$ is the projection.
\end{prop}

Boundaries of  Alexandrov spaces are typical examples of extremal subsets.

\begin{prop}[\cite{Ya:four} Prop 5.10] \label{prop:collar}
The boundary $\partial X$ of any finite dimensional Alexandrov space $X$ has a 
collar neighborhood. 
\end{prop}

An $n$-dimensional Alexandrov space is called {\em smoothable} if 
it is a Gromov-Hausdorff limit of $n$-dimensional closed Riemannian manifolds 
with a uniform lower sectional curvature bounds.

\begin{thm} [\cite{Kap}] \label{thm:iterated-sp}
Let $X$ be a smoothable Alexandrov space.
Then for any $p\in X$, every iterated space of directions 
\[
        \Sigma_{\xi_{k}}( \Sigma_{k-1}( \cdots(\Sigma_{\xi_1}(\Sigma_p(X))\cdots)),
\]
is homeomorphic to a sphere, where 
$$
     \xi_1\in\Sigma_p(X), \, \xi_2\in\Sigma_{\xi_1}(X), \ldots, \, 
     \xi_k\in\Sigma_{\xi_{k-1}}(   \cdots(\Sigma_{\xi_1}(\Sigma_p(X))\cdots)    ).
$$
\end{thm}

%%%%%%%%%%%%%%%%%%%%%%%%%%%%% Gluing    %%%%%%%%%%%%%%%%%%%%%%%%%%%%%%%%%

\subsection{Manifolds with boundary and gluing}\label{sec:gluing}

In this section, we consider a Riemannian manifold $M$ with boundary in 
${\mathcal M}(n,\kappa,\lambda,d)$.
First, we recall some fundamental properties of  $\pa M$, which were derived by Wong\cite{wong0}.
We also  recall Wong's cylindrical extension procedure based on Kosovskii's Gluing theorem \cite{Kos}.

Let $M$ be a Riemannian manifold with boundary, and
$\partial M^\alpha$ denote a boundary component of $\pa M$.
$(\pa M^\alpha)^{\rm int}$ means $\pa M^\alpha$ with intrinsic length metric. 

The following is an immediate consequence of the Gauss equation.

\begin{prop} \label{prop:bdy}
For every $M\in \ca M(n,\kappa,\lambda)$,  $\pa M$ has a uniform lower sectional curvature
bound: $K_{\pa M}\geq K$, where $K=K(\kappa,\lambda)$.
\end{prop}

\begin{prop} [\cite{wong0}]\label{prop:cpn-diam}
Let  $M\in\ca M(n,\kappa,\lambda,d)$.
\begin{enumerate}
 \item There exists a constant $D=D(n,\kappa,\lambda,d)$ such that  any boundary component 
     $\pa M^\alpha$ has intrinsic diameter bound
    \[
             \diam((\pa M^\alpha)^{\rm int})\leq D;
    \]
 \item $\partial M$ has at most $J$ components, where $J = J(n, \kappa, \lambda, d);$
% \item $K_{\pa M}\geq K$, where $K=K(k,\lambda)$.
\end{enumerate}
\end{prop}

It follows from Proposition \ref{prop:cpn-diam} that  every boundary component of
$M\in\ca M(n,\kappa,\lambda,d)$ is an Alexandrov space with curvature $\ge K$ and 
diameter $\le D$, where $K=K(\kappa,\lambda)$, $D=D(n,\kappa,\lambda,d)$

In general, a Riemannian manifold with boundary is not necessarily an Alexandrov space.
Wong (\cite{wong0}) carried out a gluing of   warped  cylinders and  $M$ along their boundaries  in such a way that 
the resulting manifold  becomes an Alexandrov space having totally geodesic boundary.

This is based on Kosovskii's gluing theorem:

\begin{thm}[\cite{Kos}] \label{thm:Kos}
Let $M_0$ and $M_1$ be Riemannian manifolds with boundaries $\Gamma_0$ and $\Gamma_1$ respectively
with sectional curvature $K_{M_i}\ge \kappa$ for $i=0,1$.
 Assume that there exists an isometry $\phi: \Gamma_0\to \Gamma_1$, and
let $M$ denote the space with length metric obtained by gluing  $M_0$ and $M_1$ 
along their boundaries via  $\phi$.
Let $L_i$, $i=0,1$, be the second fundamental form of   $\Gamma:=\Gamma_0\cong_{\phi}\Gamma_1\subset M$
with respect to the normal inward to $M_i$.
Then $M$ is an Alexandrov space with curvature  $\ge\kappa$ if and only if the sum  $L:=L_1+L_2$
is positive semidefinite.
\end{thm}

\begin{rem} \label{rem:Kos}
Actually, for every $\delta>0$, a smooth Riemannian metric $g_{\delta}$ on $M$ is  constructed
in \cite{Kos} in such a way that 
the sectional curvature of $g_{\delta}$ is greater than $\kappa(\delta)$ with 
$\lim_{\delta\to 0}\kappa(\delta)=\kappa$ and that $(M, g_{\delta})$ Gromov-Hausdorff converges 
to $M$ as $\delta\to 0$.
\end{rem}

Now let us recall the extension construction in \cite{wong0}.

Suppose $M$ is an $n$-dimensional complete  Riemannian manifold with 
\[
     K_M\geq\kappa,\,\, \lambda^-\leq II_{\pa M}\leq\lambda^+.
\]
Let  $\overline{\lambda}:=\min\{0,\lambda^-\}$. Then for arbitrarily  $t_0>0$ and  $0<\ve_0<1$
there exists a monotone non-increasing function  $\phi: [0,t_0]\to\mathbb{R}^+$ 
satisfying
\begin{align*}
     \phi''(t)+K\phi(t)\leq0, \,\,&\phi(0)=1, \,\,\phi(t_0)=\ve_0,\\
     -\infty<\phi'(0)\leq\overline{\lambda},\,\, &\phi'(t_0)=0,
\end{align*}
for some constant  $K=K(\lambda,\ve_0,t_0)$.
Now consider the warped product metric on  $\partial M\times [0, t_0]$ defined by 
\[
   g(x,t)=dt^2+\phi^2(t)g_{\partial M}(x)
\]
where $g_{\partial M}$ is the Riemannian metric of $\partial M$ induced from 
that of $M$. We denote by   $\partial M\times_{\phi} [0, t_0]$  the warped product.
It follows from the construction that
\begin{equation}
 \left\{ 
\begin{aligned}
   &\cdot\,  II_{\partial M\times\{ 0\}}\geq|\min\{0, \lambda^-\}|, \\
   &\cdot\,   II_{\partial M\times\{t_0\}}\equiv0, \\
   &\cdot\,  \text{the sectional curvature of  $\partial M\times_{\phi}[0, t_0]$ is greater than}\\
  &  \hspace{0.3cm}\text{ a constant  $c(\kappa,\lambda^\pm,\ve_0, t_0)$} \\
  & \cdot\,  \text {the second fundamental form of $\partial M\times \{t\}$ is given by }\\
   & \hspace{1cm}  II_{\partial M\times\{ t\}}(V,W)=\frac{\phi'(t)}{\phi(t)}g(V, W), \\
  &  \hspace{0.5cm}\text{for vector fields   $V, W$  on    $\partial M\times \{t\}$}.
\end{aligned} \right.   \label{eq:sumary}
\end{equation} 

Clearly, $\partial M\times\{0\}$ in $\partial M\times_\phi[0,t_0]$ is canonically isometric to
$\partial M$. Thus we can glue $M$ and $\partial M\times_\phi[0,t_0]$
along $\partial M$ and  $\partial M\times\{0\}$.
The resulting space
\[
                       \tilde M := M\amalg_{\pa M}(\pa M\times_\phi[0,t_0])
\]
carries the structure of differentiable manifold of class $C^{1,\alpha}$ with $C^0$-Riemannian metric
(\cite{Kos}).
Obviously  $M$ is diffeomorphic to $\tilde M$.

%Analogously, we can also define the warped product of a metric space and an interval.

\begin{prop} [\cite{wong0}] \label{prop:extendAS}
For  $M\in\ca M(n,\kappa,\lambda)$, we have 
\begin{enumerate}
 \item $\tilde M$ is an Alexandrov space with curvature $\ge \tilde{\kappa}$,  
         where $\tilde{\kappa}=\tilde{\kappa}(\kappa,\lambda);$
 \item the exterior metric $M^{\rm ext}$ is $L$-bi-Lipschitz homeomorphic to $M$ 
         for the uniform constant $L=1/\ve_0;$
\item  $\diam(\tilde M)\le \diam(M) + 2t_0$.
\end{enumerate}
\end{prop}

The notion of warped product also works for metric spaces.

Let $X$ and $Y$ be metric spaces, and $\phi:Y\to\mathbb R_+$ a positive continuous 
function. Then the warped product $X\times_{\phi}Y$ is defined as follows (see \cite{wong1}).
For a curve $\gamma=(\sigma,\nu):[a,b]\to X\times Y$, the length of $\gamma$ is 
defined as
\[
    L_{\phi}(\gamma)=\sup_{|\Delta|\to 0} \sum_{i=1}^k
               \sqrt{ \phi^2(\nu(s_i))|\sigma(t_{i-1}),\sigma(t_i)|^2 + |\nu(t_{i-1}),\nu(t_i)|^2},
\]
where $\Delta:a=t_0<t_1<\cdots<t_k=b$ and $s_i$ is any element of $[t_{i-1},t_i]$.
The  warped product $X\times_{\phi}Y$ is defined as the topological space 
$X\times Y$ equipped with the length metric induced from $L_{\phi}$.

\begin{prop}[\cite{wong1}, Proposition B.2.6]\label{prop:warped}
Let $X_i$ be a convergent sequence of length spaces.
If $Y$ is a compact length space, we have 
\[
                  {\lim}_{GH}(X_i \times_{\phi} Y) = ({\lim}_{GH}X_i)\times_{\phi} Y.
\]
whenever the limits exist.
\end{prop}

%%%%%%%%%%%%%%%%%%%%%%%  Descriptions of limit spaces   %%%%%%%%%%%%%%%%%%%%%%%%%%

\section{Descriptions of limit spaces and examples}\label{sec:str-limit}

Under the notations in section \ref{sec:gluing},
throughout this section unless otherwise stated, we assume
 $M_i\in\ca M(n,\kappa,\lambda,d)$ Gromov-Hausdorff converges to a compact length space
 $N$, where $\inrad(M_i)\to0$. Let $\tilde M_i$ converge to a compact  Alexandrov space $Y$,
and $M_i^{\rm ext}$ converge to a closed subset $X$ of $Y$ under the convergence $\tilde M_i\to Y$.

Here we fix  some notations used later on.
\begin{itemize}
\item $C_{M_i}$ denotes $\partial M_i \times_{\phi} [0, t_0];$
\item
$C_{M_i,t}$ denotes the subspace $\pa M_i\times_\phi\{t\}$ in $C_{M_i};$
\item
For $C_{M_i}\subset\tilde M_i$, $C_{M_i}^{\rm ext}$ denotes $(C_{M_i}, d_{\tilde M_i})$.
\end{itemize}

In this section, we first investigate the relation between the limit $C$ (resp.   $C_0$) of $C_{M_i}$ (resp of  $\partial M_i$)
and $Y$ (resp.  $X$), and  
discuss the intrinsic structure of $X$ and prove that 
$X^{\rm int}$ is isometric to $N$ (Proposition \ref{prop:intrinsic}).  Then we describe the metric structure of 
$Y$ (Proposition \ref{prop:YNC})

%%%%%%%%%%%%%%%%%%%%%%%%%%%%%%%%%%%%%%%%%%%%%%%%%%%%%%%%%%%%%%%%%%%%%%%%%%%%%%%%%%%%%%%%
\subsection{Descriptions of  $X$ and $Y$}

%In this subsection, we will prove that $Y$ is homeomorphic to a singular-$I$-bundle (Proposition \ref{thm:quotY}).

Under the notation presented in the begining of this section, 
in  view of  Proposition \ref{prop:cpn-diam} and \eqref{eq:sumary}, passing to a subsequence, we may assume that $C_{M_i}$ converges 
to some compact Alexandrov space $C$ with cuvrvature $\ge K=K(\kappa,\lambda)$. 
Here  $C_{M_i}$ is not necessarily connected, and therefore 
the convergence $C_{M_i}\to C$ should be understood componentwisely. 
%Note that  by Theorem \ref{partial M has finite cpn and diam}(iv),
%the number of components of  $C_i$ and hence $C$ is uniformly bounded.
It follows from Proposition \ref{prop:warped} that 
\[
                    C= C_0\times_{\phi} [0, t_0], \,\,   C_0=\lim_{i\to\infty} (\pa M_i)^{\rm int},
\]
where  $(\pa M_i)^{\rm int}$ denotes  $\pa M_i$ endowed with length metric induced by its original metric.
For simplicity we denote 
\begin{gather*}
           C_0:=C_0\times\{0\},\,\,  C_t:=C_0\times\{t\}\subset C,
               % C_{M_i,t} =\pa M_i\times\{t\} \subset C_{M_i}.
\end{gather*}
Since the identity map $\iota_i: C_{M_i}\to C_{M_i}^{\rm ext}$ is 1-Lipschitz,
we can define a surjective 1-Lipschitz map $\eta: C\to Y$ in the limits.
More precisely, define $\eta:C\to Y$ by
\[
                    \eta=\lim_{i\to\infty} g_i\circ\iota_i\circ  f_i,
\]
where  $f_i:C\to C_{M_i}$,  $g_i:\tilde M_i\to Y$ are component-wise  $\ve_i$-approximations
with $\lim \ve_i=0$.
%We will prove that $\eta$ is a quotient map.

From now on,  we consider 
\[
        \eta_0: = \eta|_{C_0\times \{ 0 \} } : C_0\to X,
\]
which is also a surjective $1$-Lipschitz map with respect to the exterior metrics of 
$C_0$ and $X$, and hence with respect to the interior metrics, too.
%
%$ is denoted by $\eta_0$. 

The following two lemmas are obvious.

\begin{lem}\label{lem:loc-isom}
The map $\eta:C\setminus C_0\to Y\setminus X$ is a bijective local isometry.
\end{lem}

\begin{lem}\label{lem:dist-bdry}
For  $(p,t)\in C\setminus C_0$, we have  $|\eta(p,t), X|=t$.
\end{lem}

We now study the multiplicities of the gluing map $\eta_0$.

\begin{lem} \label{prop:preimage}
For every $x\in X$, we have the following:
\begin{enumerate}
 \item $\#\eta_0^{-1}(x)\leq 2;$ 
 \item Suppose $\#\eta_0^{-1}(x)=2$ for some $x\in X$,  and take $p_k\in C_0$, $k=1,2$, with 
$\eta_0(p_k)=x$.  Then $\Sigma_x(Y)$ is isometric to a spherical suspension with the 
two vertices $\{ \xi_1, \xi_2\}$, where 
    \[
                 \xi_k :=\uparrow_x^{\eta(p_k, t_0)}.
   \]
\end{enumerate}
\end{lem}

\begin{proof}
Suppose that $\#\eta_0^{-1}(x)\ge 2$ and take $p_1, p_2\in\eta_0^{-1}(x)$, and 
let $y_i:=\eta(p_i,t)$, $i=1,2$,  for some $t>0$. 
We show that  $|y_1,y_2|=2t$ or equivalently, 
\begin{align}
      \tilde\angle y_1 p y_2= \pi,     \label{eq:spread} 
\end{align}
if  $t<\phi(t_0)|p_1p_2|_{C_0^{\rm int}}/2$, which yields $\#\eta_0^{-1}(x)=2$
and the conclusions $(1)$ and $(2)$. 

Let $\gamma:[0,\ell]\to Y$ be a  minimal geodesic in $Y$ joining $y_1$ and $y_2$.
If $\gamma$ meets $X$, we certainly have  $|y_i,y_2|=2t$.
%If  $p_1$ and $p_2$ are contained in distinct connected components of  $C_0$, 
%$\gamma$ must meet $X$, and therefor $|y_i,y_2|=2t$.
Suppose that $\gamma$ does not meet $X$.
%$p_1$ and $p_2$ are contained in the same connected component of $C_0$.
%
Then $\tilde\gamma=\eta^{-1}(\gamma)$ is well-defined and is a minimal geodesic 
joining $(p_1,t)$ and $(p_2,t)$.
Write $\tilde\gamma$ as  $\tilde\gamma(s)=(\sigma(s),\nu(s))\in C_0\times_{\phi} [0,t_0]$.
% be
%a shortest path in $C$ such that
%$\tilde\gamma(0)=(p_1,t)$ and $\tilde\gamma(1)=(p_2, t)$.
Then we have 
\begin{align*}
L(\gamma)=L(\tilde\gamma)
=& \int_0^{\ell} \sqrt{\phi^2(\nu(s))
             |\dot\sigma(s)|^2+|\dot{\nu}(s)|^2}\,ds.\\
\ge& \int_0^{\ell} \phi(t_0) |\dot\sigma(s)|\,dt \ge \phi(t_0) |p_1,p_2|_{C_0^{\rm int}}.          
\end{align*}
Thus we have  $|y_1,y_2| = L(\gamma)\ge \phi(t_0)|p_1,p_2|_{C_0^{\rm int}}$.
%If $\gamma$ does not meet $X$, the geodesic $\tilde\gamma=\eta^{-1}\circ\gamma$ joining
On the other hand, the triangle inequality shows that 
$|y_1,y_2|\le 2t <\phi(t_0)|p_1,p_2|_{C_0^{\rm int}}$.
This is a contradiction, and therefore  $\gamma$ meets $X$ and  $|y_1, y_2|=2t$
\end{proof}

Next we construct a good approximation map $\tilde M_i\to Y$,
which helps us to grasp a whole picture on the several convergences. 

Let $\psi_i:\partial M_i=C_{M_i,0}\to C_0$ be an $\epsilon_i$-approximation
with $\lim_{i\to\infty}\epsilon_i=0$.

\begin{lem} [\cite{wong1}]
The map $\Psi_i:C_{M_i}\to C$ defined by 
\[
            \Psi_i(p,t)=(\psi_i(p),t)
\] 
is an $\epsilon_i'$-approximation with  $\lim_{i\to\infty}\epsilon_i'=0$.
Actually, for any approximation map  $\Psi_i':C_{M_i}\to C$ there is a   $\psi_i:\partial M_i=C_{M_i, 0}\to C_0$
such that $|\Psi_i(p,t), \Psi_i'(p,t)|<\epsilon_i'$ for  $\Psi_i=(\psi_i, {\rm id})$.
\end{lem}

\begin{proof}
This follows from Proposition \ref{prop:warped}. 
\end{proof}

Recall that $\eta:C\setminus C_0\to Y\setminus X$ is a locally isometric bijection.
In particular for every $y=(p,t_0)\in C_{t_0}\subset Y$, there is a unique 
minimal geodesic $\gamma_y:[0,t_0]\to Y$ between $X$ and $y$ such that 
$\gamma_y(0)\in X$, $\gamma(t_0)=y$. Actually $\gamma_y$ is defined as 
$\gamma_y(t)=\eta(p,t)$.
Define $g_i^*:C_{M_i}^{\rm ext}\to Y$ by
\begin{align}
         g_i^*(p,t)=\eta\circ\Psi_i\circ\iota_i^{-1}(p,t)=\eta(\psi_(p), t). \label{eq:g*}
\end{align}

\begin{prop} \label{prop:g*}
The map  $g_i^*:C_{M_i}^{\rm ext}\to Y$ defined above provides an $\epsilon_i'$-approximation.
\end{prop}

Let $g_i:C_{M_i}^{\rm ext}\to Y$ be any $\epsilon_i$-approximation such that 
$g_i=g_i^*$ on $C_{M_i,t_0}$, namely $g_i(p, t_0)=g_i^*(p,t_0)$.

For the proof of Proposition \ref{prop:g*}, it suffices to show the following.

\begin{lem} \label{lem:gg*}
 $|g_i(p,t), g_i^*(p,t)| < \epsilon_i'$ for all $(p,t)\in C_{M_i}^{\rm ext}$.
\end{lem}

\begin{proof}
We have to show that 
\[
                  \lim_{i\to\infty} \sup_{(p,t)\in C_{M_i}}\,|g_i(p,t), g_i^*(p,t)| = 0.
\]
Suppose the contrary. Then there are subsequence $\{ j\}\subset \{ i\}$ 
and $(p_j, t_j)\in C_{M_j}$ such that 
\begin{align}
        |g_j(p_j,t_j), g_j^*(p_j,t_j)| \ge c>0, \label{eq:positive}
\end{align}
for some constant $c$ independent of $j$.
Passing to a subsequence, we may assume that $(\psi_j(p_j), t_j)$ converges to $(p_{\infty},t_{\infty})\in C$.
Let $\gamma_j(t)=(p_j, t)$, $0\le t\le t_0$, which is a minimal geodesic in $C_{M_j}^{\rm ext}$ 
between $\partial M_j$ and $C_{M_i, t_0}$.
Now $g_j^{*} \circ\gamma_j(t)=\eta(\psi_j(p_j), t)$ converges to a minimal geodesic
$\gamma_{\infty}(t)=\eta(p_{\infty}, t)$ realizing the distance between 
$X$ and $(p_{\infty}, t_0)\in C_{t_0}\subset Y$.  
Since $g_j$ is $\epsilon_j$-approximation, any limit  of $g_j\circ\gamma_j$, say $\hat\gamma$, must 
also be minimal geodesic between $X$ and  $(p_{\infty}, t_0)$.
From the uniqueness of such geodesic, we have $\gamma_{\infty}(t)=\hat\gamma_{\infty}(t)$,
which contradicts \eqref{eq:positive}. 
 \end{proof}

%\subsection{Intrinsic structure of X}
\par\medskip
Next, we determine the intrinsic structure of $X$, and prove 
Proposition \ref{prop:intrinsic} below, 
which will be crucial in our start for the description of $Y$ in terms of $N$ (see Proposition \ref{prop:YNC})

Recall that $X\subset Y$ is the limit of   $M_i^{\rm ext}$ 
under the convergence $\tilde M_i \to Y$.
By Proposition \ref{prop:extendAS}, 
% Since $d_{M_i}\geq d_{M_i^{\rm ext}}$ and $M_i$ is $L$-bi-Lipschitz homeomorphic to $M_i^{\rm ext}$ for a uniform 
%constant $L$, we have 
%
the identity $\iota_i:M_i \to M_i^{\rm ext}$ is a $L$-bi-Lipschitz 
homeomorphism.  Therefore we have that

\begin{lem} \label{lem:NX} For a subsequnce, 
 $\iota_i:M_i \to M_i^{\rm ext}$ converges to an $L$-bi-Lipschitz 
homeomorphism $\iota_{\infty}:N  \to X$. 
\end{lem}

\begin{prop}\label{prop:intrinsic}
$X^{\rm int}$ is isometric to $N$.
\end{prop}

\begin{proof} Passing to a subsequence if necessary, we may assume 
that the $L$-Lipschitz map  $\iota_i:M_i^{\rm ext} \to M_i$, where $L=1/\e_0$,  converges to a surjective map
$h:X\to N$ satisfying 
\[
                   |x,y|_X \le  |h(x),h(y)|_N \le L|x,y|_X,
\]
for every $x, y\in X$. Let $\sigma:[0, d]\to N$ be a minimal geodesic joining $h(x)$ and $h(y)$.
Then we have 
\[
    |h(x), h(y)|_N = L(\sigma)\ge L(h^{-1}(\sigma)) \ge |x,y|_{X^{\rm int}}.
\]
Next we show the reverse inequality. %at $|h(x),h(y)|_N \le |x,y|$.
Let $\gamma:[0,\ell]\to X$ be a minimal geodesic in $X^{\rm int}$ joining $x$ to $y$.
For any $\ve>0$, take a subdivision $\Delta$ of  $\gamma$: 
$x=x_0<x_1<\cdots <x_{\alpha}<\cdots x_k=y$ such that 
denoting by $\gamma_{\Delta}$ the broken geodesic consisting of  minimal geodesic 
joining $x_{\alpha-1}$ and $x_{\alpha}$ in $Y$ for $1\le \alpha\le k$, we have 
\begin{enumerate}
 \item $|L(\gamma_{\Delta}) - |x,y|_{X^{\rm int}}| <\ve;$
 \item $\max_t |\gamma_{\Delta}(t), X| < \ve$.
\end{enumerate}
Take $p_{\alpha}^i\in M_i$ converging to $x_{\alpha}$  under the convergence 
$\tilde M_i \to Y$, and denote by $\gamma_{\Delta}^i$ a broken geodesic consisting of  minimal geodesic 
joining $p_{\alpha-1}^i$ and $p_{\alpha}^i$ in $\tilde M_i$ for $1\le \alpha\le k$.
Note that for large enough $i$
\begin{enumerate}
 \item $|L(\gamma_{\Delta}) - L(\gamma_{\Delta}^i)| <\ve;$
 \item $\max_t |\gamma_{\Delta}^i(t), M_i| < \ve$.
\end{enumerate}
Let $\sigma_i:=\pi_i\circ\gamma_{\Delta}^i$, where $\pi_i:\tilde M_i \to M_i$ is the canonical 
projection defined by $\pi_i(p,t)=p$.
From the warped product metric construction, we have $L(\gamma_{\Delta}^i)\ge \phi(\ve)L(\sigma_i)$
for large $i$. It follows that 
\begin{align*}
    |x,y|_{X^{\rm int}} &\ge L(\gamma_{\Delta}) - \ve >  L(\gamma_{\Delta}^i) - 2\ve \\
                            &\ge  \phi(\ve)L(\sigma_i) -2\ve \\
                            &\ge  \phi(\ve)|p_i,q_i|_{M_i} -2\ve, 
\end{align*}
where $p_i\to x$ and $q_i\to y$ under $\tilde M_i \to Y$.
Letting $|\Delta|\to 0$ and $i\to\infty$, we conclude that $ |x,y|_{X^{\rm int}}\ge |h(x), h(y)|_N$.
This completes the proof.
\end{proof}

Let $X^{\rm int}\cup_{\eta_0}C_0\times_{\phi}[0,t_0]$ denote the length space obtained by the result of
gluing of the two length spaces $X^{\rm int}$ and $C_0\times_{\phi}[0,t_0]$ by the map 
$\eta_0:C_0\times 0\to X^{\rm int}$.
%It is straightforward to see that the canonical map $Y\to X^{\rm int}\cup_{\eta_0}C_0\times_{\phi}[0,t_0]$ 
%is an isometry. Combined with Proposition \ref{prop:intrinsic}, we have 

\begin{prop}\label{prop:YNC}
$Y$ is isometric to the length space 
\[  
                     X^{\rm int}  \cup_{\eta_0}C_0\times_{\phi}[0,t_0].
\] 
\end{prop}

\begin{proof}
Let $Z:= X^{\rm int}  \cup_{\eta_0}C_0\times_{\phi}[0,t_0]$, and $\Phi:Y\to Z$ the canonical map.
Note that $\Phi$ is bijective.
For every $y_0, y_1\in Y$, let $\gamma:[0,\ell]\to Y$ be a minimal geodesic joining $y_0$ and $y_1$. 
Decompose $\gamma$ into the two parts:
\[
     \gamma = \gamma_{Y\setminus X} \cup \gamma_{X},
\]
where $\gamma_{Y\setminus X}=\gamma\cap (Y\setminus X)$ and $\gamma_X=\gamma\cap X$.
Let $\gamma_{Y\setminus X}=\cup_{\alpha} \gamma_{\alpha}$ be the at most countable union 
consisting of open arc components of $\gamma_{Y\setminus X}$.
For any $\e>0$, take $\gamma_{\alpha}$ of length $\le \e$ such that the endpoints  
$z_{\alpha}$ and $w_{\alpha}$  of  $\gamma_{\alpha}$ are contained 
in $X$ if such a $\gamma_{\alpha}$ exists.
Take $p_i, q_i\in M_i$ such that $p_i\to z_{\alpha}$,  $q_i\to w_{\alpha}$ under the convergence $\tilde M_i\to Y$.
For a minimal geodesic $\gamma_i$ joining $p_i$ and $q_i$ in $\tilde M_i$, let $\sigma_i:=\pi_i(\gamma_i)$,
where $\pi_i:\tilde M_i \to M_i$ is the projection. Note that $\max |\gamma_i(t), M_i|<\e$.
Using the warped metric structure, we have $L(\gamma_i) \ge \phi(\e)L(\sigma_i)$,
which implies 
\[
      |z_{\alpha}, w_{\alpha}|_Y \ge |p_i, q_i|_{\tilde M_i} - o_i \ge \phi(\e)|p_i, q_i|_{M_i} - o_i,
\]
where $\lim o_i= 0$. 
Letting $i\to \infty$, we have $ |z_{\alpha}, w_{\alpha}|_Y \ge \phi(\e) |z_{\alpha}, w_{\alpha}|_{X^{\rm int}}$.
Now we replace $\gamma_{\alpha}$ by a minimal geodesic joining $z_{\alpha}$ and $w_{\alpha}$ in $X^{\rm int}$.
Repeating this procedure at most countably many times if necessarily, we construct a Lipschitz curve $\hat\gamma$ joining $y_0$ to $y_1$
such that in the decomposition 
\[
  \hat \gamma = \hat \gamma_{Y\setminus X} \cup \hat \gamma_{X},
\]
$ \hat \gamma_{Y\setminus X}$ (resp.  $\hat \gamma_{X}$) consists of 
finitely many $Y$-minimal geodesics each of whose length $\ge \e$ (resp. finitely many $X$-minimal geodesic) and that 
\[
  |y_0, y_1|_Y  =L(\gamma) \ge \phi(\e) L(\hat\gamma) \ge     \phi(\e)|\Phi(y_0), \Phi(y_1)|_Z.
\]
Letting $\e\to 0$, we conclude that $|y_0, y_1|_Y\ge |\Phi(y_0), \Phi(y_1)|_Z$.

Next taking a $Z$-minimal geodesic joining $\Phi(y_0)$ and $\Phi(y_1)$ and replacing it by a Lipschitz 
curve in a similar way, we obtain the reverse inequality $|y_0, y_1|_Y\le |\Phi(y_0), \Phi(y_1)|_Z$.
This completes the proof.
\end{proof}

\begin{rem} \label{rem:noncpt-limit}
Both  Propositions \ref{prop:intrinsic} and \ref{prop:YNC} hold true for 
pointed Gromov-Hausdorff limits of  inradius collapsed manifolds (see Section \ref{sec:unbounded}). 
Moreover, in the above proofs, we do not need the assumption of inradius collapse. Therefore 
 Propositions \ref{prop:intrinsic} and \ref{prop:YNC} also hold for 
Gromov-Hausdorff limits of  non-inradius collapsed manifolds.% (see Lemma \ref{lem:YX0}).
%
%The above proof of Proposition \ref{prop:intrinsic} also works for pointed Gromov-Hausdorff limit
%spaces $(N,q) =\lim (M_i, p_i)$ and $(Y, X, x)=\lim (\tilde M_i, M_i, p_i)$ of pointed inradius collapsed manifolds.
%Proposition \ref{prop:YNC} also holds for the unbounded diameter case (see Section \ref{sec:unbounded}).
\end{rem}

\begin{cor}\label{cor:dim-collaps}
If $M_i \in \ca M(n,\kappa,\lambda,d)$ inradius collpases to $N$, then it actually collapses to $N$. 
Namely we have 
\begin{enumerate}
 \item $\dim M_i>\dim N;$
 \item $\lim \vol(M_i) = 0$.
\end{enumerate} 
\end{cor}
\begin{proof} (1)\,
From  Lemma \ref{lem:NX} and Proposition \ref{prop:YNC},  
we have 
\begin{align*}
      \dim M_i &=\dim\tilde M_i \ge \dim Y  \\
                   & \ge \dim X+1=\dim N+1.
\end{align*}
(2)\, We proceed by contradiction. Suppose $\vol(M_i)>v_0>0$ for some 
constant $v_0$ independent of $i$. By Proposition \ref{prop:extendAS}, there is a 
uniform bound $V$ with $\vol(\pa M_i)\le V$.
Choose any $\e_0\in (0,1)$ and $t_0\in (0, v_0/2V)$, and perform 
the extension procedure with warping  function as in Subsection \ref{sec:gluing}.
Then $C_{M_i}$ has volume 
\[
     \vol(C_{M_i}) < V t_0 < \frac{v_0}{2}.
\]
Passing to a subsequence, we may assume that $\tilde M_i$ converges to $Y$.
Since $\vol (\tilde M_i) \ge v_0$, we have $\dim Y = n$.
It follows from the volume convergence, 
\[
    \vol(Y) = \lim \vol(\tilde M_i) \ge v_0.  % \quad \vol(C_0)=\lim \vol(\pa M_i) \ge v_0.
\]
However 
\begin{align*}
    \vol(Y) &= \vol (Y\setminus X) + \vol (X) \\
              &= \vol(C_0\times_{\phi} [0, t _0])< V_0t_0 \le v_0/2,
\end{align*}
which is a contradiction.
\end{proof}

\begin{rem}
Wong proved $\dim M_i>\dim N	$ in \cite[Lemma 1]{wong2} under the condition that $N$ is an absolute Poincar\'{e} duality space.
In \cite{YZ:general}, we shall show that if $N$ is a
 closed topological manifold 
or a closed Alexandrov space, then $M_i$ inradius collapses. 
Hence Corollary \ref{cor:dim-collaps} give another version of 
Wong's result. It should also be noted that the conclusion of Corollary \ref{cor:dim-collaps} 
holds for limit spaces of inradius collapsed manifolds with respect to the pointed Gromov-Hausdorff 
topology (see Corollary \ref{cor:dim-collaps-pointed}).
\end{rem}

\begin{defn}
In view of Lemma \ref{prop:preimage} and Proposition \ref{prop:YNC}, 
we make an  identification $N=X^{\rm int}$ and set 
for $k=1,2$,
\begin{gather*}
  N_k = Xk:=\{x\in X|\#\eta_0^{-1}(x)=k\}, \\
  C_0^k:= \{ p\in C_0\,|\, \eta_0(p)\in X_k\}. 
\end{gather*}
\end{defn}

%%%%%%%%%%%%%%%%%%%%%%%%%%%%%%%%%%%%%%%%%%%%%%%%%%%%%%%%%%%%%%%%%%%%%%%%%%%%
\subsection{Examples} \label{ssec:Example}

We exhibit some examples of collapse of manifolds with boundary. All the examples except 
Example \ref{ex:G-action} are inradius collapses.

\begin{ex} \label{ex:counter2}
Let $\mathbb S^{n-1}(r):=\{ \, x\in\mathbb R^{n}\,|\, \sum_{i=1}^n (x_i)^2=r^2\,\}$.
For $\epsilon>0$, define $M_{\epsilon}$ as the closed domain in $\mathbb R^n$ bounded by
$\mathbb S^{n-1}(r+\epsilon)$ and $\mathbb S^{n-1}(r)$.
Then $K_{M_{\epsilon}}\equiv 0$ and $|\Pi_{\partial M_{\epsilon}}|\le 1/r$,
and $M_\epsilon$ inradius collapses to $N:=\mathbb S^{n-1}(r)$, where the limit space 
is an Alexandrov space with curvature $\ge r^{-2}$. 
Note that  $N_2=N$, and that 
the limit $Y$ of $\tilde M_{\epsilon}$ is isometric to the form
\[
      Y = ( \mathbb S^{n-1}(r)\amalg  \mathbb S^{n-1}(r))\times_{\phi}[0,t_0]/(f(x), 0)\sim(x,0),
\]
where $f: \mathbb S^{n-1}(r)\amalg\mathbb S^{n-1}(r) \to \mathbb S^{n-1}(r)\amalg\mathbb S^{n-1}(r)$
is the canonical involution. Equivalently $Y$ is isometric to the warped product
\[
                   \mathbb S^{n-1}(r)\times_{\tilde\phi} [-t_0, t_0],
\]
where $\tilde\phi(t)=\phi(|t|)$.

This example shows that the lower  Alexandrov curvature bound of the limit in Theorem \ref{thm:limit-alex} 
really depends on the bound $\lambda\ge |\Pi_{\partial M}|$.
\end{ex}

\begin{ex} [\cite{wong2}] \label{ex:counter}
Let $N\subset\mathbb R^2\times 0\subset \mathbb R^3$ be a non-convex domain with smooth boundary,
and let  $M_{\epsilon}'$ denote the closure of $\epsilon$-neighborhood of $N$ in $\mathbb R^{3}$.
After a slight smoothing of $M_{\epsilon}'$, we obtain a flat Riemannian manifold $M_{\epsilon}$
with  boundary such that  $\Pi_{\partial M_{\epsilon}}\ge -\lambda$ for some $\lambda>0$ independent of
$\epsilon$.
Note that  $M_{\epsilon}$ inradius collapses to $N$, where $N$ has no lower Alexandrov curvature bound.

This example shows that Theorem \ref{thm:limit-alex} does not hold
if one drops the upper bound $\lambda\ge \Pi_{\partial M}$.
\end{ex}

\begin{ex} \label{ex:counter3}
Let $N \subset \mathbb R^2$ 
be the union of the unit circle $\{ (x,y)\,|\, x^2+y^2=1\}$ and the segment $\{ (x,y)\,|\, x=0, -1\le y\le 1\}$.
Let $M_{\e}$ be the intersection of  the closed $\e$-neighborhood of $N$ in $\mathbb R^2$  and 
the unit disk $\{ (x,y)\,|\, x^2+y^2\le 1\}$.
 After slight smoothing of $M_{\e}$, it is a compact surface with 
$K_{M_{\e}}\equiv 0$, ${\rm II}_{\pa M_{\e}} \le \lambda^2$ for some $\lambda$.
However $\inf {\rm II}_{\pa M_{\e}} \to -\infty$ as $\e\to 0$, and 
$M_{\e}$ inradius collapses to $N$, which is not an Alexandrov space with curvature bounded below.

This example shows that Theorem \ref{thm:limit-alex} does not hold
if one drops the lower bound $-\lambda^2 \le {\rm II}_{\partial M}$.

\end{ex}

\begin{ex} \label{ex:twist}
Let $\pi:P\to N$ be a Riemannian  double covering between closed Riemannian manifolds with the
deck transformation  $\varphi:P\to P$.
Define $\varPhi:P\times [-\epsilon,\epsilon]\to P\times [-\epsilon,\epsilon] $ by
\[
             \varPhi(x,t) = (\varphi(x), -t),
\]
and consider $M_{\epsilon}:=P\times [-\epsilon, \epsilon]/\varPhi$, which is a twisted $I$-bundle over $N$.
Note that $M_{\epsilon}\in\mathcal M(n, \kappa, 0, d)$ for some $\kappa$ and $d$,
and that $M_{\epsilon}$ inradius collapses to $N$ as $\epsilon\to 0$. 
In this case, we have  $N_2=N$. Note that the limit $Y$ of $\tilde M_{\epsilon}$ is isometric to the form
\[
      Y = P\times_{\phi}[0,t_0]/(\varphi(x), 0)\sim(x,0),
\]
or equivalently $Y$ is doubly covered by the warped product
\[
                   P\times_{\tilde\phi} [-t_0, t_0].
\]
%where $\tilde\phi(t)=\phi(|t|)$.
\end{ex}

\begin{ex} \label{ex:D}
Let $N$ be a convex domain in $\mathbb R^{n-1}\times 0\subset \mathbb R^{n+1}$ with smooth boundary.
Let $M_{\epsilon}'$ denote the intersection of the boundary of $\epsilon$-neighborhood of $N$ in $\mathbb R^{n+1}$ with
the upper half space $H_{+}=\{ (x_1,\ldots, x_{n+1})\,|\, x_{n+1}\ge 0\,\}$.
After a slight smoothing of $M_{\epsilon}'$, we obtain a nonnegatively curved Riemannian manifold $M_{\epsilon}$
with totally geodesic boundary. 
Note that  $M_{\epsilon}$ inradius collapses to $N$ as $\epsilon\to 0$. 
Note also that $(\partial M_{\epsilon})^{\rm int}$, a smooth approximation of the boundary of 
$\epsilon$-neighborhood of $N$ in $\mathbb R^{n}$, converges to the double $D(N)$ of $N$.
It follows that $N_1=\partial N$ and $N_2=N\setminus \partial N$, and that 
the limit $Y$ of $\tilde M_{\epsilon}$ is isometric to the form
\[
      Y = D(N)\times_{\phi}[0,t_0]/(r(x),0)\sim (x,0),
\]
where $r:D(N)\to D(N)$ denotes the canonical reflection of $D(N)$.
\end{ex}

Next let us consider more general examples.
The following ones come from Example 1.2 in \cite{Ya:pinching},  where general examples of collapse of
closed manifolds were  given.

\begin{ex}  \label{ex:G-bundle}
Let $\hat\pi:M \to N$ be a fiber bundle over a closed manifold $N$ with fiber $F$ 
having non-empty boundary and 
with the structure group $G$ such that
\begin{enumerate}
 \item $G$ is a compact Lie group; 
 \item $F$ has a $G$-invariant metric $g_F$ of nonnegative curvature  which smoothly extends to the double $D(F);$
\end{enumerate}
Fix a bi-invariant metric $b$ on $G$ and a metric $h$ on $N$.
Let $\pi:P\to N$ be the principal $G$-bundle  associated with $\hat\pi:M \to N$.
Define $G$-invariant metric $g_{\epsilon}$ on $P$ by
\[
        g_{\epsilon}(u, v)= h(d\pi(u), d\pi(v))+ \epsilon^2 b(\omega(u),\omega(v)),
\]
where $\omega$ is a $G$-connection on $P$. Define a metric $\tilde g_{\epsilon}$ on $P\times D(F)$ as
\[
                    \tilde g_{\epsilon}=g_{\epsilon} + \epsilon^2 g_F.
\]
For the $G$-action on $P\times  D(F)$
defined by $(p,f)\cdot g=(pg,g^{-1}f)$, $\tilde g_{\epsilon}$ is $G$-invariant and  invariant under the action of 
reflection of $D(F)$. Therefore 
it induces a metric $g_{D(M), \epsilon}$ on $D(M)=P\times D(F)/G$. Since  $g_{D(M), \epsilon}$ is 
invariant under the action of reflection of $D(M)$, it induces a metric $g_{M,\epsilon}$ on $M$ with totally geodesic 
boundary such that  $(M, g_{M, \epsilon})$ inradius collapse to $(N,h)$ under a lower sectional 
curvature bound.
%  In the case when  $\hat\pi:M \to N$ is trivial, we observe $N=N_1$.
\end{ex}

\begin{ex} \label{ex:G-action}
Let $M$ be a compact manifold with boundary, and suppose that a compact Lie group of 
positive dimension effectively act on $M$ which extends to the action on $D(M)$.
Suppose that $D(M)$ has $G$-invariant and reflection-invariant smooth metric $g$. 
As in Example 1.2 of  \cite{Ya:pinching},  one can construct a metric $g_{D(M), \epsilon}$ on  $(D(M)$ 
which collapses to  $(D(M), g_{D(M), \epsilon})/G$ under a lower curvature bound.
It follows that the metric   $(M, g_{M, \epsilon})$ induced by $g_{D(M),\epsilon}$ also 
collapses to  $(M, g_{M, \epsilon})/G$ under a lower curvature bound. Note that 
 $(M, g_{M, \epsilon})$ has totally geodesic boundary.
\end{ex}

%%%%%%%%%%%%%%%%%%%%%%%%%%%%%%%%%%%%%%%%%%%%%%%%%%%%%%%%%%%%%%%%%%%%%%%%%%%%%%%%
\section{Metric structure of limit spaces} \label{sec:metric}

Let $X\subset Y$ and $N$ be as in Section \ref{sec:str-limit}.
The main purpose of this section is to show that $Y$ and $N$ are  actually isometric to
$C/\eta_0$ and $C_0/\eta_0$ respectively.
To study how this gluing is made, 
we first analyze the tangent cones of $C$, $C_0$, $Y$ and $X$ at gluing points, 
and their relations via the the differential $d\eta_0$ of the gluing map $\eta_0$.
It turns out that the identification map 
$\eta_0$ preserves length of curves.
Finally, we see that   $N$ is isometric to a quotient of 
 $C_0^{\rm int}$ by an isometric $\mathbb{Z}_2$ action.
 (see Proposition \ref{prop:quotient}), which implies Theorems \ref{thm:limit-alex} and \ref{thm:singIbund}. 
%%%%%%%%%%%%%%%%%%%%%%%%%%%%%%%%%%%%%%%%%%%%%%%%%%%%%%%%%%%%%%%%%%%%%%%%%%%%%%%%%%%%%%%%%%%%%
%%%%%%%%%%%%%%%%%%%%%%%%%%%%%%%%%%%%%%%%%%%%%%%%%%%%%%%%%%%%%%%%%%%%%%%%%%%%%%%%%%%%%%%%%%%%%%

\subsection{Preliminary argument} \label {ssec:metric-prelim}

In this subsection, we study geodesic behavior in $C$ and the property 
of a rescaling limit of the map $\eta:C\to Y$. These will be useful in the next subsection
to investigate geodesic behavior in $Y$.

Let $\tilde \pi:C\to C_0$ and $\pi:Y\to X$ be the projections. 
To be precise, $\pi(y):=\eta_0\circ\tilde\pi(\eta^{-1}(y))$,
which are surjective Lipschitz maps. 
For every $p\in C_0$, let $\tilde\gamma_{+}(t) = (p,t)$ and 
$\gamma_+(t)=\eta(\tilde\gamma(t))$, $t\in [0,t_0]$. We call $\tilde\gamma_{+}$
 (resp. $\gamma_+$) a {\it perpendicular to} $C_0$
(resp. to $X$) at $p$ (resp. at $\eta_0(p)$).   
The map $\tilde \pi$ and $\pi$ are the projections along perpendiculars.
Note that $\eta:C\setminus C_0 \to Y\setminus X$ is a locally isometric bijective map.
Therefore $C\setminus C_0$ and $Y\setminus X$ are isometric to each other
with respect to the their {\it length} metrics.

For simplicity, we use the notation
\[
                 C_t :=\{ x\in C\,|\, d(C_0, x)=t\}, \,\, C_t^Y :=\{ y\in Y\,|\, d(X, y)=t\}
\]
for every $t\in (0,t_0]$. We also denote by
\[
    \tilde\pi_t:C\to C_t,\,\,  \pi_t:Y\setminus X\to C_t
\]
the canonical projections along perpendiculars.  Recall that  
\begin{align*}
X_1 =\{x\in X|\#\eta_0^{-1}(x)=1\},& \,\,\, X_2 =\{x\in X|\#\eta_0^{-1}(x) = 2\}, \\
   C_0^k = \{ p\in C_0\,&|\, \eta_0(p)\in X_ki\}, \, k=1,2.
\end{align*}

%%%%%%%%%%%%%%%%%% geodesic behavior %%%%%%%%%%%%%%%%%%%%%%%%%%%%%%%%%%%%%%%%%%%%%%%%%%
\par\medskip

First we investigate the behavior of geodesics in $C$. To do this we make use of 
the Gromov-Hausdorff convergence $C_{M_i}\to C$.

Recall that for every $t\in [0, t_0]$, we set 
\[
   C_{M_i,t}=\{ x\in C_{M_i}\,|\,d(x,\partial M_i)=t\,\}.
\]
We also use $t$ to denote the distance functions on $C$ and $C_{M_i}$ from
$C_0$ and $\partial M_i$ respectively.

Let $\gamma:[0,\ell]\to C$ be a unit speed geodesic, and $\xi=\frac{\partial}{\partial t}$  the 
unit vector fields on $C$. Take a geodesic $\gamma_i$ in $C_{M_i}$ such $\gamma_i\to \gamma$.
We denote by $\Pi_t^i$  the second fundamental form of
$C_{M_i,t}$: 
\[
                 \Pi_t^i(V,W)= - \langle\nabla_V\xi_i, W\rangle, \,\, V, W\in T(C_{M_i,t}),
\]
where $\xi_i=\frac{\partial}{\partial t}$  is the unit vector fields on $C_{M_i}$.
Consider the function $\rho_i(s)=t(\gamma_i(s))=|\gamma_i(s), \pa M_i|$. We have
\begin{align*}
\rho_i'(s) &   = \langle \xi_i(\gamma_i(s)), \dot\gamma_i(s)\rangle \\
\rho_i''(s) & =\langle\nabla_{\dot\gamma_i^T}\,\xi_i,  {\dot\gamma_i^T}\rangle  = -\Pi(\dot\gamma_i^T,  \dot\gamma_i^T)\\
                              &=  \frac{\phi'(\rho_i(s))}{\phi(\rho_i(s))} |\dot\gamma_i^T(s)|^2,
\end{align*}
where $\dot\gamma_i^T(s)$ is the component of $\dot\gamma_i(s)$ tangent to $C_{M_i,\rho_i(s)}$.
Note that $0\ge \rho_i''(s)\ge -c$ for some uniform constant $c>0$. 
In particular, we have

\begin{lem}\label{lem:concave} %\label{lem:concave}
 $\rho_i$and $\rho$ are concave functions.
\end{lem}
%
%%%%%%%%%%%%%
%Suppose $\gamma(0)=p$. Since $\rho(0)=\rho'(0)=0$, we see
%\[
%     |\rho'(s)| \le \int_0^{\ell} |\rho''(s)|\,ds \le cs, \,\, \rho(\ell) =\int_0^{\ell} \rho'(s)\,ds \le c\ell^2/2,
%\]
%where $c$ is a uniform constant. In particular, if $q\in B^C(p,\delta)\cap d_{a_m}^{-1}(r_m)$, we have
%\[
%        \angle(\uparrow_q^p,C_{t(q)}) < \tau(\delta).
%\]
%Since $\Sigma_q(C)={\rm Susp}(\Sigma(C_{t(q)})$, it follows that
%\[
%    \pi/2 < \angle(\xi(q), \uparrow_q^p) < \pi/2 + \tau(\delta).
%\]
%%%%%%%%%%%%%%%%%

\begin{lem} \label{lem:ext/int}
For every $t\in [0, t_0]$, and $p_1,p_2\in C_t$, we have
\[
    \left| \frac{|p_1,p_2|_{C_t^{int}}}{|p_1,p_2|_{C}} -1\right| < O(|p_1,p_2|_C^2).
\]
%where $C_t=\{ x\in C\,|\,d(x,C_0)=t\,\}$.
\end{lem}
\begin{proof}
Let $\gamma:[0,\ell]\to C$ be a unit speed  minimal geodesic joining $p_1$ to $p_2$.
Take a unit speed  minimal geodesic $\gamma_i:[0,\ell_i]\to C_{M_i}$ such that $\gamma_i\to \gamma$ 
under the Gromov-Hausodorff convergence $C_{M_i}\to C$.
We may assume that $\rho_i(\gamma_i(0))=\rho_i(\gamma_i(\ell_i))=t$.
Putting 
\[
    \rho_i(s)=\rho_i(\gamma_i(s)) = |\gamma_i(s), \partial M_i|,
\]
$\rho_i(s)$ takes  a maximum $t_i=\rho_i(u_i)>t$ at some $u_i\in (0,\ell)$.
By the mean value theorem, we obtain
\[
\frac{\rho_i(u_i)-t}{u_i} = \rho_i'(v_i), \,\, \frac{t- \rho_i(u_i)}{\ell_i -u_i} = \rho_i'(v_i'), \,\, 
\frac{\rho_i'(v_i')-\rho'_i(v_i)}{v_i'-v_i}=\rho''(w_i),
\]
for some $0<v_i<u_i<v_i'<\ell$ and $v_i<w_i<v_i'$.
Adding the first two equalities, we get
\begin{equation}
 \begin{aligned}
   \rho_i(u_i)-t & \le \frac{(\ell_i-u_i)u_i}{\ell_i}(v_i'-v_i)(-\rho_i''(w_i)) \\
                 & \le c|\gamma_i(0),\gamma_i(\ell_i)|^2.      \label{eq:level-deviat}
 \end{aligned}
\end{equation}
Setting $t^*:= \max_{[0,\ell]}\rho$ and letting $i\to\infty$, we have 
\begin{equation}
  t^*-t = \max_{[0,\ell]}\rho -t \le c|p_1,p_2|^2,  \label{eq:deviat3}
\end{equation}
and hence
\begin{equation}
      \left | \frac{\phi(t)}{\phi(t^*)} -1 \right|  \le c'|p_1,p_2|^2.  \label{eq:deviat}
\end{equation}
Let $\pi_t: C\to C_t$ be the canonical projection.
Since $\pi_t$ has Lipschitz constant $\frac{\phi(t)}{\phi(t^*)}$ on the domain bounded by $C_t$ and $C_{t^*}$, 
it follows from \eqref{eq:deviat} that
\begin{align}
      |p_1,p_2|_{C_t^{int}}\le L(\pi_t\circ\gamma)\le \frac{\phi(t)}{\phi(t^*)}|p_1,p_2|<(1+O(|p_1,p_2|^2)|p_1,p_2|.  \label{eq:deviat2}  
\end{align}
This completes the proof.
\end{proof}

\begin{lem} \label{lem:angle}
For every $p_1,p_2\in C_t$ and  unit speed  minimal geodesic $\gamma:[0,\ell]\to C$ joining $p_1$ to $p_2$, we have
\[
    \rho'(0) \le C|p_1,p_2|,
\]
where $\rho(s) = |\gamma(s), C_t|$.
\end{lem}

\begin{proof}
Let $\rho(s)$ takes the maximum at $s=s_0$. Using the mean value theorem, we obtain
$-\rho'(0)/s_0\ge \inf \rho'' \ge -c$, from which the conclusion is immediate.
\end{proof}

%%%%%%%%%   \eta_{\infty}  %%%%%%%%%%%%%%%%%%%%%%%%%%%%%%%%%%%%%%%%%%%%%%%%%%%%%%%%%%%%%%%%%
Next we discuss a rescaling limit of the map $\eta:C\to Y$. 
Fix $p\in C_0$ and $x=\eta_0(p)\in X$, and
let $t_i$ be an arbitrary sequence of positive numbers with $\lim t_i=0$.
Passing to a subsequence, we may assume that 
\[
     \eta_i=\eta:  \left(\frac{1}{t_i}C, p\right) \to \left(\frac{1}{t_i}Y, x\right)
\]
converges to a $1$-Lipschitz map 
\[
   \eta_{\infty}:  (T_p(C),o_ p) \to (T_x(Y), o_x)
\]
between  the tangent cones of the Alexandrov spaces.
We may also assume that 
$\left( \frac{1}{t_i} X, x \right)$ converges to a closed subset  $(T^*_x(X), o_x)$ of 
$(T_x(Y), o_x)$ under the convergence $\left(\frac{1}{t_i}Y, x\right) \to (T_xY, o_x)$.

\begin{slem} \label{slem:lim-eta}
$\eta_{\infty}:  T_p(C)\setminus T_p(C_0) \to T_xY\setminus T_x^{*}(X)$ is a bijective local isometry.
\end{slem}

\begin{proof}
Let $\tilde\rho=|\, \cdot \,, C_0|$,  $\rho=|\,\cdot \,, X|$.
Under the  $1/t_i$-rescaling, we may assume that  $\tilde\rho$,  $\rho$ converges
to  the maps 
\[
      \tilde\rho_{\infty}=|\,\cdot \,, T_p(C_0)|, \,\,\, \rho_{\infty}=|\,\cdot \,, T_x^*(X)|,
\]
respectively
satisfying $\tilde\rho_{\infty} = \rho_{\infty}\circ \eta_{\infty}$.
For any $\tilde w\in  T_p(C)\setminus T_p(C_0) $, let $\e=\tilde\rho_{\infty}(\tilde w)$ and 
$w=\eta_{\infty}(\tilde w)$.
Since $\rho_{\infty}(w)=\e$, it is easily checked that $\eta_{\infty}:B(\tilde w,\e/2)\to B(w, \e/2)$
is an isometry.

Next let us show that $\eta_{\infty}:  T_p(C)\setminus T_p(C_0) \to T_xY\setminus T_x^{*}(X)$ is bijective.
Suppose that $w:=\eta_{\infty}(\tilde w_1)=\eta_{\infty}(\tilde w_2)$ for $\tilde w_j\in T_p(C)\setminus T_P(C_0)$.
Take $q_1^i, q_2^i\in C$  such that $q_j^i$  converge to 
$\tilde w_j$ under the  $1/t_i$-rescaling.  Let  $\eta(q_j^i)=y_j^i$. Since $y_j^i$ converges to the same point $w$,
any minimal geodesic joining $y_1^i$ and $y_2^i$ does not meet $X$.
This implies that $|q_1^i, q_2^i|=|y_1^i, y_2^i|$. However this must imply that $\tilde w_1=\tilde w_2$.
Hence $\eta_{\infty}$ is injective on $T_p(C)\setminus T_p(C_0)$. It is easy to see that 
$\eta_{\infty}:  T_p(C)\setminus T_p(C_0) \to T_xY\setminus T_x^{*}(X)$ is surjective, and hence the proof is omitted.
\end{proof}

%%%%%%%%

%%%%%%%%%%%%%%%%%%%%%%%%%%%%%%%%%%%%%%%%%%%%%%%%%%%%%%%%%%%%%%%%%%%%%%%%%%%%%%%

\subsection{Spaces of directions and differential of $\eta_0$}  \label{ssec:differential}

In this subsection, we study the the spaces of directions of $C$, $C_0$, $Y$ and $X$ at the points where
the gluing is done, and the relation between them. 
We also study the differential of the gluing map $\eta_0$ at those points.

\begin{lem} \label{lem:susp0}
For every $p\in C_0$, let $\tilde\gamma_{+}(t) = (p,t)$ and $\gamma_+(t)=\eta(\tilde\gamma(t))$. Then 
\begin{enumerate}
 \item  $\Sigma_p(C)$ is isometric to the half-spherical suspension $\{ \tilde\gamma_{+}'(0) \}*\Sigma_p(C_0);$
 \item  for every $s\in (0, t_0)$, $\Sigma_{(p, s)}(C)$ and $\Sigma_{\eta(p,s)}(\eta(C))$ are isometric to 
          the spherical suspensions $\{ \pm\tilde\gamma_{+}'(s) \}*\Sigma_p(C_0)$ and
          $\{ \pm\gamma_{+}'(s) \}*\Sigma_p(C_0)$ respectively.
\end{enumerate}
%$\tilde\xi_{+}\in \Sigma_p(C)$ be the direction of the mimimal geodesi 
%$\tilde\gamma_{+}$ from $p$ to $C_{t_0}$.
%Then  $\Sigma_p(C)$ is isometric to the half-spherical suspension $\{\tilde\xi_{+} \}*\Sigma_p(C_0)$.
\end{lem}

\begin{proof}
From the suspension structure $C=C_0\times_{\phi}[0,t_0]$, obviously we have $T_p(C)=T_p(C_0)\times [0,\infty)$,
which implies the conclusion $(1)$. Since both $\tilde\gamma_{+}$ and $\gamma_{+}$ are geodesic, the splitting
theorem shows $(2)$.
\end{proof}

\begin{lem} \label{lem:susp1}
For every  $x\in X$ and  $\xi\in T_x(Y)\setminus K(\Sigma_x(X))$ which is not a perpendicular direction, 
assume that there is a geodesic  $\gamma:[0,\ell]\to Y$  with $\gamma'(0)=\xi$, and let 
\[
    \tilde\gamma=\eta^{-1}(\gamma), \,\,\,  
    \tilde\sigma=\tilde\pi\circ\tilde\gamma, \,\,\, \sigma = \pi \circ \gamma,  \,\,\, p:=\tilde\gamma(0).
\]
Let $\tilde\gamma_{+}$ be the perpendicular to $C_0$ at $p$, and set $\gamma_+ = \eta(\tilde\gamma_+)$.
%and set  equivalently 
%$\eta_0(\tilde\sigma)=\sigma$).  
Put 
\[
     \tilde\xi = \tilde\gamma'(0),   \,\,\,   \tilde\xi_{+} = \tilde\gamma_{+}'(0),    \,\,\,   \tilde v = \tilde\sigma'(0),  \,\,\,  
     \xi_{+} = \gamma_{+}'(0).
 \]
Then 
\begin{enumerate}
 \item $\sigma$ defines a unique vector $v=\sigma'(0)\in K(\Sigma_x(X))$ and we have 
     \begin{equation}
     %\left\{
        \begin{aligned}
              &\angle(\xi_+, \xi) = \angle(\tilde\xi_+, \tilde\xi),  \,\,\,   \angle(\xi,v) = \angle(\tilde\xi, \tilde v) \\
              & \angle(\xi_{+},\xi)+\angle(\xi, v)=\angle(\xi_{+}, v)=\pi/2;
                       \label{eq:sus2}
        \end{aligned}
     \end{equation}
  \item  there is a unique limit $\eta_{\infty}:T_p(C) \to T_x(Y)$ of 
     $\eta_t =\eta: (\frac{1}{t} C,p) \to (\frac{1}{t}Y, x)$ as $t\to 0$, and we have 
    \[
          \eta_{\infty}(\tilde v) = v, \,\,\,   |\tilde v| = |v|.
     \]
\end{enumerate}
%
%Let $v\in\Sigma_x(X)$ be a direction defined by the curve $\sigma$. 
%By definition,  this means that  $v=\lim_{i\to\infty} \uparrow_x^{\sigma(t_i)}$ for a sequence $t_i\to 0$.  
%
%
%Let $\tilde\gamma$, $\tilde\gamma_+$ and $\tilde\gamma_t$ be geodesics in $C$ 
%such that $\eta(\tilde\gamma)=\gamma$, $\eta(\tilde\gamma_+)=\gamma_+$
%and  $\eta(\tilde\gamma_t)=\gamma_t$, 
%and set  $\tilde\sigma=\tilde\pi\circ\tilde\gamma$, or equivalently 
%$\eta_0(\tilde\sigma)=\sigma$).  
%there is a unique perpendicular  $\gamma_{+}$  to $X$
%at $x$ and a unique  $v\in\Sigma_x(X)$ such that if $\xi_{+}\in \Sigma_x(Y)$ denotes the direction of 
%$\gamma_{+}$, then we have
%     
% %  \item conversely, every $v\in \Sigma_x(X)$ satisfies  $\angle(\xi_+, v)=\pi/2;$
%%   \item there exists a unique limit 
%%        \[
%%                  \lim_{\delta\to 0}(\frac{1}{\delta}X, x) = (T_x(X),o_x)=(K(\Sigma_x(X)),o_x).
%%        \]
%%    under the convergence  $\lim_{\delta\to 0}(\frac{1}{\delta} Y, x) =(T_x(Y),o_x)$
%% \end{enumerate}
\end{lem}

\begin{proof}
%First consider the case when $\xi=\gamma'(0)$ 
%for a minimal geodesic $\gamma:[0,\ell]\to Y$ starting from $x$.
%Set 
%$$
%                    \sigma(t):=\pi(\gamma(t)),
%$$
%and
%let $\gamma_t:[0,t_0]\to Y$ be the perpendicular to $X$ at $\sigma(t)$ 
%through $\gamma(t)$.
%The limit $\gamma_{+}$ of $\gamma_t$ as $t \to 0$ is a perpendicular to $X$ at $x$.
%
Let $\zeta\in\Sigma_x(X)$ be a direction defined by the curve $\sigma$. 
By definition,  this means that  $\zeta=\lim_{i\to\infty} \uparrow_x^{\sigma(t_i)}$ for a sequence $t_i\to 0$.  
%Let $\tilde\gamma$, $\tilde\gamma_+$ and $\tilde\gamma_t$ be geodesics in $C$ 
%such that $\eta(\tilde\gamma)=\gamma$, $\eta(\tilde\gamma_+)=\gamma_+$
%and  $\eta(\tilde\gamma_t)=\gamma_t$, 
%and set  $\tilde\sigma=\tilde\pi\circ\tilde\gamma$, or equivalently 
%$\eta_0(\tilde\sigma)=\sigma$).  
Since $\gamma$ is minimal, so is $\tilde\gamma$.
%Note that $\tilde\sigma(t)=\tilde\pi(\tilde\gamma(t))$.
%
%Let $\tilde\xi$ and $\tilde \xi_+$ be the directions at $p$ defined by 
%$\tilde\gamma$ and $\tilde\gamma_+$ respectively.
%%
%Let $\tilde v$ be the direction at $p$ defined by $\tilde\sigma$.
Note that $\tilde v$ is uniquely determined since  $\tilde\sigma$ is a shortest curve.
From Lemma \ref{lem:susp0}, we have
\begin{equation}
     \angle(\tilde\xi_{+},\tilde\xi)+\angle(\tilde\xi, \tilde v)=\angle(\tilde\xi_{+},\tilde v)=\pi/2.
                      \label{eq:sus1}
\end{equation}

Now we show \eqref{eq:sus2}.
Consider the $1/t_i$-rescaling limits,
\[
      (T_x(Y), o_x) = \lim_{i\to\infty}\left( \frac{1}{t_i}Y,x\right), \,\,
       (T_p(C), o_p)= \lim_{i\to\infty}\left( \frac{1}{t_i}C, p\right).
\]
Let $\gamma_{t_i}$ (resp. $\tilde\gamma_{t_i}$) be the perpendicular to $X$ at $\sigma(t_i)$ (resp. to $C_0$ at $\tilde\sigma(t_i)$).
Passing to a subsequence, we may assume that the  quadruplet $(\gamma_{+}, \gamma,  \sigma,  \gamma_{t_i})$ 
converges to  $(\gamma_{+\infty}, \gamma_{\infty}, \sigma_{\infty}, \gamma_{\infty 1})$ 
under the convergence $\left( \frac{1}{t_i} Y, x\right)\to\left(T_x(Y), o_x\right)$.
For instance, this explicitly means that the Lipschitz curve $\frac{1}{t_i} \sigma(t_i t)$ converges to a Lipschitz curve   $\sigma_{\infty}(t)$ in $T_x(Y)$. Thus $\gamma_{+\infty}$ and $\gamma_{\infty 1}$ are perpendicular to $T_x^{*}(X)$ at $o_x$ and $\sigma_{\infty}(1)$
and $\gamma_{\infty}$ is the geodesic from $o_x$ with $\gamma_{\infty}(1)=\xi$.
Here we  assume that 
$\left( \frac{1}{t_i} X, x \right)$ converges to a closed subset  $(T^*_x(X), o_x)$ of 
$T_x(Y), o_x)$.

Similarly passing to a subsequence, we may assume that the quadruplet $(\tilde \gamma_{+}, \tilde\gamma,  \tilde\sigma,  \tilde\gamma_{t_i})$ 
converges to  $(\tilde \gamma_{+\infty}, \tilde\gamma_{\infty}, \tilde\sigma_{\infty}, \tilde\gamma_{\infty 1})$ 
under the convergence $\left( \frac{1}{t_i}C, p\right)\to\left(T_p(C), o_p\right)$.
Thus $\tilde\gamma_{+\infty}$ and $\tilde\gamma_{\infty 1}$ are perpendicular to $T_p(C_0)$ at $o_p$ and $\tilde\sigma_{\infty}(1)$
and $\tilde\gamma_{\infty}$ is the geodesic from $o_p$ with $\tilde\gamma_{\infty}(1)=\tilde\xi$.

%
%Passing to a subsequence, we may assume that 
%\begin{enumerate}
% \item Lipschitz curve $\frac{1}{t_i} \sigma(t_i t)$ converges to a Lipschitz curve   $\sigma_{\infty}(t)$ in $T_x(Y);$
% \item  Lipschitz curve $\frac{1}{t_i} \tilde\sigma(t_i t)$ converges to a Lipschitz curve   $\tilde\sigma_{\infty}(t)$
%    in $T_x(Y)$. 
%\end{enumerate}
%Similarly, the geodesic $\gamma_{t_i}$ (resp.  $\tilde\gamma_{t_i}$)
%converges to  a geodesic ray $\gamma_{\infty 1}$ (resp.  $\tilde \gamma_{\infty 1}$)
%starting from $\sigma_{\infty}(1)$ through $v=\gamma_{\infty}(1)$ (resp. from $\tilde\sigma_{\infty}(1)$
%through $\tilde v=\tilde\gamma_{\infty}(1)$).
%%
%
%and perpendicular to 
%$T_x(X):=\lim_{\delta\to 0}\left( \frac{1}{\delta}X,x\right)\subset T_x(Y)$ 
%(resp. $T_p(C)$) under those convergences.
%

We set
\[
                  \rho(t)= |C_0, \tilde\gamma(t)|=|X, \gamma(t)|.
\]
Notice that 
%\begin{equation}
\begin{enumerate}
 \item $|\xi, \sigma_{\infty}(1)| =|\tilde\xi, \tilde\sigma_{\infty}(1)| =\rho'(0);$
 \item $\rho'(0)=  |\tilde\xi|\sin  \angle(\tilde\xi, \tilde v).$
                     \label{eq:sus2'}
\end{enumerate}
%\end{equation}

%%

Let $\tilde\lambda$ be a minimal geodesic joining $\tilde\xi=\tilde\gamma_{\infty}(1)$ to the geodesic 
$\tilde\gamma_{+\infty}$. % in the direction $\tilde\xi_{+}$. %Since $L(\tilde\lambda)=L(\
Let $\eta_{\infty}:T_p(C) \to T_x(Y)$ be any limit of 
     $\eta_{t_i} =\eta: (\frac{1}{t_i} C,p) \to (\frac{1}{t_i}Y, x)$.
Since $\eta_{\infty}$ is $1$-Lipschitz, we have 
\begin{align*}
      |\tilde\xi|\sin\angle(\tilde\xi_{+}, \tilde\xi) =L(\tilde\lambda)=L(\eta_{\infty}\circ\tilde\lambda)\ge |\xi| \sin\angle(\xi_{+}, \xi),
\end{align*}
and hence
\begin{align}
      \angle(\xi_{+}, \xi)\le \angle(\tilde\xi_{+}, \tilde\xi).    \label{eq:sin}
\end{align}
Next we show that 
\begin{align}
       \angle(\xi, \zeta) = \angle(\tilde\xi, \tilde v).  \label{eq:=v}
\end{align}
Put for simplicity
\begin{align*}
    & \tilde\theta :=\angle(\tilde\xi, \tilde v), \qquad 
                                   \tilde\theta_i :=\angle(\tilde\gamma'(t_i), T_{\tilde\gamma(t_i)}  C_{\rho(t_i)}),  \\
    &   \theta:= \angle(\xi,  \zeta), \qquad 
                                   \theta_i :=\angle(\gamma'(t_i),  T_{\gamma(t_i)} C_{\rho(t_i)}^Y).
\end{align*}

From the warping product structure of $C$, we easily have 
\[
     \lim_{i\to\infty}\tilde\theta_i = \tilde\theta.
\]
On the other hand, under the convergence $\left( \frac{1}{t_i}C, p\right)\to\left(T_p(C), o_p\right)$
(resp.  under the convergence  $\left( \frac{1}{t_i}Y,x\right) \to \left(T_xY,o_x\right)$), 
we may assume that $C_{s t_i}$ converges to some space, denoted by  $C_{s\infty}$.
(resp.  $C^Y_{s t_i}$ converges to some space $C^Y_{s\infty}$).
Then we have  
\begin{enumerate}
 \item $\tilde\theta =\angle(\tilde\gamma_{\infty}'(0), \tilde\sigma_{\infty}'(0));$ 
 \item  $\tilde\theta =\lim \tilde\theta_i = \angle(\tilde\gamma_{\infty}'(1), T_{\tilde\gamma_{\infty}(1)}(C_{\rho'(0)\infty}))$.
\end{enumerate}
Since $\angle(\tilde\gamma_{\infty}'(1), \tilde\gamma_{\infty 1}'(\rho'(0))) = \pi/2-\tilde\theta$,
we have
\begin{align}
     \angle  o_p \tilde\gamma_{\infty}(1) \tilde\sigma_{\infty}(1)) = \pi/2 -\tilde\theta.\label{eq:flat1}
\end{align}
On the other hand, since $\eta: C\setminus C_0\to C^Y\setminus X$ is a local isometry, we have
\[
        \tilde\theta_i = \theta_i.    %\angle(\gamma'(t_i), T_{\gamma(t_i)}  C_{\rho(t_i)}).
\]
%Now under the convergence  $\left( \frac{1}{t_i}Y,x\right) \to \left(T_xY,o_x\right)$, 
%we may assume that $C_{s t_i}$ converges to some space, denoted by  $C_{s\infty}$.
%Therefore,  we may assume that  $C_{\rho(t_i)}$ converges to  $C_{\rho'(0)\infty}$.
%%In view of \cite{BGP},  
From the lower semi-continuity, of angles, 
we have 
\[
      \lim\theta_i = \angle(\gamma_{\infty}'(1), T_{\gamma_{\infty}(1)}(C_{\rho'(0)\infty})).
\]
It follows from the spherical suspension structure of $\Sigma_{\gamma_{\infty}(1)}T_p(C)$ that  
\begin{equation*}
 \begin{aligned}
  \angle(\gamma_{\infty}'(1), \gamma_{\infty 1}'(\rho'(0)))& 
                                            = \pi/2 - \angle(\gamma_{\infty}'(1), T_{\gamma_{\infty}(1)}(C_{\rho'(0)\infty}))  \\
                                          &=   \pi/2-\tilde\theta,
 \end{aligned}
\end{equation*}
%\end{equation}•
and hence  
\begin{align}
     \angle o_x  \xi \sigma_{\infty}(1) = \pi/2 -\tilde\theta. \label{eq:flat2}
\end{align}    
By \eqref{eq:flat1} and \eqref{eq:flat2},  the two Euclidean triangles $\triangle o_x\xi \sigma_{\infty}(1)$ 
and $\triangle o_p\tilde\xi\tilde \sigma_{\infty}(1)$ are
congruent to each other, and we conclude that $\angle(\xi, \zeta) = \angle(\tilde \xi,\tilde v)$ as required.

The first variation formula immediately implies  $\angle(\xi_{+}, \zeta)\ge \pi/2$. 
It follows \eqref{eq:sin} and  \eqref{eq:=v} that 
\begin{equation}
 \begin{aligned}
 \pi/2  \le  \angle(\xi_{+},\zeta) &\le \angle(\xi_{+},\xi)+\angle(\xi, \zeta) \\
                   & \le \angle(\tilde\xi_{+},\tilde\xi)+\angle(\tilde\xi, \tilde v) = \pi/2.  \label{eq:xi_+}
\end{aligned}
\end{equation}
%Therefore from \eqref{eq:sin} and \eqref{eq:xi_+}, 
Thus we conclude that 
\[
      \angle(\xi_{+},\xi)+\angle(\xi, \zeta) = \angle(\xi_{+}, \zeta) =\pi/2, 
\]
which shows the  uniqueness of $\zeta$. Namely  $\sigma$ determines a unique direction at $x$.
Note that 
\begin{equation}
%\left \{
 \begin{aligned}
         v &:= \sigma'(0) = \lim_{i\to\infty} \frac{|x,\sigma(t_i)|}{t_i}\zeta=|o_x,\sigma_{\infty}(1)|\zeta, \\
         |v|& = |o_x, \sigma_{\infty}(1)| = |o_p, \tilde \sigma_{\infty}(1)| = |\tilde v|. \label{eqrtatio-lim}
\end{aligned}
\end{equation}
Since $\eta_{\infty}(\tilde v) = v$, this shows that $\eta_{\infty}$ does not depend on 
the choice of $t_i\to 0$.
This completes the proof.
\end{proof}

\begin{rem}  
The argument in the  proof of Lemma \ref{lem:susp1} also shows that 
 \begin{align*}
    |\tilde\xi| \cos\tilde \theta & =|o_p,\tilde\sigma_{\infty}(1)| = |o_x,\sigma_{\infty}(1)| \\
                           & \le L(\sigma_{\infty}|_{[0,1]}) \le L(\tilde\sigma_{\infty}|_{[0,1]})=  \cos \tilde\theta,
\end{align*}
which implies that 
\begin{align}
%\textstyle{ 
  \sigma_{\infty}  \,\,\, \mathrm{ is\,\, minimizing\,\, in\,\, the \,\,direction}  \,\,\, \sigma'(0). \hspace{3cm} \label{sigma-min}
\end{align}
\end{rem}

\begin{cor} \label{cor:susp1}
For every  $x\in X$ and  $\xi\in\Sigma_x(Y)\setminus \Sigma_x(X)$ which is not a perpendicular direction, 
there is a unique perpendicular direction  $\xi_{+}\in \Sigma_x(Y)$  to $X$ at $x$ 
and a unique  $v\in\Sigma_x(X)$ such that 
      \begin{equation}
         \angle(\xi_{+},\xi)+\angle(\xi, v)=\angle(\xi_{+}, v)=\pi/2;
                       \label{eq:sus3}
     \end{equation}
 %  \item conversely, every $v\in \Sigma_x(X)$ satisfies  $\angle(\xi_+, v)=\pi/2;$
%   \item there exists a unique limit 
%        \[
%                  \lim_{\delta\to 0}(\frac{1}{\delta}X, x) = (T_x(X),o_x)=(K(\Sigma_x(X)),o_x).
%        \]
%    under the convergence  $\lim_{\delta\to 0}(\frac{1}{\delta} Y, x) =(T_x(Y),o_x)$
% \end{enumerate}
\end{cor}

\begin{proof}
This immediately follows from Lemma \ref{lem:susp1} and a limit argument.
\end{proof}

By Lemma \ref{lem:susp1}, for every geodesic $\gamma$ in $Y$ starting from $x\in X$
such that $\gamma'(0)\in \Sigma_x(Y) \setminus \Sigma_x(X)$,  
the Lipschitz curve $\sigma=\pi(\gamma)$
determines a unique direction $[\sigma]\in \Sigma_x(X)$.  In general, we call such a direction $[\sigma]$
an {\it intrinsic direction} if $\sigma$ is a Lipschitz curve  in $X$ starting from $x$ 
and having a unique direction $[\sigma]=\sigma'(0)$  in the sense 
that for any sequence $t_i\to 0$, $\uparrow_x^{\sigma(t_i)}$ converges to $[\sigma]$.

The next lemma shows that every direction in $\Sigma_x(X)$ can be approximated by 
intrinsic directions.

\begin{lem} \label{lem:v}
For every $v\in\Sigma_x(X)$, we have the following:
\begin{enumerate}
 \item For any perpendicular direction $\xi_+\in\Sigma_x(Y)$, we have
     \[
       \angle(\xi_+, v)=\pi/2.
     \]
 \item There are intrinsic directions $[\sigma_i]\in \Sigma_x(X)$ satisfying 
         \[
                 \lim [\sigma_i] =v.
         \]
\end{enumerate}
\end{lem}

\begin{proof}
For every $v\in\Sigma_x(X)$ take a sequence $y_i\in X$ with $y_i\to x$ and
$v_i:=\uparrow_x^{y_i}\to v$.  Let $\mu_i:[0, s_i]\to Y$ be a minimal geodesic 
from $x$ to $y_i$. Let $\gamma_+$ be a perpendicular to $X$ at $x$ with $\gamma_+'(0)=\xi_+$.
Let $\lambda_i$ be a minimal geodesic joining $\gamma_+(t_0)$ to 
$y_i$. Considering perpendiculars to $X$ through the points of $\lambda_i$ 
and taking the limit, 
we obtain  a perpendicular  $\gamma_{y_i}$  to $X$ at $y_i$.
%
%We may assume that $\lambda_i(t_0)\to\gamma_{+}(t_0)$.
%Take a point $y_i\in \lambda_i$ such that $|\angle (\xi_+, \xi_i)-\pi/4|<\epsilon_i$
%with $\lim\epsilon_i=0$, where $\xi_i:=\uparrow_x^{y_i}$.
Let $\gamma_i:[0,t_i]\to Y$ be a minimal geodesic from $x$ to $\gamma_{y_i}(s_i)$, and set 
\[
         \sigma_i(t):=\pi(\gamma_i(t)), \,\,\tilde\gamma_i=\eta^{-1}(\gamma_i), \,\,
          \tilde\sigma_i=\tilde\pi(\tilde\gamma_i).
\]
By Lemma \ref{lem:susp1}, $\sigma_i$ defines a unique direction $\hat v_i\in\Sigma_x(X)$ such that
\begin{align}
     \angle(\xi_{+},\xi_i)+\angle(\xi_i, \hat v_i)=\angle(\xi_{+}, \hat v_i)=\pi/2, \label{eq:sus_i}
\end{align}
where $\xi_i=\gamma_i'(0)$. Note that $y_i=\sigma_i(t_i)$.  

We now use an argument similar to that of Lemma \ref{lem:susp1}.
Consider the  convergence
\[
  \left(\frac{1}{t_i}Y, x\right)\to \left(T_x(Y), o_x\right),   \,\,\left(\frac{1}{t_i}C, p\right)\to\left(T_p(C), o_p\right). \,\, %t_i=|x,x_i|.
\]
%Then $x_i$ converges to $v\in\Sigma_x(X)\subset T_x(Y)$ under the above convergence.
Passing to a subsequence, we may assume that $\xi_i$ converge to some $\xi\in\Sigma_x(Y)\subset T_x(Y)$ .
We may also assume that 
\begin{enumerate}
% \item[(a)] $s_i/t_i$ converges to $s_{\infty}>0;$
 \item[(a)] $\gamma_i(t_is)$ and $\sigma_i(t_is)$ converge to geodesic $\gamma_{\infty}(s)$ and 
          a Lipschitz curve $\sigma_{\infty}(s)$ respectively;
 \item[(b)] $\tilde \gamma_i(t_is)$ and $\tilde\sigma_i(t_is)$ converge to geodesics 
          $\tilde\gamma_{\infty}(s)$ and $\tilde\sigma_{\infty}(s)$ in $T_p(C)$ respectively.
\end{enumerate}
%We show that $\sigma_{\infty}$ is a minimal geodesic.
Let $\eta_{\infty}:  (T_p(C),o_ p) \to (T_xY, o_x)$ be the  $1$-Lipschitz map  defined in Lemma 
\ref{lem:susp1} (2)  as the  limit of 
\[
     \eta_i=\eta:  \left(\frac{1}{t_i}C, p\right) \to \left(\frac{1}{t_i}Y, x\right).
\]
Note that  $\eta_{\infty}(\tilde\sigma_{\infty}(s))=\sigma_{\infty}(s)$.
Consider the geodesic triangles 
\begin{gather*}
\Delta_{o_x}:=\Delta o_x\gamma_{\infty}(1)\sigma_{\infty}(1)\subset T_x(Y), \\ 
\Delta_{o_p}:=\Delta o_p\tilde\gamma_{\infty}(1)\tilde\sigma_{\infty}(1) \subset T_p(C).
\end{gather*}
%
%Obviously, we obtain
%\begin{gather*}
%   |o_x, \gamma_{\infty}(1)| =  |o_p, \tilde\gamma_{\infty}(1)|, \\
%      |\gamma_{\infty}(1),  \sigma_{\infty}(1)| =  |\tilde\gamma_{\infty}(1),  \tilde\sigma_{\infty}(1)| 
%\end{gather*}
%Note that both  $\Sigma_{\gamma_{\infty}(s_{\infty})}(T_x(Y))$ and  
%$\Sigma_{\tilde\gamma_{\infty}(s_{\infty})}(T_p(C))$ 
%have the suspension structure  (see Lemma \ref{lem:susp0}). 
%
%Then from construction we have
%\[
%    |\gamma_{\infty}(s), \sigma_{\infty}(s)| = |\tilde\gamma_{\infty}(s), \tilde\sigma_{\infty}(s)|.
%\]
An argument similar to that in Lemma \ref{lem:susp1} implies that 
\begin{align}
   \angle o_x\gamma_{\infty}(1)\sigma_{\infty}(1)
                        = \angle o_p\tilde\gamma_{\infty}(1)\tilde\sigma_{\infty}(1),  \label{eq:equal}
\end{align}
It should be remarked that in the case of Lemma \ref{lem:susp1}, 
the geodesic $\gamma:[0,\ell]\to Y$ in the direction $\xi$ was given in the beginning, and 
we considered  the points $\gamma(t_i)$ with $t_i\to 0$.  On the other hand, in the present case, 
we have only geodesic $\gamma_i:[0,t_i]\to Y$.
Therefore  we take a point $z_i\in Y$ instead,  in such a way that 
\[
       \tilde\angle x\gamma_i(t_i) z_i >\pi -o_i, \qquad |\gamma_i(t_i) ,z_i|=t_i,
\] 
where $\lim o_i = 0$.
Then with almost parallel argument,  we obtain \eqref{eq:equal} % $\angle o_x\xi v = \angle o_p\tilde \xi \tilde v$,
and that $\Delta_{o_x}$ and $\Delta_{o_p}$ are congruent to each other as Euclidean flat triangles.
%spans a flat triangle isometric to 
%ones in $\mathbb R^2$.
%From the above equalities, 
In particular we conclude that 
\begin{align}
          |o_x, \sigma_{\infty}(1)| =  |o_p, \tilde\sigma_{\infty}(1)|.  \label{eq:base-eq}
\end{align}
Since $L(\sigma_{\infty})\le L(\tilde\sigma_{\infty})$, this implies that 
$\sigma_{\infty}$ is a minimal geodesic in the direction $v$.
As in Lemma \ref{lem:susp1}, \eqref{eq:equal} also implies that  
\begin{align}
     \angle(\xi_{+},\xi)+\angle(\xi, v)=\angle(\xi_{+}, v)=\pi/2, \label{eq:sus_*}
\end{align}
where $\xi = \lim \xi_i$. This proves (1).
It follows from \eqref{eq:sus_i},  \eqref{eq:sus_*} that 
\begin{align}
              v = \lim \hat v_i.\,  \label{eq:v}
\end{align}
which shows (2).

Furthermore, it follows from
\begin{equation*}
% \begin{aligned*} 
   \begin{cases} 
   &\frac{L(\tilde\sigma_i)}{t_i}\ge \frac{L(\sigma_i)}{t_i}\ge \frac{L(\mu_i)}{t_i}, \\
   &\lim  \frac{L(\tilde\sigma_i)}{t_i} =|o_p, \tilde\sigma_{\infty}(1)|,  \,\,\, 
                             \lim  \frac{L(\mu_i)}{t_i} =|o_x, \sigma_{\infty}(1)|, 
     \end{cases}
%\end{aligned*}
\end{equation*}
that
\begin{align}
         \lim \frac{L(\sigma_i)}{t_i} = L(\sigma_{\infty}).
\end{align}
\end{proof}

%%%%%%%%%%%%%%%%%%%%%%%%%%%%%%%%%%%%%%%%%%%%%%%%%%%%%%%
Let $\Sigma_x^0(X)$ denote the set of intrinsic directions $[\sigma]\in\Sigma_x(Y)$ of Lipschitz curves $\sigma:[0,\e)\to X$
starting from $x$ such that the direction $[\sigma]$ is uniquely determined.
% from $\sigma$ in the sense 
%that for any sequence $t_i\to 0$, $\uparrow_x^{\sigma(t_i)}$ converges to $[\sigma]$.
From Lemma \ref{lem:v}, we immediately have the following.

\begin{prop} \label{prop:closure}
 $\Sigma_x(X)$ coincides with the closure of $\Sigma_x^0(X)$ in $\Sigma_x(Y)$.
\end{prop}

The space of direction $\Sigma_x(X)$ was originally defined in an extrinsic way (see Subsection\ref{ssec:Alex}).
Proposition \ref{prop:closure} shows that it coincides with the one defined in an intrinsic way.

%%%%%%%%%%%%%%%%%%%%%%%%%%%%%%%%%%%%%%%%%%%%%%%%%%%%%%%%%%%%

%%%%%%%%%%%%%%%%%%%%%%%%%%%%%%%%%%%%%%%%%%%%%%%%%%%%%%%%%%
%

For $x\in X_1$ (resp.  $x\in X_2$),   let   $\xi_{+}\in \Sigma_x(Y)$ (resp.  $\xi_{\pm}\in\Sigma_x(Y)$) be 
the unique  (resp. the two)  direction (resp directions) of the perpendicular  (resp. perpendiculars) to $X$ at $x$. 

\begin{cor} \label{lem:susp2}
For every $x\in X$, we have the following:
%Then  we have
 \begin{enumerate}
  \item If $x\in X_1$, then 
    \[
              \Sigma_x(X) = \{ v\in\Sigma_x(Y)\,|\,\angle(\xi_{+}, v)=\pi/2\}.
    \]
  \item If $x\in X_2$, then $\Sigma_x(Y)$ is isometric to the spherical suspension $\{\xi_{\pm} \}*\Sigma_x(X)$
  \end{enumerate}
In either case, $\Sigma_x(X)$ is an Alexandrov space with curvature $\ge 1$ of dimension equal to  $\dim Y-2$.
\end{cor}

%\begin{rem}
%The suspension structure in Corollary \ref{lem:susp2} $(2)$ also follows from
%the proof of Lemma \ref{prop:preimage}
%\end{rem}

\begin{proof}
$(1)$ is a direct consequence of Corollary \ref{cor:susp1} and  Lemma\ref{lem:v}.
$(2)$  is a direct consequence of Lemma \ref{prop:preimage}, Corollary \ref{cor:susp1} and  Lemma\ref{lem:v}.
In an Alexandrov space $\Sigma$ with curvature $\ge 1$, for any $\xi\in\Sigma$, the set 
$\{ v\in\Sigma\,|\, |v,\xi|\ge \pi/2\}$ is convex, which implies the last conclusion. 
%
%
%let $\Sigma_{\pm}$ denote the union of all minimal geodesics joining $\xi_{\pm}$ to
%the elements of $\Sigma_x(X)$.
%We see that
%$\Sigma_x(Y)$ is the union of $\Sigma_{+}$ and $\Sigma_{-}$ glued along $\Sigma_x(X)$.
%Therefore $\angle(\xi_{+},\xi_{-})=\pi$ and $\Sigma_x(Y)$ is isometric to 
%the spherical suspension $\{ \xi_{\pm}\}*\Sigma_x(X)$.
%
%It turns out that  $\Sigma_x(X)$ is locally convex in $\Sigma_x(Y)$ in either case.
%Therefore $\Sigma_x(X)$ is an Alexandrov space with curvature $\ge 1$ of dimension
%$=\dim\Sigma_x(Y)-1$.
%$(3)$ also follows in a way similar to the proof of Lemma \ref{lem:susp1}.
\end{proof}

%%%%%%%%%%%%%%%%%%%%%%%%%%%%%%%%%%%%%%%%%%%%%%%%%%%%
Our next purpose is to show the following.

\begin{prop} \label{prop:tang-cone}
%Assume  $x\in X_1$. 
Under the convergence  $\lim_{\delta\to 0}(\frac{1}{\delta} Y, x) =(T_x(Y),o_x)$,
$%\displaystyle{
 (\frac{1}{\delta}X, x)$ converges to the Euclidean cone $K(\Sigma_x(X)),o_x)$
as $\delta\to 0$.
\end{prop} 

We set 
\[
       T_x(X) = K(\Sigma_x(X))
\]
and call it the {\it tangent cone of $X$ at $x$}.

For the proof of Proposition \ref{prop:tang-cone}, we need three lemmas.

%%%%%%%%%%%%% Preliminaries for the proof of tangent cone lemma %%%%%%%%%%%%%%%%%%%%%%%%%%%%%%%%%%%%%%%%%%%
%%%% prelim 1 %%%%%%%%%%%%%%%%%%%%% 

\begin{lem} \label{lem:tangX}
For every minimal geodesic $\gamma:[0,\ell]\to Y$ joining any $x\in X$ and $y\in Y$, 
the curve $\sigma(t):=\pi(\gamma(t))$ has a unique direction at $t=0$, and hence defines 
an intrinsic direction $[\sigma]\in\Sigma_x(X)$. 
\end{lem}

\begin{proof}
 By Lemma \ref{lem:susp1}, it suffices to consider only the case $\xi:=\gamma'(0)\in\Sigma_x(X)$.
Suppose that for sequences $s_j\to 0$ and $t_j\to 0$  we have limits: 
\[
    \eta := \lim_{j\to\infty} \uparrow_x^{\sigma(s_j)},  \quad  \zeta := \lim_{j\to\infty} \uparrow_x^{\sigma(t_j)}.% \eta \neq \zeta.
\]
Take $\xi_i\in\Sigma_x(Y)\setminus \Sigma_x(X)$ such that $\xi_i\to\xi$ and geodesic $\gamma_i$ in the direction $\xi_i$
is defined. Set $\sigma_i=\pi\circ\gamma_i$.
By Lemma \ref{lem:susp1}, $\sigma_i$ defines an intrinsic direction $[\sigma_i]\in\Sigma_x(X)$, 
and passing to a subsequence we may assume that 
\[
    \angle(\xi_+,\xi_i) + \angle(\xi_i, [\sigma_i]) = \angle(\xi, [\sigma_i]) = \pi/2,
\]
for some perpendicular direction $\xi_+$ at $x$. It follows that $\angle(\xi_i, [\sigma_i]) \to 0$
and $[\sigma_i]\to \xi$.
From　$\xi_i\to \xi$, we have  
\[
     |\gamma_i(t), \gamma(t)|<o_i t,
\]
where $\lim o_i = 0$. Since $\pi$ is $1$-Lipschitz, it follows that  
\begin{align}
         |\sigma_i(t), \sigma(t)| <o_i t,  \label{eq:sigma(i)}
\end{align}
which implies that 
\begin{equation}
      |\sigma_i(s_j), \sigma(s_j)|<o_i s_j,  \quad  |\sigma_i(t_j), \sigma(t_j)|<o_i t_j.  \label{eq:sigma(i2)}
\end{equation}
Passing to a subsequence, we may assume that there are limits:
\begin{align*}
   &\alpha=\lim_{j\to\infty} \frac{|x, \sigma(s_j)|}{s_j}, \quad  \alpha_i=\lim_{j\to\infty} \frac{|x, \sigma_i(s_j)|}{s_j},  \\
   &\beta=\lim_{j\to\infty} \frac{|x, \sigma(t_j)|}{t_j}, \quad  \beta_i=\lim_{j\to\infty} \frac{|x, \sigma_i(t_j)|}{t_j}.
\end{align*}
\eqref{eq:sigma(i)} implies $\alpha_i\to \alpha$ and $\beta_i\to\beta$. 
On the other hand, since $\xi\in\Sigma_x(X)$,  \eqref{eqrtatio-lim}  
shows $\alpha_j\to 1$ and $\beta_j\to 1$, Thus we have $\alpha=\beta=1$.
Since \eqref{eq:sigma(i2)} implies 
\begin{equation}
      |\alpha_i[\sigma_i], \alpha\eta|\le o_i, \quad      |\beta_i[\sigma_i], \beta\zeta|\le o_i,    \label{eq:alpha}
\end{equation}
we conclude that 
\[
           |[\sigma_i], \eta|\le o_i, \quad      |[\sigma_i], \zeta|\le o_i,  
\]
and hence the uniqueness $\eta=\zeta=\xi$.
This completes the proof.
\end{proof}

%%%%%%%%  prelim 2 %%%%%%%%%%%%%%%%%%%%%%%%%%%%%%%%%%%%%%%%
\begin{lem} \label{lem:deviation}
For every $x, y\in X$ and every minimal geodesic 
$\mu:[0,\ell]\to Y$ joining them, let $\sigma=\pi\circ\mu$
and set $\rho(t)=|\mu(t), X|$. Then we have 
\begin{enumerate}
 \item $\max \rho\le O(|x,y|^2);$
 \item $\angle(\mu'0), [\sigma])\le O(|x,y|);$
 \item$\displaystyle {\left| \frac{L(\sigma)}{L(\mu)} - 1\right|< O(|x,y|^2)}$.
\end{enumerate}
\end{lem} 

\begin{proof} (1) \,
Let $\rho(s^*)=\max \rho$ and take $0\le a<b\le\ell$ such that 
$s^*\in (a,b)$, $\rho >0$ on $(a,b)$ and $\rho(a)=\rho(b)=0$.
Then $\tilde\mu =\eta^{-1}(\mu|_{[a,b]})$ and $\tilde\sigma=\tilde\pi(\tilde\mu)$ are 
well defined.
By \eqref{eq:level-deviat}, we have
\begin{align*}
   \rho(s^*)=|\tilde\mu(s^*), C_0|\le O(|\tilde\mu(a), \tilde\mu(b)|^2)\le  O(|\mu(a), \mu(b)|^2)\le O(|x,y|^2).
\end{align*}

(2)\,
We may assume $\mu'(0)\neq [\sigma]$.
Take the smallest $s_1\in (0,\ell]$ satisfying $\mu(s_1)\in X$.
Note that $\tilde\mu=\eta^{-1}(\mu|_{0,s_1]})$ is well defined 
and  $|x, \mu(s_1)|=|\tilde\mu(0), \tilde\mu(s_1)|$.
For $\rho(s)=|\mu(s), X| = |\tilde\mu(s), C_0|$, $0\le s\le s_1$,  
Lemmas \ref{lem:susp1} and \ref{lem:angle} imply 
\begin{align*}
   \rho'(0) &= \sin\angle (\mu'(0), [\sigma]) \\
              & \le C|\tilde\mu(0), \tilde\mu(s_1)| = C|x,\mu(s_1)|\le C|x,y|.
\end{align*}

(3)\,
Take at most countable disjoint open intervals $(a_i,b_i)$ of $[0,\ell]$ such that 
\begin{itemize}%enumerate}
 \item $\mu(a_i), \mu(b_i)\in X$ and $\mu((a_i,b_i))\subset Y\setminus X;$
 \item $\mu(s)\in X$ for all $s\in J:= [0,\ell]\setminus\cup (a_i,b_i)$.
\end{itemize}%enumerate}
Set 
\[
     \mu_i=\mu|_{[a_i,b_i]}, \,\,  \sigma_i=\sigma|_{[a_i,b_i]},  \,\, \tilde\mu_i=\eta^{-1}(\mu_i),\,\, \tilde\sigma_i=\eta^{-1}(\sigma_i).
\]
Let $L(J)$ denote the measure of $J$.
Since $L(\tilde\mu_i)=L(\mu_i)$, $L(\tilde\sigma_i)\ge L(\sigma_i)$ and $\mu(J)=\sigma(J)$, 
Lemma \ref{lem:ext/int} implies that 
\begin{align*}
   L(\sigma) &= \sum \,L(\sigma_i) + L(J) \\
                & \le \sum \,L(\tilde\sigma_i) + L(J)\\
 & \le \sum (1+O(|\tilde\sigma_i(a_i), \tilde\sigma_i(b_i)|^2) L(\tilde\mu_i)+L(J)\\
 &=   \sum (1+O(|\mu(a_i), \mu(b_i)|^2)  L(\mu_i)+  L(J)\\
 &\le (1+O(|x,y|^2)  L(\mu). 
\end{align*}
%and (3) are immediate from \eqref{eq:deviat} and   \eqref{eq:deviat2}.
%By Lemma \ref{lem:susp1},  $\rho'(0)=\sin\angle(\gamma'(0),[\sigma]) < O(|x,y|)$,
%from which (2) follows. {\bf The proof is not complete yet !}
\end{proof}

%%%%%%%%%%%%%%%%%%%%%%  prelim 3 %%%%%%%%%%%%%%%%%%%%%%%%%%%%%%%%%%%%

\begin{lem} \label{lem:limit-ray}
For any $v=[\sigma]\in\Sigma_x^0(X)$, if we consider the arc-length parameter of $\sigma$, 
$\sigma(\delta t)$ converges to the geodesic ray $\sigma_{\infty}(t)$ in $T_x(Y)$ from the origin
$o_x$ in the direction $v$ as $\delta\to 0$ under the convergence $(\frac{1}{\delta} Y, x) \to T_x(Y), o_x)$.
\end{lem}

\begin{proof} Since $\sigma$ determines the unique direction $v$ we have  for any $0<R_1<R_2$, 
\[
         \lim_{\delta\to 0} \uparrow_x^{\sigma(\delta R_1)}  =  \lim_{\delta\to 0} \uparrow_x^{\sigma(\delta R_2)}.
\]
This implies that the image $\sigma_{\infty}([0,\infty))$ coincides with the ray in the direction $v$.
\end{proof}

\begin{proof}[Proof of Proposition \ref{prop:tang-cone}]
We have to show that  the Gromov-Hausdorff distance between $(\frac{1}{\delta} B(x, \delta R;X),x)$ and
$(B(o_x, R; K(\Sigma_x(X)), o_x)$ converges to zero as $\delta\to 0$ for any fixed $R>0$.
For every small  $\e>0$ take an $\e$-dense subset $\{ [\sigma_i] \}_{i=1}^I$ of $\Sigma_x^0(X)$,
and put $K:=[R/\e]+1$.
Taking small enough $\delta$, we may assume that $\sigma_i$ are defined on $[0, \delta R]$ and that 
$\sigma_i$ can be written as $\sigma_i=\pi(\gamma_i)$, where $\gamma_i$ is a minimal geodesic in $Y$
joining $x$ to $\sigma_i(\delta R)$.
Let us  consider the following sets:
\begin{align*}
     N_{\infty}^{\e} &:=\left\{      \frac{kR}{K} [\sigma]\, |\, 1\le i\le I, 0\le k\le K\right\}\subset B(o_x, R; K(\Sigma_x(X))), \\
    N^{\e} & :=  \left\{  \sigma_i\left( \frac{\delta kR}{K}\right) \, |\, 1\le i\le I, 0\le k\le K\right\} 
                                   \subset  \frac{1}{\delta}B(x, \delta R;X)
\end{align*}
First we show that both $N_{\infty}^{\e}$ and $ N^{\e}$ are $(\tau(R|\e)+\tau(R|\delta))$-dense.
 For simplicity, set
\[
       v_{i,k}= \frac{kR}{K} [\sigma], \qquad  x_{i,k}= \sigma_i\left( \frac{\delta kR}{K}\right).
\]
For every $y\in B(x, \delta R;X)$, let $\gamma:[0,\ell]\to Y$ be a minimal geodesic joining $x$ to $y$, 
and let $\sigma:=\pi(\gamma)$.
Choose $i$ and $k$ with $\angle([\sigma], [\sigma_i])<\e$ and $|\ell - kR/K|<R/K<\e$.
Since Lemma \ref{lem:deviation} implies that
\[
    \angle(\gamma'(0), \gamma_i'(0)) \le  \angle(\gamma'(0), [\sigma]) + \angle([\sigma],\sigma_i]) + \angle([\sigma_i], \gamma_i'(0))
    \le \e + 2\tau(\delta).
\]
we obtain 
\[
       |\gamma(\ell), \gamma_i(\ell)|\le C\ell (\e + 2\tau(\delta)).
\]
It follows from Lemma \ref{lem:deviation} that 
\begin{align*}
     |y, x_{i,k}| &\le |\gamma(\ell), \gamma_i(\ell)| + |\gamma_i(\ell), \sigma_i(\ell)| +  |\sigma_i(\ell), x_{i,k}| \\
                 & \le c\ell (\e + 2\tau(\delta)) + (\delta R)^2 + C\delta R/K\\
                 & \le \delta(\tau(R|\e) +\tau(R|\delta)).
\end{align*}
Thus $N^{\e}$ is $(\tau(R|\e) +\tau(R|\delta))$-dense in $\frac{1}{\delta}B(x, \delta R)$.

For every $v\in B(o_x, R; K(\Sigma_x(X)))$, take $i$ and $k$ satisfying
$ \angle(v, [\sigma_i])<\e$ and $| |v|- kR/K|<R/K<\e$.
Then we have 
\[
        |v, v_{i,k}| \le R/K + \e R  = \tau(R|\e).
\]
Hence  $N^{\e}_{\infty}$ is $\tau(R|\e)$-dense in $B(o_x, R; K(\Sigma_x(X)))$.

Finally define $f:N^{\e}_{\infty}\to N^{\e}$ by $f(v_{i,k}) = x_{i,k}$.
For simplicity put
\[
  w_{i,k}= \frac{kR}{K}\gamma_i'(0), \qquad  y_{i,k}= \gamma_i\left( \frac{\delta kR}{K}\right).
\] 
By Lemma \ref{lem:deviation}, we then have 
\begin{align*}
     & ||x_{i,k}, x_{j,\ell}| -   |y_{i,k}, y_{j,\ell}||  < 2C(\delta R)^2  \\
     & ||y_{i,k}, y_{j,\ell}| -   |w_{i,k}, w_{j,\ell}||  < \tau(R|\delta) \\
     & |||w_{i,k}, w_{j,\ell}| - |v_{i,k}, v_{j,\ell}||  < \tau(R|\delta),
\end{align*}
which implies that $f$ is $\tau(R|\delta)$-approximation.
In this way, we conclude that 
$d_{GH}((\frac{1}{\delta} B(x, \delta R),x), (B(o_x, R; T_x(X)))) <\tau(R|\e)+\tau(R|\delta)$.
This completes the proof of  Proposition \ref{prop:tang-cone}.
\end{proof}

%%%%%%%%%%%%%%  distance-ratio %%%%%%%%%%%%%%%%%%%%%%%%%%%%%%%%%%%%%%%%%

\begin{lem} \label{lem:dist-ratio}
Fix any $x\in X$ and take $p\in C_0$ with $\eta_0(p)=x$. 
Then for every $y\in X$, there is a point $q\in\eta_0^{-1}(y)$
such that 
\[
      \left| \frac{|x,y|_Y}{|p,q|_C}-1 \right| < \tau_x(|x,y|_Y),
\]
where $ \tau_x(t)$ is a function depending on $x$ with $\lim_{t\to 0}\tau_x(t)=0$.
\end{lem} 

\begin{proof}
Suppose the lemma does not hold. Then since $\eta_0$ is $1$-Lipschitz,  
we have a sequence $y_i\in X$ with $\lim y_i=x$ such that 
for every $q_i\in\eta_0^{-1}(y_i)$, 
\begin{align}
   \frac{|x,y_i|_Y}{|p,q_i|_C} < 1- \e,   \label{eq:sdist-ratio }   
\end{align}
for some $\e>0$ independent of $i$.

%%%%%%%%%%%%%%%%%%%%%%%%%

We proceed as in the proof of Lemma \ref{lem:v}.
Let $\mu_i:[0, s_i]\to Y$ be a minimal geodesic from $x$ to $y_i$,
and take a perpendicular  $\gamma_{y_i}$  to $X$ at $y_i$.
Let $\gamma_i:[0,t_i]\to Y$ be a minimal geodesic from $x$ to $\gamma_{y_i}(s_i)$, and set 
\[
         \sigma_i(t):=\pi(\gamma_i(t)), \,\, \tilde\gamma_i=\eta^{-1}(\gamma_i), \,\,
          \tilde\sigma_i=\tilde\pi(\tilde\gamma_i).
\]
Let $q_i:= \tilde\sigma(t_i)$.
Under the  convergences
\[
  \left(\frac{1}{t_i}Y, x\right)\to \left(T_x(Y), o_x\right),  \,\,\left(\frac{1}{t_i}C, p\right)\to\left(T_p(C), o_p\right)
\]
%Then $x_i$ converges to $v\in\Sigma_x(X)\subset T_x(Y)$ under the above convergence.
passing to a subsequence if necessarily, we may assume that 
the triplet $(\mu_i(t_is), \gamma_i(t_is), \sigma_i(t_is))$ (resp.  the pair $(\tilde\gamma_i(t_is), \tilde\sigma_i(t_is))$  
converges to a triplet  
$(\mu_{\infty}(s), \gamma_{\infty}(s), \sigma_{\infty}(s))$ 
(resp. a double $(\tilde\gamma_{\infty}(s), \tilde\sigma_{\infty}(s))$.
%minimal geodesic $\mu_{\infty}(s)$  and  a Lipschitz curve $\sigma_{\infty}(s)$.
%
From \eqref{eq:base-eq}, we see that 
\[
    \lim \frac{|x,y_i|_Y}{t_i} =|o_x,\sigma_{\infty}(1)|=|o_p,\tilde\sigma_{\infty}(1)|=\lim\frac{|p,q_i|_C}{t_i},
\]
%
%$\sigma_{\infty}$ is a shortest curve in the direction $\mu_{\infty}'(0)$
%and that $L(\sigma_{\infty})=L(\tilde\sigma_{\infty})$.
%In the following sublemma, we show that $L(\sigma_i)/t_i$ converges to $L(\sigma_{\infty})=1$,
which yields a contradiction to the hypothesis \eqref{eq:sdist-ratio }:
\[
      \lim_{i\to\infty}\, \frac{|x,y_i|_Y}{|p,q_i|_C}  =1.
\]
\end{proof}

%%%%%%%%%%%%%%%%%%%%%%%%%%%%%%%%%%%%%%%%%%%%%%%%%%%

For any $p\in C_0$, by Lemma \ref{lem:susp1}, as $t\to 0$,
$\eta_0:(\frac{1}{t}C_0,p)\to (\frac{1}{t}X,x)$ converges to a $1$-Lipschitz map
$(d\eta_0)_p:T_p(C_0)\to T_x(X)$, which is called the {\it differential} of $\eta_0$ at $p$.

 Lemma \ref{lem:susp1} immediately implies the following.

\begin{prop} \label{lem:length}
For every $p\in C_0$, the differential $d\eta_0:T_p(C_0)\to T_x(X)$ satisfies
\[
                |d\eta_0(\tilde v)| = |\tilde v|.
\]
for every $\tilde v\in T_p(C_0)$.
In particular,  $\eta_0:C_0\to X$ preserves the length of Lipschitz curves in $C_0$.
\end{prop}

By Proposition \ref{lem:length}, $d\eta_0$ provides a surjective $1$-Lipschitz map
$d\eta_0:\Sigma_p(C_0)\to\Sigma_x(X)$.

\begin{rem}  \upshape
By Lemma \ref{lem:susp2},    $x\in X_2$ is a regular point of $Y$ if and only if
the tangent cone $T_x(X)$ is isometric to $\mathbb R^{m-1}$, where $m=\dim Y$. 
From this reason, in that case we call  $x$ a {\it regular point} of $X$, and set $X^{reg}:=X\cap Y^{reg}$. 
Later we show that every $x\in X_1$ is a {\it singular point} of $X$ unless $X=X_1$ 
(see Corollary \ref{cor:1/2}).
\end{rem}

\begin{prop} \label{prop:isometry}
For every $p\in C_0^2$, we have
\begin{enumerate}
   \item the differential $d\eta_p$ provides an isometry $d\eta_p : T_p(C)\to T_x^{+}(Y)$ which preserves
      the half suspension structures of both $\Sigma_p(C)=\{\tilde\xi_{+}\}*\Sigma_p(C_0)$ and
      $\Sigma_x^{+}(Y):=\{ \xi_{+}\}*\Sigma_x(X)$, where $T_x^{+}(Y)=T_x(X)\times\mathbb R_{+};$
   \item $p\in C_0^{reg}$ if and only if $x\in X^{reg}$. In this case, $(d\eta_0)_p : T_p(C_0)\to T_x(X)$
      is a linear isometry.
 \end{enumerate}
\end{prop}

\begin{proof}
(1)\,
For every $\tilde v_1, \tilde v_2\in\Sigma_p(C_0)$, put $v_i:=d\eta_0(\tilde v_i)$.
We show that $\angle(\tilde v_1, \tilde v_2)=\angle(v_1, v_2)$.
Let $\tilde\xi_i$ (resp.  $\xi_i$) be the midpoint of the geodesic joining $\tilde\xi_{+}$ to $\tilde v_i$
(resp. $\xi_{+}$ to $v_i$). Note that
$d\eta(\tilde\xi_i)=\xi_i$.
We may assume that there are geodesics $\tilde\gamma_i(t)$ with $\tilde\gamma_i'(0)=\tilde\xi_i$,  and
set $\gamma_i(t):=\eta(\tilde\gamma_i(t))$.
Since $T_x(Y)=T_x(X)\times\mathbb R$, any minimal geodesic joining $\tilde\gamma_1(t)$ and
$\tilde\gamma_2(t)$ does not meet $X$ for any small $t>0$.
It follows from the fact that $\eta:C\setminus C_0\to Y\setminus X$ is locally isometric that
\[
         |\tilde\gamma_1(t),\tilde\gamma_2(t)|=  |\gamma_1(t), \gamma_2(t)|,
\]
which implies that $\angle(\tilde\xi_1,\tilde\xi_2) = \angle(\xi_1,\xi_2)$. From the suspension structures,
we conclude that $\angle(\tilde v_1, \tilde v_2)=\angle(v_1, v_2)$.

$(2)$ is an immediate consequence of $(1)$.
\end{proof}

%%%%%%%%%%%%%%%%%%%%%%%%%%%%%%%%%%%%%%%%%%%%%%%%%%%%%%%%%%%%%%%%%%%%%%%%%%%%%%%%%%%%%%%%%%%%%%%%
\subsection{Gluing maps} \label{ssec:gluing}

Using the results of the last subsection, we study the metric properties of 
the gluing map.

From Lemma \ref{prop:preimage}, we can define a map $f:C_0\to C_0$ as  follows: For an arbitrary point $p\in C_0$, 
let $f(p):=q$  if $\{p,q\}=\eta_0^{-1}(\eta_0(p))$, where  $q$ may be equal to $p$ if $\eta_0(p)\in X_1$. 
Note that $f$ is an involutive map, i.e.,  $f^2=id$. Moreover

\begin{lem} \label{lem:contin}
 $f:C_0\to C_0$  is a homeomorphism.
\end{lem}

\begin{proof}%[Proof of Lemma\ref{lem:contin}]
Since $f$ is involutive, it suffices to prove that $f$ is continuous.
For a sequence  $p_i$ converging  to a point $p$ in $C_0$, we show that 
$f(p_i)\to f(p)$.
Set $x=\eta_0(p), x_i=\eta_0(p_i)$. 

${\mathrm Case\,1)}$\, \, $x\in X_1$.
In this case, $f(p)=p$.
If $x_i\in X_1$, then $f(p_i)=p_i$, and we have nothing to do.
Suppose $x_i\in X_2$.  Let $\gamma_{x_i}^{\pm}$ be the two perpendiculars to $X$ at $x_i$.
Letting $s_i=|x,x_i|$, consider minimal geodesics $\gamma_i^{\pm}$ joining $x$ to $\gamma_{x_i}^{\pm}(s_i)$.
If we set $\tilde\sigma_i^{\pm}:=\tilde\pi\circ\tilde\gamma_i^{\pm}$, where $\tilde\gamma_i^{\pm}=\eta^{-1}(\gamma_i^{\pm})$,
then $\tilde\sigma_i^{\pm}$ are minimal geodesics joining $p$ to $p_i$ and $f(p)$ to $f(p_i)$ respectively.
Lemma \ref{lem:dist-ratio} then shows that 
\begin{align}
      \left| \frac{|x, x_i|}{|p,p_i|}-1 \right| < \tau_x(|x,x_i|), \,\,\,  \left| \frac{|x, x_i|}{|f(p), f(p_i)|}-1 \right| < \tau_x(|x,x_i|), \label{eq:two-ratio}
\end{align}
which implies  $f(p_i)\to f(p)$.

${\mathrm Case\,2)}$\, \, $x\in X_2$.
In this case, $f(p)\neq p$.
Let $\gamma_{x}^{\pm}$ be the two perpendiculars to $X$ at $x$.
By Lemma \ref{lem:susp2},  we have
\begin{equation}
    \angle \gamma_{x}^{+}(s_0)x_i \gamma_{x}^{-}(s_0) \ge \tilde\angle \gamma_x^{+}(s_0)x_i\gamma_{x}^{-}(s_0) 
                                           >  \pi-\tau(s_0), \label{eq:wide}
\end{equation}
If $x_i\in X_1$, then both  $\angle(\xi_i^+, \uparrow_{x_i}^{\gamma_x^+(s_0)})$ and 
$\angle(\xi_i^+, \uparrow_{x_i}^{\gamma_x^{-}(s_0)})$ become small, yielding a contradiction to 
\eqref{eq:wide}. Thus we have $x_i\in X_2$.
Then in a way similar to ${\mathrm Case\,1)}$,  we have the formula \eqref{eq:two-ratio},
which implies  $f(p_i)\to f(p)$.
\end{proof}

\begin{cor} \label{cor:cover}
$\eta_0|_{C_0^2}:C_0^2\to X_2$ is a double covering space and $X_2$ is open in $X$.
\end{cor}

\begin{proof}
For $x\in X_2$ set  $\eta_0^{-1}(x)=\{ p_1,p_2\}$, 
and take an open neighborhood $D_1$ of $p_1$ in $C_0$ such that 
$D_1\cap f(D_1)$ is empty. We set $D_2=f(D_1)$. We show that $E:=\eta_0(D_i)$ is open in $X$.
Suppose that $E$ is not open, and take $y\in E$ for which there are $y_i\in X\setminus E$
converging to $y$. 
%%%%
Let $\{ q_1, q_2\} := \eta_0^{-1}(y)$ with $q_k\in D_k$, $k=1,2$.
Applyinjg Lemma \ref{lem:dist-ratio} to $y_i\to y$ and $q_k\in \eta_0^{-1}(y)$, 
we have $q_{k,i} \in \eta_0^{-1}(y_i)$ such that 
\[
      \left| \frac{|q_k, q_{k,i}|}{|y,y_i|}-1 \right| < \tau_y(|y,y_i|), \,\, k=1,2.
\]
This implies that $q_{k,i}\in D_k$ and $y_i\in E$ for large $i$.
%
%a contradiction.
%
%
%Choose any $q_i\in\eta_0^{-1}(y_i)$. Passing to a subsequence,
%we may assume that $q_i$ converges to a point $q$.
%It turns out that $\eta_0^{-1}(y)$ contains at least three points $q_1,q_2$ and $q$,
%where $q_i\in D_i$, $q\notin D_1\cup D_2$. 
Since this is a contradiction, $E$ is open.  
Similarly one can show that 
each restriction $\eta_0|_{D_k}:D_k\to E$ is an open map, and hence is a homeomorphism.
\end{proof}

\begin{cor}\label{prop:XY}
$Y$ and $X$ are homeomorphic to the quotient spaces $C_0\times_{\phi} [0,t_0]/f$ and $C_0/f$ 
respectively, where $(x,0)$ and $(f(x),0)$ are identified for every $x\in C_0$.
\end{cor}

\begin{cor}\label{cor:cpntC0}
If the inradius of  $M_i\in\ca M(n,\kappa,\lambda,d)$ converges to zero,  then the number of components of 
$\pa M_i$  is at most two for large enough $i$.
\end{cor}

\begin{proof}
Since $f$ is an involutive homeomorphism, $f$ gives a transposition of two components of $C_0$.
The conclusion is immediate from the connectedness of $X$.
\end{proof}

\begin{rem}
In Theorem \ref{thm:two}, we remove the diameter bound to get the diameter
 free result.  
\end{rem}

\begin{lem} \label{lem:loc-isometry}
$\eta_0|_{C_0^2}: (C_0^2)^{\rm int} \to X_2^{\rm int}$ is a local isometry.
\end{lem}

\begin{proof}
Since $\eta_0|_{C_0^2}:C_0^2\to X_2$ is a covering by Corollary \ref{cor:cover}, we can find 
relatively compact open subsets $D$ and $E$ of  $C_0^2$ and $X^2$ respectively such that
$\eta_0:D\to E$ is a homeomorphism. We must show that  $\eta_0:D\to E$ is an isometry
with respect to the interior distances of $C_0$ and $X$ respectively.
Since $\eta_0$ is $1$-Lipschitz, it suffices to show
that $g:=\eta_0^{-1}:E\to D$ is $1$-Lipschitz.
We may assume that $D$ is small enough so as to satisfy that
for every $x,y\in E$, there is a minimal geodesic $\gamma:[0,1]\to X_2$ joining $x$ to $y$.
We do not know if $g\circ\gamma$ is a Lipschitz curve yet.
However by Proposition \ref{lem:length}, $g\circ\gamma$ has the  speed $v_{g\circ\gamma}(t)$ (see \cite{BBI})
\[
        v_{g\circ\gamma}(t) =\lim_{\epsilon\to 0} \frac{|g\circ\gamma(t),g\circ\gamma(t+\epsilon)|}{|\epsilon|},
\]
which is equal to the speed $v_{\gamma}(t)$ of $\gamma$, and therefore
\[
     |x,y|=L(\gamma)=\int_0^1  v_{g\circ\gamma}(t) dt=L(g\circ\gamma) \ge |g(x),g(y)|.
\]
This completes the proof.
\end{proof}

%Recall that $C_0^i:= \{ p\in C_0\,|\, \eta_0(p)\in X_i\}$, $i=1,2$.

\begin{lem} \label{lem:X1}
If $X_1$ has non-empty interior in $X$, then $X=X_1$ and
$\eta_0: (C_0)^{\rm int} \to X^{\rm int}$ is an isometry.
\end{lem}
\begin{proof}
If the interior $U$ of $X_1$ is non-empty, then $V:=\eta_0^{-1}(U)\subset C_0^1$ is open
in $C_0$. From the non-branching property of geodesics in Alexandrov spaces,
we have $V=C_0$ and $X=X_1$.
An argument similar to the proof of Lemma \ref{lem:loc-isometry}
shows that $\eta_0:(C_0)^{\rm int}\to X^{\rm int}$ is an isometry.
\end{proof}

\begin{prop} \label{lem:isometry}
 $f:(C_0)^{\rm int}\to (C_0)^{\rm int}$ is an isometry.
\end{prop}

\begin{proof}
For $x\in X_2$ with $\eta_0^{-1}(x)=\{ q_1,q_2\}$, by lemma \ref{lem:loc-isometry}, we can take disjoint open sets 
$q_i\in D_i$, $i=1,2$, and $E$
such that $\eta_0^i=\eta_0|_{D_i}:D_i\to E$ are isometry.
Thus $f|_{D_1}=(\eta_0^2)^{-1}\circ\eta_0^1:D_1\to D_2$
is an isometry with respect to the interior distances.
Note that $f$ is identity on $C_0^1$, and by Lemma \ref{lem:loc-isometry}, $f:(C_0^2)^{\rm int} \to (C_0^2)^{\rm int}$ is 
a locally isometry.
For every $p_1,p_2\in C_0$ we show that $|f(p_1),f(p_2)|=|p_1,p_2|$. This is obvious if $p_1,p_2\in C_0^1$.
Let $\gamma:[0,1]\to C_0$ be a minimal geodesic joining $p_1$ to $p_2$.
If  $p_1,p_2\in C_0^2$, applying Lemma \ref{lem:loc-isometry}, we may assume that $\gamma$ meets $C_0^1$.
Let $t_0\in (0,1)$ be the smallest  parameter with $\gamma(t_0)\in C_0^1$.
By Lemma \ref{lem:loc-isometry}, we have $|f(p_1), f(\gamma(t_0))|=|p_1, \gamma(t_0)|$.
Therefore the non-branching property of geodesics in Alexandrov space implies that
$\gamma\cap C_0^1$ consists of only the single point $\gamma(t_0)$, and therefore
we also have  $|f(p_2), f(\gamma(t_0))|=|p_2, \gamma(t_0)|$.
It follows that
\begin{align*}
   |f(p_1), f(p_2)| &\le |f(p_1),f(\gamma(t_0))|  + |f(\gamma(t_0)). f(p_2)| \\
                      &\le  |p_1,\gamma(t_0)|+|\gamma(t_0),p_2| =|p_1,p_2|.
\end{align*}
Repeating this, we also have $|p_1,p_2|\le |f(p_1),f(p_2)|$, and $ |f(p_1),f(p_2)|=|p_1,p_2|$.
The case of $p_1\in C_0^1$ and $p_2\in C_0^2$ is similar, and hence is omitted.
This completes the proof.
\end{proof}

%%%%%%%%%%%%%%%%%%%%%%%%%%%%%%%%%%%%%%%%%%%%%%%%%%%%%%%%%%%%%%%%%%%%%%%%%%%%%%%%%%%%%%%%%%%%%%%%%%
\subsection{Structure theorems}

In this subsection, making use of the results on gluing maps in the last subsection,
we obtain structure results for limit spaces. 

We begin with 

\begin{lem} \label{lem:convex}
$X_2$ is convex in $X$.
\end{lem}

\begin{proof}
Suppose this is not the case. Then we have a minimal geodesic $\gamma:[0,1]\to X$
joining points $x,y\in X_2$ such that $\gamma$ is not entirely contained in $X_2$.
Let $t_1$ be the first parameter with $\gamma(t_1)\in X_1$. Set $z:=\gamma(t_1)$.
By Lemma \ref{lem:loc-isometry}, for any $p\in \eta_0^{-1}(x)$, there exists a unique geodesic
$\tilde\gamma:[0,t_1]\to C_0$
such that $\tilde\gamma(0)=p$ and $\eta_0\circ\tilde\gamma(t)=\gamma(t)$,
for every $t\in [0,t_1]$. Put $\tilde z:=\tilde\gamma(t_1)\in C_0^1$, and take
$\tilde v\in \Sigma_{\tilde z}(C_0)$ such that
$(d\eta_0)_{\tilde z}(\tilde v)=\frac{d}{dt}\gamma(t_0)\in\Sigma_z(X)$.
Let $\tilde\gamma_1:[0,t_1]\to C_0$ and   $\gamma_1:[0,t_1]\to X$ be the reversed geodesic to
$\tilde\gamma$ and $\gamma_{[0,t_1]}$:
$\tilde\gamma_1(t)=\tilde\gamma(t_0-t)$,  $\gamma_1(t)=\gamma(t_1-t)$, and set
$\tilde\gamma_2(t):=f(\tilde\gamma_1(t))$.
Since $(d\eta_0)_{\tilde z}$ preserves norm and is $1$-Lipschitz, we have
\[
\angle(\tilde v, \tilde\gamma_i'(0)) \ge \angle\left(\frac{d}{dt}\gamma(t_1), \frac{d}{dt}\gamma_1(0)\right) =\pi,
\]
for $i=1,2$. Since  $\tilde\gamma_1'(0) \neq  \tilde\gamma_2'(0)$,  this is impossible in
the  Alexandrov space $C_0$.
\end{proof}

\begin{lem} \label{lem:lift}
For every $x,y\in X$, let $\gamma:[0,1]\to X$ be a minimal geodesic joining $x$ to $y$,
and let $p\in C_0$ be such that $\eta_0(p)=x$. Then there exists a unique minimal 
geodesic $\tilde \gamma:[0,1]\to C_0$ starting from $p$ such that $\eta_0\circ\tilde\gamma=\gamma$. 

In particular, if $X_1$ is not empty, then $C_0$ is connected.
\end{lem}

\begin{proof}
From Lemmas \ref{lem:convex} and the discussion there using non-branching property of geodesics
in Alexandrov spaces, we have only the following possibilities:
\begin{enumerate}
 \item $\gamma$ is included in $X_1$ or $X_2;$
 \item only one end point of $\gamma$ is contained in $X_1$ and the other part of $\gamma$ is 
          included in $X_2$.
\end{enumerate}
The conclusion follows immediately from Lemmas \ref{lem:loc-isometry} and \ref{lem:X1}.
\end{proof}

\begin{prop} \label{prop:quotient}
 $N=X^{\rm int}$ is isometric to $C_0^{\rm int}/f$.
\end{prop}

\begin{proof}
In the case of $X=X_1$ or $X=X_2$, the conclusion follows from Lemma \ref{lem:X1} or Proposition \ref{lem:loc-isometry}
respectively. Next assume that both $X_1$ and $X_2$ are non-empty.
We set $Z:=C_0^{\rm int}/f$, which is an Alexandrov space, and decompose $Z$ as
\[
                       Z=Z_1\cup Z_2, \,\,\,   Z_i:=C_0^i/f,   \,\,i=1,2.
\]
For every $[p]\in Z_1$, $\Sigma_{[p]}(Z)$ is isometric to $\Sigma_p(C_0)/f_*$, where
$f_*: \Sigma_p(C_0)\to \Sigma_p(C_0)$ is an isometry induced by $f$.
Since $X_1$ is a proper subset of $X$, $f_*$ defines a non-trivial isometric  $\mathbb Z_2$-action
on $\Sigma_p(C_0)$.  Thus $[p]$ is a single point of $Z$: $[p]\in Z^{sing}$, and therefore $Z_1\subset Z^{sing}$.
Thus $Z^{\rm reg}\subset Z_2$. Now by Proposition \ref{prop:isometry}, there exists an
isometry $F_0:Z_2\to X_2^{\rm int}$.
Since $Z^{\reg}$ is convex in $Z$ (see \cite{Petrunin}), $F_0$ defines a $1$-Lipschitz map
$F_1:(Z^{\rm reg})^{\rm ext} \to X$ which extends to a $1$-Lipschitz map
$F:Z\to X$, where $(Z^{\rm reg})^{\rm ext}$ denotes the exterior metric of $Z^{\rm reg}$.

Conversely since $X_2$ is convex in $X$ by Lemma \ref{lem:convex}, $F_0^{-1}$ defines a
$1$-Lipschitz map
$G_1:(X_2)^{\rm ext} \to Z_2$ which extends to a $1$-Lipschitz map
$G: X\to Z$ satisfying $G\circ F=1_Z$. Therefore $X$  must be isometric to $Z$.
\end{proof}

\begin{proof}[Proof of Theorem \ref{thm:limit-alex}]
By Proposition \ref{lem:isometry}, $f:C_0^{\rm int}\to C_0^{\rm int}$ is
an involutive isometry.
By Propositions \ref{prop:intrinsic} and \ref{prop:quotient}, $N$ is isometric to $C_0^{\rm int}/f$. 
Since $C_0^{\rm int}$ is an Alexandrov space with curvature $\ge  c(\kappa,\lambda)$, so is $N$.
\end{proof}

In view of Proposition \ref{prop:fixed-ext}, Proposition \ref{prop:quotient} immediately implies 

\begin{cor}
$X_1$ is an extremal subset of $X^{\rm int}$.
\end{cor}

\begin{thm}\label{thm:singIbund}
Let $M_i\in\ca M(n,\kappa,\lambda,d)$ inradius collapse to a compact length space $N$.
Let $\tilde M_i$ Gromov-Hausdroff converge to $Y$, and $M_i^{\rm ext}$ converge to 
$X\subset Y$ under the convergence $\tilde M_i \to Y$.
Then
\begin{itemize}
  \item[(1)] $X^{\rm int}$ is isometric to $N;$ 
 \item[(2)] $Y$ is isometric to $C_0^{\rm int}\times_{\phi}[0,t_0]/(f(x),0)\sim (x,0)$, or equivalently,
   isometric to the following quotient by an isometric involution $\tilde f = (f, -{\rm id})$.
    \[
          C_0^{\rm int}\times_{\tilde\phi}[-t_0, t_0]/\tilde f,
     \]
   where $\tilde\phi(t)=\phi(|t|)$.

   In particular, $Y$ is a singular $I$-bundle over $N$, where singular fibers occur 
   exactly on $X_1$ unless $X=X_1$.
 \end{itemize}
\end{thm}

Compare Examples \ref{ex:counter2}, \ref{ex:twist} and \ref{ex:D}.

\begin{proof}[Proof of Theorem \ref{thm:singIbund}]
(1) is just Proposition \ref{prop:intrinsic}. (2) follows immediately from Propositions
\ref{prop:YNC} and \ref{prop:quotient}.
\end{proof}

\begin{rem}
Theorem \ref{thm:singIbund} can be generalized to the unbounded diameter case
(see Section \ref{sec:unbounded}).
\end{rem}

\begin{prop}\label{prop:direc}
 If $x\in X_1$, then $\Sigma_x(X)$ is isometric to the quotient space  $\Sigma_p(C_0)/f_{*}$,
     and $\Sigma_x(Y)$ is isometric to the quotient space  $\Sigma_p(C)/f_{*}$,
     where $f_{*}:\Sigma_p(C_0)\to \Sigma_p(C_0)$ is an isometry induced by $f$.
\end{prop}

\begin{proof}
Take an $f$-invariant neighborhood $U_p$ of $p$
in $C_0$, where $\eta_0(p)=x$. It is easy to check that $V_x:=\eta_0(U_p)$ is a neighborhood of $x$ isometric to $U_p/f$.
The conclusion of $(2)$ follows immediately.
\end{proof}

\begin{cor}\label{cor:1/2}
Let $\dim N=m$.
Suppose that both $X_1$ and $X_2$ is non-empty.
Then every element $x\in X_1$ satisfies that
\[
     \vol \,\Sigma_x(X) \le \frac{1}{2}\vol \,\mathbb S^{m-1}.
\]
In particular, $\dim (X_1\cap \pa X) \le m-1$ and  $\dim (X_1\cap {\rm int} X) \le m-2$.
\end{cor}
\begin{proof}
For $x\in X_1$, take $p\in C_0$ with $\eta_0(p)=x$.
Note that $C_0$ is connected by Lemma \ref{lem:lift}.
If $f_{*}:\Sigma_p(C_0) \to \Sigma_p(C_0)$ is  the identity, then the non-branching property of
geodesics in Alexandrov spaces implies that  $f$ is the identity on $C_0$.
Therefore $f_*$ must be non-trivial on $\Sigma_p(C_0)$.
The conclusion follows  since
\[
    \vol \,\Sigma_x(X) =(1/2)\vol\, \Sigma_p(C_0) \le (1/2)\vol \mathbb S^{m-1}.
\]
\end{proof}

By Corollary \ref{cor:1/2}, if every $x\in X$ satisfies that
\[
   \vol \Sigma_x(X) > (1/2)\vol \mathbb S^{m-1},
\]
then $X=X_1$ or $X=X_2$.

Next let us consider such a case. 
If $X=X_1$, then by Lemma \ref{lem:X1}, $\eta_0$ is an isometry. 
If $X=X_2$, then by Lemma \ref{lem:loc-isometry}, $\eta_0$ is a locally isometric double covering.
Therefore it is straightforward to see the following.

\begin{cor}\label{prop:IbundleY}
If $X=X_1$ or $X_2$, then $Y$ can be classified by $N$ as follows.
\begin{itemize}
\item[(1)] if $X=X_1$, then $Y$ is isometric to $N\times_\phi[0,t_0]$.
\item[(2)] if $X=X_2$,  then  either $Y$ is isometric to the gluing
     \begin{equation*}
                    N\times_{\tilde\phi} [-t_0,t_0], \label{eq:product}
     \end{equation*}
      with length metric, or else,  $Y$ is a nontrivial $I$-bundle over $N$, and is doubly covered by 
    \begin{equation*}
                    C_0^{\rm int}\times_{\tilde\phi} [-t_0,t_0], \label{eq:product2}
     \end{equation*}  
      where $\tilde\phi(t)=\phi(|t|)$.
\end{itemize}
%\end{itemize}
\end{cor}

Compare Examples \ref{ex:counter2}  and \ref{ex:twist}.

From now, we write for simplicity as $C_0:= C_0^{\rm int}$.

%%%%%%%%%%%%%%%%%%%%%%%%%%%%%%%%%%%%%%%%%%%%%%%%%%%%%%%%%%%%%%%%%%%%%%%%%%%%%%%%%%%%%%%%%%%%%%%%%%%%%
\section{Inradius collapsed manifolds with bounded diameters}\label{sec:fib}

In this section, we investigate the structure of inradius collapsed manifolds $M_i$ applying 
the structure results for limit spaces in Section \ref{sec:metric}. First we study the case of inradius collapse of
codimension one to determine the manifold structure. To carry out this,  
some additional considerations on the limit spaces are needed to determine the singularities of 
singular $I$-fibered spaces. In the second part of this section, we 
study inradius collapse to almost regular spaces.

%%%%%%%%%%%%%%%%%%%%%%%%%%%%%%%%%%%%%%%%%%%%%%%%%%%%%%%%%%%%%%%%%%%%%%%%%%%%%%%%%%%%%%%%%%%%%

\subsection{Inradius collapse of codimension one} \label{ssec:codim1}

We consider $M_i\in\ca M(n,\kappa,\lambda, d)$ inradius collapse to an $(n-1)$-dimensional 
Alexandrov space $N$. Then by Theorem \ref{thm:stability},  $M_i$ is homeomorphic to $Y$, 
and by Theorem \ref{thm:singIbund}, we have  
\[
      Y = C_0\times_{\tilde\phi}[-t_0, t_0]/ \tilde f,    \,\, N=C_0/f,
\]
where $\tilde f=(f, -{\rm id})$ is an isometric involution. 
and the singular locus of the singular $I$-bundle structure on $Y$ defined by the above form
coincides with $C_0^1$ unless $X\neq X_1$. Later in Lemma \ref{lem:partial}, we show that
$\eta_0(C_0^1)=\partial N$.

Assuming that $N$ has non-empty boundary, 
we begin with construction of singularity models of singular $I$-fibered spaces  around each 
boundary component of the limit space $N$.

By Proposition \ref{prop:collar}, 
each component $\partial_{\alpha} N$ of $\partial N$ has a collar neighborhood 
$V_{\alpha}$. Let $\varphi:V_{\alpha}\to \partial_{\alpha}N\times [0,1)$ be a homeomorphism.
Let $\pi:Y\to N$ be the projection. Then $I$-fiber structure on $\pi^{-1}\varphi^{-1}(\{ p\}\times [0,1)$
is isomorphic to the form 
\[
       R_{t_0}:=[0,1)\times [-t_0, t_0]/(0,y)\sim (0,-y),
\]
with the projection $\pi:R_{t_0}\to [0,1)$ indecued by $(x,y) \to x$.
Therefore $\pi^{-1}(V_{\alpha})$ is an $R_{t_0}$-bundle over $\partial_{\alpha}N$.
%
%If $S_{t_0}:=(-1,1)\times [-t_0, t_0]$, then $R_{t_0} =S_{t_0}/\tilde f$,
%where $\tilde f(x, y) =(-x, -y)$. 
%Let $g:R_{t_0}\to R_{t_0}$ be the isometric involution induced from $(x,y)\to (x,-y)$.

Now we define two singularity model for the singular $I$-bundle $\pi^{-1}(V_{\alpha})$:
one is the case when  $\pi^{-1}(V_{\alpha})$ is a trivial $R_{t_0}$-bundle over $\partial_{\alpha}N$,
and the other one is the case of non-trivial $R_{t_0}$-bundle.

\begin{defn} \label{def:model}
$(1)$.\,
First, set 
$$
                      \mathcal U_1(\partial_{\alpha}N):=\partial_\alpha N\times R_{t_0},
$$
and define
$\pi: \mathcal U_1(\partial_{\alpha}N)\to\partial_{\alpha}N\times [0,1)$ by $\pi(p,x,y)=(p,x)$
for $(p,x,y)\in \partial_\alpha N\times R_{t_0}$. This gives  $\mathcal U_1(\partial_{\alpha}N)$
the structure 
of a singular $I$-bundle over $\partial_{\alpha}N\times [0,1)$ whose singular locus is 
 $\partial_{\alpha}N\times 0$. We call this {\em the product singular $I$-bundle model}
around   $\partial_{\alpha}N$. 
%Let
%\[  
%   \tilde{\mathcal U}_1(\partial_{\alpha}N):=\partial_\alpha N\times S_{t_0}
%\]
%with an  isometric involution $\tilde f_1(p,x,y)=(p, -x,-y)$.
%Then $\mathcal U_1(\partial_{\alpha}N)=\tilde{\mathcal U}_1(\partial_{\alpha}N)/\tilde f_1$,
%and the fixed point set of $\tilde f_1$ coincides with $\partial N\times 0$.

$(2)$.\,
For the second model, suppose that  $\partial_{\alpha} N$ admits a double covering space
$\rho:P_{\alpha}\to \partial_{\alpha} N$ with the deck transformation $\varphi$.
Let 
$$
                        \mathcal U_2(\partial_{\alpha}N):=(P_{\alpha}\times R_{t_0})/\varPhi,
$$
where $\varPhi$ is the isometric involusion on $P_{\alpha}\times R_{t_0}$ defined by
$\varPhi=(\varphi, g)$, where $g: R_{t_0} \to R_{t_0}$ is the reflection induced from 
$(x,y)\to (x, -y)$.
Define
$\pi: \mathcal U_2(\partial_{\alpha}N)\to\partial_{\alpha}N\times [0,1)$ by 
$\pi([p,x,y)])=(\rho(p),x)$
for $(p,x,y)\in P_{\alpha}\times R_{\epsilon}$. This gives  $ \mathcal U_2(\partial_{\alpha}N)$ the structure 
of a singular $I$-bundle over $\partial_{\alpha}N\times [0,1)$ whose singular locus is 
 $\partial_{\alpha}N\times 0$. The second model is a twisted one, and is doubly covered
by the first model $\mathcal U_1(P_{\alpha})=P_{\alpha}\times R_{\epsilon}$.
 We call this the {\em twisted singular $I$-bundle model}
around  $\partial_{\alpha}N$. 
%Let 
%\[  
%   \tilde{\mathcal U}_2(\partial_{\alpha}N):=P_{\alpha} \times S_{t_0},
%\]
%and  define isometric involution on $\tilde{\mathcal U}_2(\partial_{\alpha}N)$ by
% $\tilde f_2(p,x,y)=(p, -x,-y)$.   
%Note that the reflection 
\end{defn}

\begin{ex} 
Let us consider the codimension one inradius collapse in Example \ref{ex:D}.
Recall that the limit space  $Y$ of $\tilde M_{\epsilon}$ is isometric to the form
\[
      Y = D(E)\times_{\tilde\phi}[-t_0,t_0]/(x,t)\sim (r(x),-t),
\]
where $r:D(E)\to D(E)$ denotes the canonical reflection of $D(E)$.
If  $\pi:Y\to E$ denotes the projection, then $\pi^{-1}(V)$ is isomorphic to 
the product singular $I$-bundle model around   $\partial E$, 
where $V$ is any collar neighborhood of $\partial E$.
\end{ex}

\begin{ex} \label{ex:Klein}
Let $Q_{\epsilon}$ denote the space obtained from the disjoint union of two copies of 
the completion $\bar R_{\epsilon}$ of  $R_{\epsilon}$
glued along each segment $1\times [-\epsilon,\epsilon]$ of the boundaries:
\[
            Q_{\epsilon}= \bar R_{\epsilon} \amalg_{1\times  [-\epsilon,\epsilon]}  \bar R_{\epsilon}.
\]
Let $r:Q_{\epsilon}\to Q_{\epsilon} $ be the reflection induced  from $(x,y)\to (x,-y)$.
Let $M_{\epsilon}=(\mathbb S^1(1)\times Q_{\epsilon})/(z, p) \sim (-z, r(p))$.
As $\epsilon\to 0$, $M_{\epsilon}$ inradius collapses to $\mathbb S^1(1/2)\times [0,2]$.
Let $\pi_{\epsilon}:M_{\epsilon}\to \mathbb S^1(1/2)\times [0,2]$ be the projection
induced  by $[z,(x,y)]\to (z,x)$.
Then both  $\pi_{\epsilon}^{-1}(\mathbb S^1(1/2)\times [0,1)$ and  
$\pi_{\epsilon}^{-1} (\mathbb S^1(1/2)\times (1, 2])$ are solid Klein bottle and  their $I$-fiber structures
are isomorphic to the twisted singular $I$-bundle model around respective boundary
of  $\mathbb S^1(1/2)\times [0,2]$.
\end{ex}

The following is a detailed version of Theorem \ref{thm:codim1}.
Recall that $D^2_+$ is the upper half disk on $xy$-plane, and $J:=D^2_{+}\cap \{ y=0\}$.

\begin{thm}\label{thm:codim1-detail}
Let $M_i\in\ca M(n,\kappa,\lambda, d)$ inradius collapse to an $(n-1)$-dimensional 
Alexandrov space $N$.  Then there is a  singular $I$-fiber bundle:
\[
      I  \rightarrow  M_i  \overset{\pi} \rightarrow   N
\]
whose singular locus coincides with $\partial N$, 
and $M_i$ is a gluing of $I$-bundle $N\tilde\times I$ over $N$ and $D^2_{+}$-bundle
$\pa N \tilde\times D^2_+$ over $\pa N$ , 
\[
      M_i = N\tilde\times I\cup \pa N \tilde\times D^2_+,
\]
where the gluing is done via $\pa N \tilde\times I = \pa N \tilde\times J$,
and $\tilde\times$ denotes either the product  or a twisted product.
In particular $M_i$ has the same homotopy type as $N$.

More precisely, 
\begin{enumerate}
 \item If  $N$ has no boundary, then  $M_i$ is homeomorphic to a product $N\times I$ 
          or a twisted product $N  \mathbin{\stackrel{\sim}{\times}} I;$
 \item If  $N$ has non-empty boundary, each component $\partial_{\alpha}N$ of $\partial N$
         has a neighborhood $V$ such that  $\pi^{-1}(V)$ is isomorphic to either 
         the product or the twisted singular $I$-fiber bundle around  $\partial_{\alpha}N$;
 \item If  $\pi^{-1}(V)$ is isomorphic to  the product singular $I$-fiber bundle for some component  
         $\partial_{\alpha}N$,
         then $M_i$ is homeomorphic to $D(N)\times [-1,1]/(x,t)\sim (r(x), -t)$, where $r$ is the canonical
         reflection of the double $D(N)$.
\end{enumerate}
\end{thm}

%\begin{thm}\label{thm:codim1'}
%Let $M_i\in\ca M(n,\kappa,\lambda, d)$ inradius collapse to an $(n-1)$-dimensional 
%Alexandrov space $N$.  Then there is a  singular $I$-bundle: %structure on $M_i$ over $N$:
%\[
%      I  \rightarrow  M_i  \overset{\pi} \rightarrow   N.
%\]
%More precisely, 
%\begin{enumerate}
% \item If  $N$ has no boundary, then  $M_i$ is homeomorphic to a product $N\times I$ 
%          or a twisted product $N  \mathbin{\stackrel{\sim}{\times}} I;$
% \item If  $N$ has non-empty boundary, each component $\partial_{\alpha}N$ of $\partial N$
%         has a neighborhood $V$ such that  $\pi^{-1}(V)$ is isomorphic to either 
%         $\mathcal U_1(\partial_{\alpha}N)$ or $ \mathcal U_2(\partial_{\alpha}N)$ as $I$-fibered spaces;
% \item If  $\pi^{-1}(V)$ is isomorphic to  $\mathcal U_1(\partial_{\alpha}N)$ for some component  $\partial_{\alpha}N$,
%         then $M_i$ is homeomorphic to $D(N)\times [-1,1]/(x,t)\sim (r(x), -t)$, where $r$ is the canonical
%         reflection of the double $D(N)$.
%\end{enumerate}
%\end{thm}

Recall that   
\[
      Y = C_0\times_{\tilde\phi}[-t_0, t_0]/ \tilde f,    
\]
where $\tilde f=(f, -{\rm id})$,  $C_0$ and $Y$ are the noncollapsing limit of $(\partial M_i)^{\rm int}$ and $\tilde M_i$
respectively. 
Therefore both  $C_0$ and $Y\setminus C_{t_0}$ are smoothable spaces 
in the sense of \cite{Kap}. See also Remark \ref{rem:Kos}.

Let $F\subset C_0$ denote the fixed point set of the isometry $f:C_0\to C_0$.
%Naturally $F$ can be embeded as a closed subset  $\eta_0(F)$of $N=C_0/f$.
%
%
By Proposition \ref{prop:fixed-ext} and Theorem \ref{thm:exr-prop} , 
$\eta_0(F)$ is an extremal subset of $N$ and it has a topological stratification.
%Furthermore $\eta_0(F)$ is totally quasigeodesic, and the local intricsic metric is 
%bi-Lipschitz to the ambient metric (see also [Kap]).

\begin{lem} \label{lem:partial}
$\eta_0(F)$ coincides  with $\partial N$ if $f$ is not the identity.
\end{lem}

We postpone the proof of Lemma \ref{lem:partial}  for a moment.

\begin{proof}[ Proof of Theorem \ref{thm:codim1-detail}]
$(1)$\, By Lemma \ref{lem:partial}, if $N$ has no boundary, $F$ is empty, and 
therefore either $N=N_1$ or $N=N_2$.
If $N=N_1$, then $C_0=N$ and $Y$ is homeomorphic to $N\times I$.
If $N=N_2$, then $N=C_0/f$ has no boundary, and $Y$ is homeomorphic to
either $N\times I$ or $C_0\times [-1,1]/(x,t)\sim (f(x), -t)$ which is a 
twisted $I$ bundle over $N$.

$(2)$\, 
Suppose $N$ has non-empty boundary. 
Note that 
\[
                                   N_1=\eta_0(F).
\]
By Proposition \ref{prop:collar}, 
each component $\partial_{\alpha} N$ of $\partial N$ has a collar neighborhood 
$V_{\alpha}$. Let $\varphi:V_{\alpha}\to \partial_{\alpha}N\times [0,1)$ be a homeomorphism.
Let $\pi:Y\to N$ be the projection.
By the $I$-fiber structure of $Y$, 
$\pi^{-1}(\varphi^{-1}(x\times [0,1))$ is canonically homeomorphic to 
$R_{t_0}$. In particular $\pi^{-1}(V_{\alpha})$ is an $R_{t_0}$-bundle over $\partial_{\alpha}N$.
If this bundle is trivial, $\pi^{-1}(V_{\alpha})$ is isomorphic to the product singular $I$-bundle 
structure $\mathcal U_1(\partial_{\alpha}N) = \partial_{\alpha} N\times R_{t_0}$.

Suppose that this bundle is nontrivial, 
and 
let $P_{\alpha}$ be the boundary of $\pi^{-1}(\varphi^{-1}(\partial_{\alpha} N\times \{ 1/2\}))$,
which is a double covering of $\partial N_{\alpha}$. Let $\varPhi=(\varphi, g)$,
and $\rho:P_{\alpha}\to \partial_{\alpha}N$ the projection.

\begin{lem}
 $\pi^{-1}(V_{\alpha})$ is isomorphic to the twisted singular $I$-bundle 
structure $\mathcal U_2(\partial_{\alpha}N)=(P_{\alpha}\times R_{t_0})/\varPhi$.
\end{lem}

\begin{proof}
%Let $Q_{\alpha}:=\pi^{-1}(\varphi^{-1}(\partial_{\alpha} N\times \{ 1/2\}))$,
%which is a non-trivial $I$-bundle.
%
Note that 
\begin{gather*}
    \mathcal U_2(\partial_{\alpha} N):=(P_{\alpha}\times R_{t_0})/(p,x,y)\sim (\varphi(p), x, -y), \\
              \pi^{-1}(V_{\alpha})=\pi^{-1}\varphi^{-1}(\partial_{\alpha}N\times [0,1).
\end{gather*}
We define  a map $\Psi:\mathcal U_2(\partial_{\alpha} N)\to \pi^{-1}(V_{\alpha})$ as follows: 
Note that for each $(p,x)\in P_{\alpha}\times [0,1)$, 
 $\{ p, \varphi(p)\}$ can be identified with with the boundary of the 
$I$-fiber $I_{\rho(p),x}:= \pi^{-1}\varphi^{-1}(\rho(p)\times \{ x \})$.
Define  $\Psi(p,x,y)$, $-t_0\le y\le t_0$, be the arc on the fiber $I_{\rho(p),x}$ from 
$p$ to $\varphi(p)$.
It is easy to see that $\Psi:\mathcal U_2(\partial_{\alpha} N)\to \pi^{-1}(V_{\alpha})$ gives an 
isomorphism between $I$ fibered spaces.
%
%Now let $\Psi(x,t)$, $-t_0\le t\le t_0$, be the arc  $I_{\pi(x)}$ from$x$ to $\varphi(x)$.
%
\end{proof}

$(3)$\, Put ${\rm int}N:=N\setminus \partial N$ for simplicity.

\begin{ass} \label{ass:embed}
There is an isometric imbedding $g:N\to C_0$ such that $\eta_0\circ g=1_N$.
\end{ass}

\begin{proof}
Set $F_{\alpha}:=\eta_0^{-1}(\partial_{\alpha}N)$.
From the assumption, we may assume that $F_{\alpha}$ is two-sided in the sense 
that the complement of $F_{\alpha}$ in some connected neighborhood of it is disconnected.
Thus there is a connected neighborhood $V_{\alpha}$ of $\partial_{\alpha}N$ in ${\rm int} N$
for which there is an isometric imbedding  $g_{\alpha}:V_{\alpha}\to C_0\setminus F$ such that 
$\eta_0\circ g_{\alpha}=1_{V_{\alpha}}$.

Let $W$ be the maximal connected open subset of ${\rm int}N$ for which 
there is an isometric imbedding  $g_0:W \to C_0\setminus F$ such that 
$\eta_0\circ g_0=1_{W}$ and $g_0(W)\supset g_{\alpha}(V_{\alpha})$. 
We only have to show that $W={\rm int}N$.
Otherwise, there is a point $x\in \partial W\cap {\rm int}N$.
Take a connected neighborhood $W_x$ of $x$ in ${\rm int}N$  such that 
$\eta_0^{-1}(W_x)$ is a disjoint union of open sets $U_1$ and $U_2$
such that $\eta_0:U_i\to W_x$ is an isometry for $i=1,2$.
Obviously one of $U_i$, say $U_1$, meets $g_0(W)$ and the other does not.
We extend $g_0$ to $g_1:W\cup W_x\to C_0\setminus F$ 
by requiring $g_1|_{W_x}=\eta_0^{-1}:W_x\to U_1$.
Since $g_1$ is an isometric imbedding, this is a contradiction to 
the maximality of $W$.

Thus we have an isometric imbedding $g_0:{\rm int}N \to C_0\setminus F$.
Since ${\rm int}N$ is convex and $\eta_0$ is $1$-Lipschitz, $g_0$ preserves
the distance. It follows that $g_0$ extends to 
an isometric imbedding $g:N\to C_0$ which preserves distance. 
\end{proof}

Assertion \ref{ass:embed} shows that 
every component of $F$ is two-sided.
It follows that $C_0=D(N)$, and that $f$ is the reflection of the double $D(N)$.
This completes the proof of Theorem \ref{thm:codim1-detail}
\end{proof}

\begin{proof} [Proof of Lemma \ref{lem:partial}]
Obviously $\partial N\subset \eta_0(F)$.
Suppose that $\eta_0(F)\cap ({\rm int}N)$ is not empty.

\begin{slem} \label{slem:dim}
  $\dim(\eta_0(F)\cap {\rm int}N)\le m-2$, where $m:=\dim N$.
\end{slem}

\begin{proof}
If  $\dim(\eta_0(F)\cap {\rm int}N) = m-1$, then the top-dimensional strata $S$ of 
$\eta_0(F)\cap {\rm int}N$ is a topological $(m-1)$-manifold, and therefore it meets 
the $m$-dimensional strata of $N$ because $N^{\rm sing}\cap {\rm int}N$ has 
codimension $\ge 2$ (Theorem \ref{thm:dim-sing}). 
Take $p\in\eta_0^{-1}(S)$.
It is now easy to see that 
$f$ is the reflection with respect to  $\eta_0^{-1}(S)$ in a small neighborhood of $p$.
It follows that $S$ is a subset of $\partial N$, contradiction to the hypothesis.
\end{proof}

Take a point  $x=\eta_0(p)\in\eta_0(F)\cap {\rm int}N$, and 
consider the directional defivatives  $f_*:\Sigma_p(C_0)\to \Sigma_p(C_0)$ of $f$ at $p$
which is again an isometric involution with fixed point set 
             $$F_*:=\Sigma_p(F)$$
By Corollary \ref{cor:ext-dir} and Sublemma \ref{slem:dim},  $\dim F_*\le m-3$ while $\dim \Sigma_p(C_0)=m-1$.
Repeating this we have a finite sequence of directional derivatives of $f$, $f_*\ldots$,  
each of which is an isometric involution:
\[
     f_{*k}:  \Sigma_{*k}(C_0)\to  \Sigma_{*k}(C_0),
\]
where  $\Sigma_{*k}(C_0)$ denotes a $k$ iterated space of directions,
\[
       \Sigma_{*k}(C_0) = \Sigma_{\xi_{k-1}}(\cdots(\Sigma_{\xi_1}(\Sigma_p(C_0))\cdots),
\]
and $\xi_i$ is taken from the fixed point set of the iterated directional derivatives:
$$
\xi_1\in\Sigma_p(F), \, \xi_2\in\Sigma_{\xi_1}(F_{*}), \ldots, \, \xi_k\in\Sigma_{\xi_{i-1}}(F_{*(k-1)}),
$$
and  $F_{*i}$ denotes the fixed point set of  $f_{*i}:\Sigma_{*i}(C_0)\to \Sigma_{*i}(C_0)$
which coincides with  $F_{*i}=\Sigma_{\xi_{i-1}}(F_{*(i-1)})$.

Note that the iterated space of directions $\Sigma_{*k}(C_0)$ has dimension $m-k$, 
and the iterated fixed point set $F_{*k}\subset \Sigma_{*k}(C_0)$ has dimension $\le m-k-2$.
It follows that for some $k\le m-2$,  $F_{*k}$ becomes a finite set.
It follows that for any  $\xi_{k+1}\in F_{*k}$, 
$$
               f_{*(k+1)}:\Sigma_{\xi_{k+1}}(\Sigma_{*k}(C_0))\to \Sigma_{\xi_{k+1}}(\Sigma_{*k}(C_0))
$$
has no fixed points.  
Put 
\[
             D:= C_0\times_{\tilde\phi}[-t_0,t_0],
\]
and let $\tilde f$ be an isometric involution on $D$ defined by $\tilde f=(f,-{\rm id})$.
From Theorem \ref{thm:singIbund}, 
\[
                      Y = D/\tilde f.
\]
Let  $x=\eta_0(p)$,  $p=(p,0)$, $\xi_i\in\Sigma_{\xi_{i-1}} (F_{*(i-1)})$, $1\le i\le k+1$, be as above.
Note that 
\[
                       \Sigma_x(Y)=\Sigma_p(D)/\tilde f_*, \,\, \Sigma_x(X)=\Sigma_p(C_0)/f_*.
\]
Let $\zeta_1\in\Sigma_x(\eta_0(F))\subset \Sigma_x(X)\subset\Sigma_x(Y)$ 
be the element corresponding to $\xi_1\in\Sigma_p(F)\subset \Sigma_p(C_0)\subset \Sigma_p(D)$.
Note that 
\[
                              \Sigma_p(D)=\{ \xi_{\pm}\}*\Sigma_p(C_0)
\]
and $\tilde f_*=(f_*, -{\rm id})$ interchanges $\xi_{+}$ and $\xi_{-}$ and preserves $\Sigma_p(C_0)$.
Next consider 
\[
     \Sigma_{\zeta_1}(\Sigma_x(Y)) = \Sigma_{\xi_1}(\Sigma_p(D))/\tilde f_{**},
\]
where $\tilde f_{**}$ denotes the directional derivatives of $f_{*}$ at $\zeta_1$.
Note that $\Sigma_{\xi_1}(\Sigma_p(D))$ is still isometric to $\{ \xi_{\pm}\}*\Sigma_{\xi_1}(\Sigma_p(C_0))$ 
and $\tilde f_{**}=(f_{**}, -{\rm id})$ interchanges $\xi_{+}$ 
and $\xi_{-}$ and preserves $ \Sigma_{\xi_1}(\Sigma_p(C_0))$.
Similarly and finally we consider
\begin{align} 
     \Sigma_{\zeta_{k+1}}(\Sigma_{*k}(Y)) = \Sigma_{\xi_{k+1}}(\Sigma_{*k}(D))/\tilde f_{*k+1}, \label{eq:induct}
\end{align}
where $\zeta_{k+1} \in \Sigma_{*k}(Y)$ is the element corresponding to 
$\xi_{k+1}\in\Sigma_{*k}(D)$, and $\tilde f_{*{k+1}}=(f_{*{k+1}}, -{\rm id})$ freely acts on 
$\Sigma_{\xi_{k+1}}(\Sigma_{*k}(D))$.
Recall that 
\[
       \ell :=\dim \Sigma_{\xi_{k+1}}(\Sigma_{*k}(D)) =m-k \ge 2.
\]
Note that the iterated spaces of directions of  the smoothable spaces $Y\setminus \partial C_{t_0}$ % and $C_0$
must be all homeomorphic to spheres (Theorem \ref{thm:iterated-sp}). However \eqref{eq:induct} shows that 
$\Sigma_{\zeta_{k+1}}(\Sigma_{*k}(Y))$ is homeomorphic to a quotient
$\mathbb S^{\ell}/\mathbb Z_2$ for  $\ell\ge 2$
by a free $\mathbb Z_2$-action,  which is a contradiction. 
This completes the proof of Lemma \ref{lem:partial}.
\end{proof}

\subsection{Inradius collapse to almost regular spaces}

Next we consider the case where $M_i$ inradius collapses to almost regular 
Alexandrov space $N$. The idea of using an equivariant fibration-capping theorem 
in \cite{Ya:four} was inspired by a recent work \cite{MY2:dim3bdy}.

First we recall this theorem.
 Let $X$ be a $k$-dimensional complete Alexandrov space with curvature
$\ge\kappa$ possibly non-empty boundary.
We denote by $D(X)$ the double of $X$, which is also
an Alexandrov space with curvature $\ge\kappa$.
(see \cite{Pr:alexII}). By definition, $D(X)=X\cup X^*$ glued
along their boundaries, where $X^*$ is another copy of $X$.

A $(k,\delta)$-strainer $\{ (a_i,b_i)\}$ of $D(X)$ at $p\in X$
is called {\it admissible} if
$a_i\in X$, $b_j\in X$ for every $1\le i\le k$, $1\le j\le k-1$
(clearly, $b_k\in X^*$ if $p\in\partial X$ for instance).
Let  $R_{\delta}^D(X)$ denote the set of points of $X$ at which
there are admissible $(k,\delta)$-strainers.
It has the structure of a Lipschitz $k$-manifold with boundary.
Note that every point of $R_{\delta}^D(X)\cap \partial X$
has a small neighborhood in $X$ almost isometric to an open subset
of the half space $\mathbb{R}^k_+$ for small $\delta$.

If $Y$ is a closed domain of $R_{\delta}^D(X)$, then
the $\delta_D$-{\it strain radius} of $Y$ is
defined as the infimum of positive numbers $\ell$ such that
there exists an admissible $(k,\delta)$-strainer of length $\ge \ell$
at every point in $Y$, denoted by
$\text{$\delta_D$-{\rm str.rad}$(Y)$}$.

For a small $\nu>0$, we put
$$
   Y_{\nu} := \{ x\in Y\,|\, d(\partial X, x)\ge\nu\}.
$$
We use the following special notations:
$$
   \partial_0 Y_{\nu} := Y_{\nu}\cap \{ d_{\partial X}=\nu\},\quad
            {\rm int}_0 Y_{\nu}:=Y_{\nu}-\partial_0 Y_{\nu}.
$$
Let $M^n$ be another $n$-dimensional complete Alexandrov space
with curvature $\ge \kappa$ having no boundary.
Let $R_\delta(M)$ denote the set of all $(n,\delta)$-strained points of $M$.

A surjective map $f:M\to X$ is called an
$\epsilon$-{\it almost Lipschitz submersion} if \par
\begin{enumerate}
 \item  it is an $\epsilon$-approximation;
 \item  for every $p,q\in M$
    $$
     \left|\frac{d(f(p), f(q))}{d(p, q)}-\sin\theta_{p,q}\right| < \epsilon,
   $$
\end{enumerate}
where $\theta_{p,q}$ denotes the infimum of $\angle qpx$ when
  $x$ runs over $f^{-1}(f(p))$.

%We denote by $\tau(\epsilon_1,\ldots,\epsilon_k)$ a function depending on
%a priori constants and $\epsilon_i$ satisfying
%$$\lim_{\epsilon_i\to 0}\tau(\epsilon_1,\ldots,\epsilon_k)=0.$$

Now let a Lie group  $G$ act on  $M^n$ and $X$ as isometries. Let 
\[
       d_{e.GH}((M,G), (X,G))
\]
denote the equivariant Gromov-Hausdorff distance as defined in Section \ref{subsec:GH}.
We need to aassume the following on the existence of slice for $G$-orbits:

\begin{asmp} \label{asmp:slice}
For each $p\in X$, there is a {\it slice} $L_p$ at $p$. Namely
$U_p:=GL_p$ provides a $G$-invariant tubular neighborhood of $Gp$
which is $G$-isomorphic to $G\times_{G_p} L_p$.
\end{asmp}

Obviously Assumption \ref{asmp:slice} is automatically satisfied if $G$ is discrete.
By  \cite{HS}, Assumption \ref{asmp:slice} also holds true if $G$ is compact.

\begin{thm}[Equivariant Fibration-Capping Theorem( \cite{Ya:four}, Thm 18.9)] \label{thm:cap}
Let $X$ and $G$ be as above such that  $X/G$ is compact.
Given $k$ and $\mu>0$ there exist positive numbers $\delta=\delta_k$,
$\epsilon_{X,G}(\mu)$ and $\nu=\nu_{X,G}(\mu)$ satisfying
the following $:$\,\, Suppose $X=R^D_{\delta}(X)$ and  $\text{$\delta_D$-{\rm str.rad}$(X)$}>\mu$.
Suppose $M=R_{\delta_n}(M)$ and
$d_{eGH}((M,G),(X,G))<\epsilon$ for some
$\epsilon\le\epsilon_{X,G}(\mu)$.
Then there exists a $G$-invariant decomposition
\[
                        M  = M_{\rm int} \cup M_{\rm cap}
\]
of $M$ into two closed domains glued along their boundaries, and a
$G$-equivariant Lipschitz map  $f:M  \to X_{\nu}$ such that
\begin{enumerate}
 \item $M_{{\rm int}}$ is the closure of $f^{-1}({\rm int}_0 X_{\nu})$, and
       $M_{\rm cap} = f^{-1}(\partial_0 X_{\nu})$;
 \item the restrictions $f|_{M_{\rm int}}: M_{\rm int} \to X_\nu$ and
     $f|_{M_{\rm cap}} : M_{\rm cap} \to \partial_0 X_\nu$ are
   \begin{enumerate}
     \item  locally trivial fiber bundles;
     \item  $\tau(\delta,\nu,\epsilon/\nu)$-Lipschitz submersions.
   \end{enumerate}
\end{enumerate}
\end{thm}

Here,  $\tau(\epsilon_1,\ldots,\epsilon_k)$ denotes  a function depending on
a priori constants and $\epsilon_i$ satisfying
$$\lim_{\epsilon_i\to 0}\tau(\epsilon_1,\ldots,\epsilon_k)=0.$$

\begin{rem}
If $X$ has no boundary,  then $X_\nu$ is replaced by $X$, $M_{cap}=\emptyset$ 
and $M=N$ in the statement above. 
%According to the construction of the fibration $f$ in the proof, if $\pa X\neq\emptyset$ and $Y=R_\delta^D=X$, then %$N=M$, $N_{\rm int}=M_{\rm int}$ and $N_{\rm cap}= M_{\rm cap}$ and $Y_\nu=X_\nu$.
%\noindent
%$(2)$ Theorem \ref{thm:cap} holds for a general Lie group $G$ if the action of $G$ on $X$
%satisfies an assumption on the slice for every $G$-orbit in $X$. 
%This assumption is automatically 
%satisfied if $G$ is discrete (see \cite{collapsing 4-manifolds} for the details).
\end{rem}

We go back to the situation of Theorem \ref{thm:RMBcap}.
Assume that $M_i$ inradius collapses to an almost regular  Alexandrov space $N$.  
Let us consider the double and the partial double of $\tilde M_i$ and $Y$ respectively 
\[
    D(\tilde M_i) := \tilde M_i\amalg_{\partial \tilde M_i} \tilde M_i, \,\,\,\,   W :=Y \amalg_{C_{t_0}} Y,
\]
where two copies of $Y$ are glued along $C_{t_0}$.
From Perelman's result \cite{Pr:alexII}, both $D(\tilde M_i)$ and  $W$ are Alexandrov space.
Note that both $D(M_i)$ and $W$ admit  canonical isometric $\mathbb Z_2$
actions by the reflections.

The proof of the following lemma is standard, and hence omitted.

\begin{lem}\label{equi-GH-conv}
$(D(\tilde M_i), \mathbb Z_2)$ converges to $(W, \mathbb Z_2)$
with respect to the equivariant Gromov-Hausdorff convergence.
\end{lem}

\begin{proof}[ Proof of Theorem \ref{thm:RMBcap}]
 By Lemma \ref{equi-GH-conv}, for any $\ve>0$, if $i$ is large,
 $$d_{eGH}((D(\tilde M_i),\mathbb{Z}_2),(W,\mathbb{Z}_2))<\ve.$$
 By Theorem \ref{thm:singIbund}, $Y$ is almost regular possibly with almost regular boundary. 
Hence, $W=R^D_{\delta}(W)$ and $\text{$\delta_D$-{\rm str.rad}$(W))$}>\mu$ for  some $\mu>0$.
 Thus by Theorem \ref{thm:cap} and its remark, there exists a $\mathbb{Z}_2$-equivariant capping fibration
\[
                      \tilde f_i:D(\tilde M_i)\to W_{\nu},
\]
where 
\[
                         W_{\nu}=\{x\in W\,|\,d(x, \partial W)\ge \nu\,\}.
\]
Notice that $W_\nu$ is homeomorphic to $W$ because of the form of $Y$. 
Obviously,  $\tilde f_i$ induces a map $f_i:\tilde M_i \to Y$.
By the remark after Corollary \ref{cor:1/2}, $\eta_0:C_0\to X$ is either 
an isometry or a locally isometric double covering.

{\rm Case  (a)}.\, 
If $\eta_0:C_0\to X$ is a double covering, then $C_{t_0}=\pa Y$.
% since in this case $X\subset {\rm int}Y$.
Hence $W$ has no boundary.
Thus in this case, $f_i:\tilde M_i\to Y$ is a fiber bundle with fiber $F_i$ which are closed almost nonngetively curved 
manifolds.  
Since $Y$ is an $I$-bundle over $N$ by Theorem \ref{thm:singIbund}, 
$\tilde M_i$ and hence $M_i$ is an $F_i\times I$-bundle over $N$.

{\rm Case  (b)}.\, 
If  $\eta_0:C_0\to X$ is an isometry, then $Y$ is isometric to $N\times_\phi[0,t_0]$,   and therefore
$\partial Y$ consists of  $\eta(C_0)=X$  and $\eta(C_{t_0})$. 
Thus  $\pa W$ consists two copies of $\eta_0(C_0)$.
Therefore by Theorem \ref{thm:cap}, there exists a $\mathbb{Z}_2$-invariant decomposition
\begin{align}
         D(\tilde M_i) =(D(\tilde M_i))_{\rm int} \cup (D(\tilde M_i))_{\rm cap}, \label{eq:inv-decom}
\end{align}
of $ D(\tilde M_i)$ into two closed domains glued along their boundaries such that
\begin{enumerate}
 \item $(D(\tilde M_i))_{{\rm int}}$ is the closure of $ \tilde f_i^{-1}({\rm int}_0 W_{\nu})$, and
       $(D(\tilde M_i))_{\rm cap} = \tilde f_i^{-1}(\partial_0 W_{\nu})$;
 \item $\tilde f_i|_{(D(\tilde M_i))_{\rm int}}: (D(\tilde M_i))_{\rm int} \to W_\nu$, 
     $ \tilde f_i|_{(D(\tilde M_i))_{\rm cap}} : (D(\tilde M_i))_{\rm cap} \to \partial_0 W_\nu$ are locally trivial fiber bundles,
\end{enumerate}
where 
\[
   \partial_0 W_{\nu} := \{x\in W\,|\,d(x, \partial W)= \nu\,\}, \quad
            {\rm int}_0 W_{\nu}:=W_{\nu} \setminus \partial_0 W_{\nu}.
\]
Since \eqref{eq:inv-decom} is $\mathbb Z_2$-invariant, it induces a 
decomposition
\[
   \tilde M_i =(\tilde M_i)_{\rm int} \cup (\tilde M_i)_{\rm cap}.
\]       
Since $\tilde f_i$ is $\mathbb Z_2$-equivariant, these fibrations induce
fibrations
\begin{align*}
      F_i \longrightarrow &(\tilde M_i)_{\rm int} \longrightarrow Y_{\nu}, \\
   {\rm Cap}_i \longrightarrow &(\tilde M_i)_{\rm cap} \longrightarrow \partial_0Y_{\nu}. 
\end{align*}
From construction, $\partial {\rm Cap}_i$ is homeomorphic to $F_i$.
Note that every cylindrical geodesic in the warped cylinder $C_i\subset \tilde M_i$   
is almost perpendicular to the fibers (\cite{Ya:pinching}, \cite{Ya:conv}).
This implies that  
$(\tilde M_i)_{\rm int} $ is homeomorphic to $\partial(\tilde M_i)_{\rm int}\times [0,1]$, 
and therefore $\tilde M_i$ and hence $M_i$ is homeomorphic to $(\tilde M_i)_{\rm cap}$.
Noting $\partial_0Y_\nu$ is homeomorphic to $N$, we obtain
a fiber bundle
\[
           {\rm Cap}_i \longrightarrow  M_i \longrightarrow  N.
\] 
This completes the proof.
\end{proof}

%%%%%%%%%%%%%%%%%%%%%%%%%%%%%%%%%%%%%%%%%%%%%%%%%%%%%%%%%%%%%%%%%%%%%%%%%%%%%%%%%%
%%%%%%%%%%%%%%%%%%%%%%%%%%%%%%%%%%%%%%%%%%%%%%%%%%%%%%%%%%%%%%%%%%%%%%%%%%%%%%%%%%
\section{The case of unbounded diameters} \label{sec:unbounded}

In this section we provide the proof of Theorem \ref{thm:two}.
In view of Corollary \ref{cor:cpntC0},  for the proof of Theorem \ref{thm:two}(1),  
it suffices to consider inradius collapsed manifolds 
with unbounded diameters.

\begin{rem} \label{rem:unclear}
Theorem \ref{thm:two} (1) was stated in \cite{wong2}, Theorem 5, where 
the following argument was employed: If $k\ge 3$ and if $p\in M$ is the furthest point
from $\partial M$,  then $B(p,r)$, $r={\rm inrad}(M)$, touches $\partial M$ at least 
three points. However it seems to the authors that this is unclear.
\end{rem}

%%%%%%%%%   unbounded moduli space and its boundary %%%%%%%%%%%%%%%%%%%%%%%%%%%%%%%%%%%%%
\subsection{Description of pointed inradius collapse}
%
%In the bounded diameter case, the number of components of $\partial M_i$ is uniformly bounded,
%and $C_0$ is the component wise Gromov-Hausdorff limit of $\partial M_i$.
In the case of unbounded diameter, we do not know the uniform boundedness of the numbers of boundary components of 
inradius collapsed manifolds yet. This forces us to reconsider the descriptions of  limit spaces in Section 
\ref{sec:str-limit}.

%This is why we need a bit careful consideration to define $C_0$, which will be carried out in the following.

Let $\ca M(n,\kappa,\lambda)$ (resp. $\ca M(n,\kappa,\lambda)_{\rm pt}$ denote the set of all isometry classes of 
$n$-dimensional complete Riemannian manifolds $M$ (resp. pointed complete Riemannian manifolds  $(M,p)$ with $p\in \partial M$) satisfying
\[
      K_M \ge \kappa, \,\, |\Pi_{\partial M}|\le\lambda.
\]
We carry out the extension procedure for  $M$ to obtain $\tilde M$.
Let  $\tilde{\ca  M}\ca M(n,\kappa,\lambda)_{\rm pt}$ denote the set of all 
$(\tilde M, M, p)$ with $M\in\ca M(n,\kappa,\lambda)$ and $p\in\partial M$.
We denote by 
\[
                   \partial_0 \tilde{ \ca M}\ca M(n,\kappa,\lambda)_{\rm pt}
\]
the set of all pointed Gromov-Hausdorff limit spaces $(Y,X, x)$ of sequences $(\tilde M_i,M_i,p_i)$ 
in  $\tilde{\ca  M}\ca M(n,\kappa,\lambda)_{\rm pt}$ with  
\[
       {\rm inrad}(M_i)\to 0.
\]
From now on,  $(\tilde M_i, M_i, p_i)$ and  $(\tilde M, M, p)$ are always assumed to be elements 
in  $\tilde{\ca  M}\ca M(n,\kappa,\lambda)_{\rm pt}$.

Now suppose that   a sequence $(M_i, p_i) \in \ca M(n,\kappa,\lambda)_{\rm pt}$ 
converges to a complete length space $(N, q)$ with ${\rm inrad}(M_i)\to 0$, while  $(\tilde M_i,M_i,p_i)$
converges to $(Y,X, x)$.
In a way  similar to Proposition \ref{prop:intrinsic}, we see that 
$X^{\rm int}$ is isometric to $N$.

Next we describe $Y$ as 
\[
  Y = X\bigcup_{\eta_0} C_0\times_{\phi} [0,t_0],
\]
as in the bounded diameter case. 
In the bounded diameter case, the number of components of $\partial M_i$ is uniformly bounded,
and $C_0$ is the component wise Gromov-Hausdorff limit of $\partial M_i$.
In the case of unbounded diameter, we do not know the boundedness of the number of components of $\partial M_i$
yet.
This is why we need a bit careful consideration to define $C_0$, which will be carried out in the following.

Let  
\[
                           C^Y_{t_0}:= \{ y\in Y\,|\, |X,y|=t_0  \,\}.
\]

We begin with the decomposition of $C^Y_{t_0}$ into the connected components:
\[
       C^Y_{t_0} =\coprod_{\alpha\in A}\, C_{t_0}^{\alpha}.
\]
Set
\[
      C^{\alpha}_0 := \frac{1}{\phi(t_0)}\, C_{t_0}^{\alpha}, \,\,\,\qquad C^{\alpha}= C_0^{\alpha}\times_{\phi} [0,t_0],
\]
and 
\[
   C:= \coprod_{\alpha\in A}\, C^{\alpha}, \qquad C_0 =  \coprod_{\alpha\in A}\, C^{\alpha}_0\times \{ 0\}.
\]
Note that each component of $C$ and $C_0^{\rm int}$ is an Alexandrov space with curvature $\ge c(\kappa,\lambda)$.

Each $p\in C_0$ can be identified with the element of $C^Y_{t_0}$, which we write as $\eta(p,t_0)$, and
there is a unique perpendicular $\gamma^{\eta(p, t_0)}(t)$, $0\le t\le t_0$, to $X$ 
satisfying  $\gamma^{\eta(p, t_0)}(t_0)=\eta(p,t_0)$.
We then define the surjective $1$-Lipschitz map $\eta:C\to Y$ as  
 \begin{align*}
    \eta(p,t) = \gamma^{\eta(p, t_0)}(t).
\end{align*}
Obviously $\eta:C\setminus C_0 \to Y\setminus X$ is bijective locally isometric map.

Let  $\varphi_i:B^{\tilde M_i}(p_i, 1/\delta_i) \to B^Y(x, 1/\delta_i)$ be a $\delta_i$-approximation, 
with $\lim \delta_i=0$. 
Note that for each component $C_{t_0}^{\alpha}$ of $C^Y_{t_0}$ and any fixed point 
$y_{\alpha}\in C_{t_0}^{\alpha}$, there is a  component, 
say $\partial^{\alpha}\tilde M_i$, of $\partial \tilde M_i$ for which we have a point
$q_i^{\alpha}\in \partial^{\alpha} \tilde M_i$ satisfying $|\varphi_i(q_i^{\alpha}), y_{\alpha}| <\delta_i$.
For a distinct component $C_{t_0}^{\beta}$ and $y_{\beta}\in C_{t_0}^{\beta}$, 
we also have  a component  $\partial^{\beta}\tilde M_i$ of $\partial \tilde M_i$ for which there is a point
$q_i^{\beta}\in \partial^{\beta} \tilde M_i$ satisfying $|\varphi_i(q_i^{\beta}), y_{\beta}| <\delta_i$.
Since $|C^{\alpha}_{t_0}, C^{\beta}_{t_0}|\ge 2t_0$, it is easily checked that 
$\partial^{\alpha}\tilde M_i$ and $\partial^{\beta}\tilde M_i$ are distinct components, 
and hence  $|\partial^{\alpha}\tilde M_i, \partial^{\beta}\tilde M_i|\ge 2t_0$.
Thus we have that 
\begin{align}
       \lim_{i\to\infty} (\partial^{\alpha}\tilde M_i, q^{\alpha}_i) = (C_{t_0}^{\alpha}, y_{\alpha}), \,\,\,
       \lim_{i\to\infty} (\partial^{\beta}\tilde M_i, q^{\beta}_i) = (C_{t_0}^{\beta}, y_{\beta}),   \label{eq:bij-conv}
\end{align}
under the convergence $(\tilde M_i,p_i) \to (Y, x)$.
In particular,  the component  $\partial^{\alpha}\tilde M_i$ is uniquely determined by  $C_{t_0}^{\alpha}$. 

Now for the map
\[
      \eta_0 := \eta|_{C_0}: C_0\to X,
\]
from an argument similar to the bounded diameter case, we see that 
\[
                             \# \, \eta_0^{-1}(x)\le 2,
\]
for every $x\in X$.
Thus we can define the involution $f : C_0 \to  C_0$ as in Section \ref{sec:metric}. 
Note that all the results in Subsections \ref{ssec:metric-prelim}, \ref{ssec:differential} and \ref{ssec:gluing}
still holds true for the present situation of noncompact $C_0$, because the arguments there are 
local. 
In particular, we see that
\begin{align}
  {\rm  the \,\, number \,\, of \,\, components \,\,\,  of  \,\,\,}  C_0 \,\,{\rm is \,\,  at \,\, most \,\, two.\hspace{2cm} } \label{eq:compo}
\end{align}
Thus we see that all the results in Section \ref{sec:metric}  
holds true except Corollary \ref{cor:cpntC0}, and therefore we conclude that 
   \[
         Y= C_0^{\rm int}\times_{\tilde\phi}[-t_0, t_0]/\tilde f,
    \]
where $\tilde\phi(t)=\phi(|t|)$, and
$N=X^{\rm int}$ is isometric to $C_0^{\rm int}/f$.

Now  we immediately have the following:

\begin{cor}\label{cor:dim-collaps-pointed}
If $(M_i,p_i) \in \ca M(n,\kappa,\lambda)$ 
inradius collapse to $(N, q)$ with respect to the pointed Gromov-Hausdorff convergence,
then we have $\dim M_i>\dim N.$
\end{cor}

\begin{thm} \label{thm:limit-alex'}
Let  a sequence of pointed complete Riemannian manifolds $(M_i,p_i)$ in $\ca M(n,\kappa,\lambda)$ 
inradius collapse to a pointed length space  $(N, q)$ with respect to the pointed Gromov-Hausdorff convergence.  
Then $N$ is an Alexandrov space with curvature $\ge  c(\kappa,\lambda)$, where
$c(\kappa,\lambda)$ is a constant depending only on $\kappa$ and $\lambda$.
\end{thm}

To have Corollary \ref{cor:cpntC0}
in the case when $Y$ is noncompact  is a main purpose of the rest of this section. 

\begin{rem}  \label{rem:difficult}
Note that \eqref{eq:compo} only  shows that there are at most two components of 
$\partial M_i$ meeting  a bounded region from the reference point $p_i$.
Thus \eqref{eq:compo} does not immediately imply that 
the number of components of $\partial M_i$ is at most two.
%for inradius collapsed manifold $M$ with unbounded diameter,
This is because the convergence  $(\tilde M_i, p_i)\to (Y,p)$ is only under the {\it pointed} Gromov-Hausdorff topology.
Namely, there is still a possibility that some component of $\partial M_i$ disappear to infinity 
under that convergence.  
\end{rem}

To overcome the difficulty stated in Remark \ref{rem:difficult}, we investigate the local connectedness 
of $\partial M_i$ in more detail. To carry out this, it is helpful to consider a special pointed Gromov-Hausdorff
approximations similar to \eqref{eq:g*}, which is verified below. 

%
%\begin{thm} \label{thm:limit-alex'}
%Let  a sequence of pointed complete Riemannian manifolds $(M_i,p_i)$ in $\ca M(n,\kappa,\lambda)$ 
%inradius collapse to a pointed length space  $(N, q)$ with respect to the pointed Gromov-Hausdorff convergence.  
%Then $N$ is an Alexandrov space with curvature $\ge  c(\kappa,\lambda)$, where
%$c(\kappa,\lambda)$ is a constant depending only on $\kappa$ and $\lambda$.
%\end{thm}

%We introduce a more refined version of the pointed Gromov-Hausdorff convergence.
Let $\iota_{\partial M}:((\partial M)^{\rm int}, d_{\partial M^{\rm int}}) \to ((\partial M)^{\rm ext}, d_{\partial M^{\rm ext}})$ 
be the canonical map,
where $(\partial M)^{\rm ext}$ is equipped with the exterior metric in $M$. 
Let $\omega_M:M\to \partial M$ be a nearest point map. It should be noted that 
although $\omega_M$ is not continuous in general,  it will not affect the argument below
(compare Proposition \ref{prop:g*}).

For $(\tilde M, M, p)$ and 
$(Y,X,x)\in  \partial_0\ca M\ca M(n,\kappa,\lambda)_{\rm pt}$ with 
\begin{equation*}
            Y = X\bigcup_{\eta_0} C_0\times_{\phi} [0,t_0], \label{eq:Y}
\end{equation*}
as described above,  set  
\begin{align*}
     \partial M^{\rm int}(p,1/\delta) &:= (\partial M\cap B^{\tilde M}(p, 1/\delta), d_{\partial M^{\rm int}}),\\
     C_0^{\rm int}(p_0,1/\delta) &:= (C_0\cap B^C(p_0, 1/\delta,  d_{C_0^{\rm int}}),
 \end{align*}
where $p_0\in \eta_0^{-1}(x)$.
Note that if $q,q'\in \partial M^{\rm int}(p,1/\delta)$ belong to distinct components of $\partial M$, then 
the distance between them in  $\partial M^{\rm int}(p,1/\delta)$ is infinity:  $d_{\partial M^{\rm int}}(q,q')=\infty$.
Similarly, we also consider 
\begin{align*}
     \partial \tilde M^{\rm int}(p,1/\delta) &:= ( \partial \tilde M\cap B^{\tilde M}(p, 1/\delta), d_{\partial \tilde M}),\\
     C_{t_0}(p_0,1/\delta) &:= (C_{t_0} \cap B^C(p_0, 1/\delta),  d_{C_{t_0}}).
 \end{align*}

\begin{slem}
The number of connected components of $\partial M_i$ which intersect $\partial M_i^{\rm int}(p_i,1/\delta_i)$ 
is at most two.
\end{slem}

\begin{proof}
Suppose that there are three points $q_i^{\alpha}\in \partial M_i^{\rm int}(p_i,1/\delta_i)$ which belong to 
three distinct components  $\partial^{\alpha} M_i$, $1\le\alpha\le 3$.
Let $\hat q_i^{\alpha}\in \partial \tilde M(p_i,1/\delta) $ be the image of  $q_i^{\alpha}$
under the projection to $\partial \tilde M_i$ along perpendiculars,
which belongs to the component $\partial^{\alpha}\tilde M_i$corresponding to $\partial^{\alpha} M_i$.
Since $C_0$ as well as $C_{t_0}$ has at most two components, we may assume that 
$\varphi_i(q_i^1)$ and $\varphi_i(q_i^2)$ 
are in the same component of $C_0$. It turns out that 
$\varphi_i(\hat q_i^1)$ and  $\varphi_i(\hat q_i^2)$ are in the same component of  $C_{t_0}$, 
%Since $\varphi(\hat q_i^1)$ and $\varphi(\hat q_i^2)$ are in the same component of $C_0$.
which  contradicts \eqref{eq:bij-conv}.
\end{proof}

\begin{defn} \label{def:pGH} \upshape
For $(\tilde M, M, p)$ and 
$(Y,X,x)\in  \partial_0\ca M\ca M(n,\kappa,\lambda)_{\rm pt}$ with 
\begin{equation*}
            Y = X\bigcup_{\eta_0} C_0\times_{\phi} [0,t_0], \label{eq:Y2}
\end{equation*}
we define the pointed Gromov-Hausdorff distance$^*$ 
\begin{align}
    d^{*}_{pGH}((\tilde M, M, p), (Y, X, x))  \label{eq:pGH}
\end{align}
as the infimum of those $\delta>0$ such that 
\begin{enumerate}
 \item there exists a component-wise $\delta$-approximation 
           $\psi:  (\partial M)^{\rm int} (p,1/\delta) \to  C_0^{\rm int}(x, 1/\delta);$ 
 \item  the map $\varphi: B^{M^{\rm ext}}(p,1/\delta) \to B^{X^{\rm ext}}(x,1/\delta)$  defined by  
          \[
                           \varphi=\eta_0\circ \psi\circ\iota_{\partial M}^{-1}\circ\omega_M
          \] 
         is a  $\delta$-approximation:
%\xymatrix{
%A \ar[rr]^B \ar[d]_C & & E\ar[d]^D\\
%F \ar[r]_G & H \ar[r]_I & J}
$$\xymatrix{%\ar@{}[dr]|\circlearrowleft 
M^{\rm ext} \ar[r]^{\omega_M} \ar[d]_\varphi & \partial M^{\rm ext} \ar[r]^{\iota_M^{-1}} & \partial M^{\rm int}\ar[d]^\psi\\
X^{\rm ext}   & & C_0^{\rm int}\ar[ll]_{\eta_0}
}$$
 \item the map $\varPhi:B^{\tilde M}(p,1/\delta) \to  B^Y(x,1/\delta)$ defined by
  \[
    \varPhi(q) = \begin{cases}   
                          \varphi(q), \, &{\rm if}\,\,\, q\in M\cap B^{\tilde M}(p,1/\delta) \\
                           (\varphi(q_0), t),   \,   
                                                      &{\rm if}\,\,\,q=(q_0,t)\in \partial M\times_{\phi} [0,t_0]\cap B^{\tilde M}(p,1/\delta).
                         \end{cases}
  \]                
is a $\delta$-approximation.
\end{enumerate}
\end{defn}

This definition is justified by the following lemma.  

\begin{lem} \label{lem:approx}
Let   $(\tilde M_i, M_i, p_i) \in \tilde{\ca  M}\ca M(n,\kappa,\lambda)_{\rm pt}$  converge to $(Y, X,x)$ 
in $\partial_0 \tilde{\ca  M}\ca M(n,\kappa,\lambda)_{\rm pt}$ for the usual pointed Gromov-Hausdorff  topology. 
Then there exists a component-wise $\delta_i$-approximation  
\[
                      \psi_i: (\partial M_i)^{\rm int}(p_i,1/\delta_i) \to C_0^{\rm int} (x, 1/\delta_i)
\]
with $\lim \delta_i= 0$ such that 
the maps 
\[
     \varphi_i: B^{M_i^{\rm ext}}(p_i,1/\delta_i)  \to  B^{X^{\rm ext}}(x,1/\delta_i), \\ 
   \varPhi_i: B^{\tilde M_i}(p_i,1/\delta_i) \to  B^Y(x,1/\delta_i)
\]  
defined as in Definition \ref{def:pGH} via $\psi_i$
are  $\delta_i'$-approximations with $\lim\delta_i'=0$.
\end{lem}

\begin{proof}
Let  $\lambda_i:B^{\tilde M_i}(p_i,1/\epsilon_i) \to  B^Y(x,1/\epsilon_i)$ be an $\epsilon_i$-approximation
with $\lim \epsilon_i=0$. We may assume that 
when it is restricted to the boundary, it provides a component-wise  $\epsilon_i$-approximation
$\lambda_i^{t_0}:B^{\tilde M_i}(p_i,1/\epsilon_i)\cap \partial\tilde M_i \to  B^Y(x,1/\epsilon_i)\cap C_{t_0}^Y$.
Since $\partial \tilde M_i$ and $C_{t_0}$ are convex and  $1/\phi(t_0)$-homothetic to 
$(\partial M_i)^{\rm int}$ and $C_{0}$ respectively, $\lambda_i^{t_0}$ gives a component-wise 
$\epsilon_i/\phi(t_0)$-approximation 
\[
   \psi_i:  (\partial M_i)^{\rm int}(p_i, 1/(\phi(t_0)\epsilon_i)) \to   C_{0}^{\rm int}(p_0, 1/(\phi(t_0)\epsilon_i)).
\]
Let $\delta_i:=\phi(t_0)\epsilon_i$, and define  $\varphi_i$ and $\varPhi_i$ as in Definition \ref{def:pGH}.
In a way similar to Proposition \ref{prop:g*}, one can easily show that
the restriction 
\[
    \varPhi_i: B^{\tilde M_i}(p_i,1/\delta_i) \setminus M_i \to  B^Y(x,1/\delta_i)\setminus X
\]
is a $\delta_i'$-approximation with $\lim \delta_i' = 0$.  In particular this implies that 
\[
    \varphi_i: B^{M_i^{\rm ext}}(p_i,1/\delta_i)\cap \partial M_i  \to  B^{X^{\rm ext}}(x,1/\delta_i)
\]
is also a $\delta_i'$-approximation. Let $\nu_i:={\rm inrad}(M_i)$. 
Since $B^{M_i^{\rm ext}}(p_i,1/\delta_i)\cap \partial M_i$ is $\nu_i$-dense in 
$B^{M_i^{\rm ext}}(p_i,1/\delta_i)\cap \partial M_i$, $\varphi_i$ is certainly a $\delta_i''$-approximation
with $\lim \delta''_i=0$.
Since $\varPhi_i$ is a natural extension of $\varphi_i$,  $\varPhi_i$ is also  a $\delta_i''$-approximation.
\end{proof}

\begin{lem} \label{lem:prec}
For each $\delta>0$ there exists a positive number $\epsilon=\epsilon(\delta)$ such that 
if  $(M,p)$ in $\ca M(n,\kappa,\lambda)_{\rm pt}$  satisfies   ${\rm inrad}(M)<\epsilon$,
then 
\[
   d^*_{pGH}((\tilde M, M,p), (Y,X,x))<\delta,
\]
for some $(Y,X,x)$ contained in $\partial_0\tilde{\ca M}\ca M(n,\kappa,\lambda)_{\rm pt}$.
\end{lem}

\begin{proof}
Lemma \ref{lem:prec} follows from Lemma \ref{lem:approx} and the precompactness of 
$\tilde{\ca  M}\ca M(n,\kappa,\lambda)_{\rm pt}$ combined with a contradiction argument.
\end{proof}

If  $(Y,X,x)\in \partial_0\ca M\ca M(n,\kappa,\lambda)_{\rm pt}$ satisfies the conclusion of 
Lemma \ref{lem:prec} for $(M,p) \in\ca M(n,\kappa,\lambda)_{\rm pt}$, we call it a 
{\it $\delta$-limit} of $(\tilde M, M, p)$, which is also denoted 
by $\mathcal Y(M, p)$ for simplicity:
\[
       \mathcal Y(M, p)=(Y,X,x).
\]

\subsection{Local/global connectedness of boundary}

In this subsection, using Lemma \ref{lem:prec}, we first investigate the local behavior of connectedness 
of boundary  of a inradius collapsed manifold, and provide the proof of Theorem \ref{thm:two}.
 
\begin{defn} \upshape
Let  $(Y,X,x)\in \partial_0\ca M\ca M(n,\kappa,\lambda)$ and $y\in X$.
We call $y$ a {\it single point} (resp {\it a  double point}) if $\# \eta_0^{-1}(y)=1$ (resp. $\# \eta_0^{-1}(y)=2$).
We say that $(Y,X,x)$ is  single (resp  double) if every element of $X$ is single (resp double).
If  $(Y,X,x)$ neither   single nor  double, it is called {\it mixed}.
We also say that $(Y,X,x)$ is  single (resp.  double) in scale $R$
if  every element of $X\cap B^Y(x, R)$  is  single (resp.  double). 
If  $(Y,X,x)$ is neither   single nor  double in scale $R$, it is called mixed in scale $R$.
\end{defn}

From now on, to prove Theorem \ref{thm:two}, we analyze the local structure of $\partial M$
about the connectedness when ${\rm inrad}(M)<\epsilon$.
By Lemma \ref{lem:prec}, for any $p\in M$, there exists a $\delta$-limit $ \mathcal Y(M, p)=(Y,X,x)$
together with 
\begin{enumerate}
 \item a $\delta$-approximation $\psi:(\partial M)^{\rm int}(p, R)\to C_0^{\rm int}(p,R);$
 \item a $\delta$-approximation  $\varphi:=\eta_0\circ\psi\circ\iota_{\partial M}^{-1}\circ\omega_M:B^{M^{\rm ext}}(p, R)\to
          B^{X^{\rm ext}}(x, R)$. 
\end{enumerate}

Note that for every $p_1,p_2\in B^{M^{\rm ext}}(p, R)$, 
\begin{equation}
  \begin{aligned}
      |\varphi(p_1), \varphi(p_2)|_{X^{\rm int}} &\le L |\varphi(p_1), \varphi(p_2)|_{X^{\rm ext}} \\
           & \le L(|p_1,p_2|_{M^{\rm ext}} + \delta) \\
           & \le  L(|p_1,p_2|_{M^{\rm int}} + \delta)     \label{eq:varphi}
\end{aligned} 
\end{equation}

Those approximation maps  are effectively used in the proofs of the following lemma.
%Lemma  \ref{lem:sdm-all}  below.
%\ref{lem:single}, \ref{lem:loc-double} and \ref{lem:loc-mixed} 

\begin{lem}\label{lem:sdm-all}
For any $R>0$ there exists  $\delta_0>0$ satisfying the following: 
For every $0<\delta\le \delta_0$, let $\epsilon=\epsilon(\delta)>0$ be as in Lemma \ref{lem:prec}.
For  each $M$ in $\ca M(n,\kappa,\lambda)$ with  ${\rm inrad}(M)<\epsilon$ and for each  $p\in\partial M$, 
we have the following: Let  $\mathcal Y(M, p) $ any $\delta$-limit  of  $(M,p)$.
%Then 
\begin{enumerate}
 \item  If $\mathcal Y(M, p) $ is  single  in scale $R$, then every $p_1, p_2 \in \partial M\cap B^{\tilde M}(p, R)$ can be joined 
          by a curve in $\partial M$ of length $\le L|p_1,p_2|_{M} + (L+1)\delta$.
 \item If  $\mathcal Y(M, p) $ is  double in scale $R$, then there exists a point $p'\in\partial M $ with $|p,p'|_{M}<\delta$
       such that every $q\in\partial M\cap B^{\tilde M}(p, R)$ can be joined to $p$ or $p'$
        by a curve in $\partial M$ of length $\le L|p,q|_M + (L+2)\delta;$
   %   \end{enumerate}
 \item  If $\mathcal Y(M, p) $ is  mixed in scale $R$, then there exists a point $p_0\in \partial M\cap B^{\tilde M}(p,R)$ 
       such that every point $q$ in $\partial M\cap B^{\tilde M}(p,R)$ can be joined to $p_0$ by a curve  
       in $\partial M$ of length $L|p_0,q|_{M} +(L+2)\delta$.
\end{enumerate}
\end{lem}

\begin{proof} Let  $(Y, X, x):= \mathcal Y(M, p)$, and $\psi$, $\varphi$ be  approximation maps 
as above. 

 $(1)$\,
Put  $x_i:=\varphi(p_i)\in X$, $i=1,2$.  
Take $\tilde x_i \in C_{0}$ such that $\eta_0(\tilde x_i)=x_i$.
Lemma \ref{lem:lift} shows  $|\tilde x_1, \tilde x_2|=|x_1, x_2|$. 
Since $\psi$ is a $\delta$-approximation and $\psi(p_i)=\tilde x_i$, it follows from 
\eqref{eq:varphi} that  
\begin{align*}
  |p_1, p_2|_{\partial M^{\rm int}} &< |\tilde x_1, \tilde x_2|_{C_0^{\rm int}} +\delta 
                                                                        = |x_1, x_2|_{X^{\rm int}} +\delta \\
                                           &<L|p_1, p_2|_{M} + (L+1)\delta.
\end{align*}
 
%
%\begin{lem} \label{lem:loc-double}
%For any $R>0$ there exists  $\delta_0<1/R$ satisfying the following: 
%For every $0<\delta\le \delta_0$, let $\epsilon=\epsilon(\delta)>0$ be as in Lemma \ref{lem:prec}.
%Then for $M$ in $\ca M(n,\kappa,\lambda)$ with  ${\rm inrad}(M)<\epsilon$, 
%if  a $\delta$-limit $\mathcal Y(M, p) $ is  double in scale $R$ for some $p\in\partial M$,
%then there exists a point $p'\in\partial M $ satisfying
%\begin{enumerate}
% \item $|p,p'|_{M}<\delta;$
%\item every $q\in\partial M\cap B^{\tilde M}(p, R)$ can be joined to $p$ or $p'$
%        by a curve in $\partial M$ of length $\le L|p,q|_M + (L+2)\delta$.
%\end{enumerate}
%\end{lem}

%\begin{proof}
$(2)$\,
%Let  $(Y, X, x):= \mathcal Y(M, p)$, and $\psi$, $\varphi$ be approximation maps 
%as above. 
Set $x:=\varphi(p)$, $y:=\varphi(q)$. Since   $(Y, X, x)$ is double in scale $R$, we can put
 $\{ \tilde x_1,\tilde x_2\} :=\eta_0^{-1}(x)$ and  $\{ \tilde y_1,\tilde y_2\} :=\eta_0^{-1}(y)$.
Let $\gamma:[0,1]\to X$ be a minimal geodesic joining $x$ to $y$. 
From Lemma \ref{lem:lift},  there are lifts $\tilde \gamma_i:[0,1]\to C_0$  of $\gamma$
starting from $\tilde x_i$, where we may assume $\tilde\gamma_i(1)=\tilde y_i$ and
$\tilde x_1=\psi(p)$. 
%Since $\psi$ is a $\delta$-approximation, 
If $\psi(q)=\tilde y_1$, then 
\begin{align*}
     |p,q|_{\partial M^{\rm int}} & <|\tilde x_1,\tilde y_1|_{C_0^{\rm int}} + \delta =|x, y|_{X^{\rm int}}+\delta \\
                                       &<L|p,q|_{M}+ (L+1)\delta.
\end{align*}
If  $\psi(q)=\tilde y_2$, then take a point $p'$ with $|\psi(p'), \tilde x_2|<\delta$.  Then similarly
we have  
$|p',q|_{\partial M^{\rm int}}<L|p,q|_{M} + (L+2)\delta.$
%This completes the proof.
%\end{proof}

%\begin{lem} \label{lem:loc-mixed}
%For any $R>0$ there exists  $\delta_0<1/R$ satisfying the following: 
%For every $0<\delta\le \delta_0$, let $\epsilon=\epsilon(\delta)>0$ be as in Lemma \ref{lem:prec}.
%Then for  $M$ in $\ca M(n,\kappa,\lambda)$ with  ${\rm inrad}(M)<\epsilon$, 
%if  a $\delta$-limit $\mathcal Y(M, p) $ is  mixed  in scale $R$ for some $p\in M$,
%then there exists a point $p_0\in \partial M\cap B^{\tilde M}(p,R)$ such that every point 
%$q$ in $\partial M\cap B^{\tilde M}(p,R)$ can be joined to $p_0$ by a curve  
%in $\partial M$ of length $L|p_0,q|_{M} +(L+2)\delta$.
%\end{lem}

%\begin{proof}

$(3)$\,
%Let  $(Y, X, x):= \mathcal Y(M, p)$, and $\psi$, $\varphi$ be  approximation maps 
%as above. 
Let $x_0\in X$ be a single point with $|x,x_0|\le R$, and take $\tilde x_0\in C_0$  and 
$p_0\in \partial M$ such that 
       $\eta_0(\tilde x_0)=x_0$  and $|\psi(p_0), \tilde x_0|<\delta$
Let   $\gamma:[0,1]\to X$ be a minimal geodesic from $x_0$ to $\varphi(q)$.
Since $\tilde x_0\in C_0^1$, there is  a unique minimal geodesic  $\tilde\gamma:[0,1]\to C_0$  
from $\tilde x_0$ to $\psi(q)$ with $\eta_0\circ\tilde\gamma=\gamma$.
%Since $\psi$ is a $\delta$-approximation, 
We then have 
\begin{align*}
                |p_0,q|_{\partial M^{\rm int}}  &<|\tilde x_0, \psi(q)|_{C_0^{\rm int}} + \delta =  |x_0, \varphi(q)|_{X^{\rm int}} +\delta  \\
                         & \le |\varphi(q), \varphi(p_0)|_{X^{\rm int}} +  |\varphi(p_0), x_0|_{X^{\rm int}} +\delta\\
                         & \le L|p_0,q|_{M} + (L+2)\delta.
\end{align*}
\end{proof}

From now on, for a fixed $R>0$, let  $\delta=\delta_0(n,\kappa,\lambda, R)>0$ and  
$\epsilon=\epsilon(\delta)>0$ be as determined in Lemma \ref{lem:sdm-all}.

\begin{lem} \label{lem:double}
If  $M\in \ca M(n,\kappa,\lambda)$ has  inradius  ${\rm inrad}(M)<\epsilon$ and  
disconnected boundary, then every $\delta$-limit $\mathcal Y(M, p) $    
is double in scale $R$ for every $p\in \partial M$.
%For  each $M$ in $\ca M(n,\kappa,\lambda)$ with  ${\rm inrad}(M)<\epsilon$,
%let   some $\delta$-limit $\mathcal Y(M, p) $ for some $p\in\partial M$ be given.
%For any $R>0$ and $\delta<1/R$, there exists $\epsilon>0$ satisfying the following: 
\end{lem}

\begin{proof}
Suppose that  some $\delta$-limit $\mathcal Y(M, p)=(Y, X,x) $ is single or mixed in scale $R$.
%We show that  $\partial M$ is connected. 
First note that by Lemma \ref{lem:sdm-all} (1), (3), every points
$q_1, q_2$ in $\partial M\cap B^{\tilde M}(p,R)$ can be joined by a curve 
 in $\partial M$. 
Take  a point $p_{\alpha}\in\partial M$ contained in a component different from the component containing $p$. 
Let $c:[0,\ell]\to M$ be a unit speed minimal geodesic in $M$ from $p$ to $p_{\alpha}$.
For each $k$ with $1\le k\le [2\ell/R]$, take $p_k\in\partial M$ with $|p_k, c(kR/2)|_M <\epsilon$.
Note that 
\[
        B^{\tilde M}(p_k, R)\cap B^{\tilde M}(p_{k+1}, R)\cap \partial M \neq \emptyset
\]
for each $1\le k\le [2\ell/R]-1$.
By applying  Lemmas  \ref{lem:sdm-all} to $p_k$
together with a standard monodoromy argument, 
we see that 
$p_{\alpha}$ can be joined to $p$ by a curve in $\partial M$, which is a contradiction. 
\end{proof}

\begin{lem} \label{lem:disconn-pi}
Suppose that  $M\in \ca M(n,\kappa,\lambda)$ has inradius ${\rm inrad}(M)<\epsilon$ and  
disconnected boundary. For any $p\in\partial M$, let $(Y, X, x)$ be any $\delta$-limit for $(M, p) $.
For any $y\in X$, take distinct points $y_1\neq y_2\in C^Y_{t_0}$ such that 
$|y_i, y| = t_0$. Then we have $|y_1, y_2| = 2 t_0$. 
\end{lem}

\begin{rem}
In Lemma \ref{lem:disconn-pi}, we need the assumption on the disconnectedness of $\partial M$.
Namely, for some $(Y, X, x)$ which is double, the conclusion of Lemma \ref{lem:disconn-pi}
does not hold. For instance, take the M\''{o}bius band 
\[
                 Y=S^1_{\ell}\tilde\times_{\tilde\phi} [-t_0, t_0].
\]
If the length $\ell$ of $X=S^1_{\ell}$ is smaller than $t_0$, then $|y_1, y_2|< 2t_0$
for every $y\in X$ and $y_i\in C^Y_{t_0}$ with $|y_i, X|=t_0$, $i=1,2$.
\end{rem}

\begin{proof}[Proof of Lemma \ref{lem:disconn-pi}]
First note that $(Y, X, x)$ is double in scale $R$.
Suppose  $|y_1, y_2| < 2 t_0$, and 
take a minimal geodesic $\gamma$ joining them in $Y$.
Then $\gamma$ does not meet $X$, and therefore we can project $\gamma$ to $C^Y_{t_0}$.
The obtained  curve $\pi_{t_0}(\gamma)$ joins $y_1$ and $y_2$ in $C_{t_0}^Y$.
Thus the two elements $\tilde y_1, \tilde y_2$ of $\eta_0^{-1}(y)$ can be joined in $C_{0}$.
Take $q_1, q_2\in\partial M$ such that $|\psi(q_k.),\tilde y_k|<\delta$ for $k=1,2$.
Lemma \ref{lem:sdm-all}(2) shows that every $p'\in \partial M\cap B^{\tilde M}(p, R)$ 
can be joined to $q_1$ or $q_2$ by a curve in $\partial M$.
 By a monodoromy argument as in Lemma \ref{lem:double}, we can conclude that  
every $q\in \partial M$ can be joined to $q_1$ or $q_2$ by a curve in $\partial M$,
which is a contradiction.
\end{proof}

We are now ready to prove Theprem \ref{thm:two}.

\begin{proof}[Proof of Thoerem \ref{thm:two}] 
We assume that  ${\rm inrad}(M)<\epsilon$.

(1)\,
Suppose that $\partial M$ is disconnected.
By Lemma \ref{lem:double},  every $\delta$-limit $\mathcal Y(M, p) $    
is double in scale $R$ for every $p\in M$. 
Take  $p_{\alpha}$ and $p_{\beta}$ from distinct components of $\partial M$. 
For every $p\in\partial M$, let $c:[0,\ell]\to M$ be a  unit speed  curve in $M$ from $p_{\alpha}$ to $p_{\beta}$
through $p$.
For each $k$ with $1\le k\le [2\ell/R]$, take $p_k\in\partial M$ with $|p_k, c(kR/2)|_M <\epsilon$.
By applying  Lemma \ref{lem:sdm-all} (2) to each $p_k$
together with a standard monodoromy argument as in Lemma \ref{lem:double},
we see that  $p$  can be joined to $p_{\alpha}$ or  $p_{\beta}$
by a curve in $\partial M$.  
Therefore we conclude that the number of boundary components of 
$M$ is at most two.

(2)\,
Suppose that $\partial M$ has two components.
By Lemma \ref{lem:double}, any $\delta$-limit $\mathcal Y(M,p)=(Y, X, x)$ is double in scale $R$
for every $p\in\partial M$.
Therefore for any $y\in X$, there are distinct $y_1\neq y_2\in C_{t_0}^Y$ with $|y_k, y|=t_0$,  $k=1,2$.
%
%Take $q_k\in\partial \tilde M$,  which are $\delta$-close to $y_k$ in the Gromov-Hausdorff
%distance. From Lemma \ref{lem:sdm-all} (2), $q_1$ and $q_2$ must belong to distinct
%components of $\partial \tilde M$, which implies $|q_1,q_2|\ge 2t_0$, and hence 
Lemma \ref{lem:disconn-pi} shows that 
\[
                             |y_1,y_2|= 2t_0.
\]
Let $W$ be a component of $\partial \tilde M$, and consider the distance function
$d_W$ from $W$. The above observation shows that 
for any $p\in M$, there exists a point $q\in \partial\tilde M$ such that 
\[
             \tilde \angle Wpq >\pi-\tau(\delta).
\]
That is,  $d_W$ is $\pi/2-\tau(\delta)$-regular on a neighborhood of $M$ in $\tilde M$.
% if $\delta=\delta(\epsilon_0, t_0)>0$
%is taken small enough. This means that 
This makes it possible to define locally defined gradient-like vector fields for $d_W$ on neighborhoods of the points
of $M$. Then by gluing those  local gradient-like vector fields, we get a globally 
defined  gradient-like vector field $V$ on $\tilde M$ whose support is contained in 
a neighborhood  of $M$. It is now straightforward to obtain a diffeomorphism 
between $\tilde M$ and $W\times [0,1]$ by means of integral curves of $V$. 
\end{proof}

%%%%%%%%%%%%%%%%%%%%%%%%%%%%%%%%%%%%%%%%%%%%%%%%%%%%%%%%%%%%%%%%%%%%%%%%%%%%%%%%%%%%%%%%%%%
\begin{thm}[Gromov\cite{G:synthetic}, Alexander-Bishop\cite{AB}] \label{thm:GAB}
There exists a positive number $\e=\e(n,\kappa,\lambda)$ such that if 
$M\in\mathcal M(n,\kappa,\lambda)$ has the two side bounds on 
sectional curvature $|K_M|\le\kappa^2$ in addition and if the inradius ${\rm inrad}(M)<\e$,
then either $M$ or its double cover is diffeomorphic to a product $W\times [0,1]$, where
$W$ is a closed manifold.
\end{thm}

The following example shows that Theorem \ref{thm:GAB} does not hold in the 
connected boundary case 
if one drops the upper sectional curvature bound $K_M\le \kappa^2$.
Namely there are some $M\in \mathcal M(n,\kappa,\lambda,d)$
with connected boundary and with small inradius that are not finitely covered by 
any topological product of the form $W\times [0,1]$, where $W$ is a closed manifold.

\begin{ex}
Let $N$ be a compact surface of genus one with connected boundary, and consider 
a Riemannian metric on $N$ such that $\pa N$ has a cylindrical neighborhood $U_{\e}$.
Namely there is an isometric embedding $f:S^1_{\ell}\times [0,\e) \to U_{\e}$ such that
$f(S^1_{\ell}\times 0)=\pa N$., where $\ell = L(\pa N)$.
Consider a segment $I=\{ (x, 0, 0)\, |\, 0\le x\le 2\e\,\}$ in the $xyz$-space $\mathbb R^3$, and
let $D_{\e}$ denote the intersection of the boundary of $\e$-neighborhood of $I$ with 
$\{ x\le \e, z\le 0\}$. Let 
\begin{align*}
             &J_{\e} := D_{\e}\cap \{ x=\e\, \}, \,\, K_{\e} := D_{\e}\cap \{ x \le 0, z=0\,\}, 
           \,\,  \\
             &L_{\e} := D_{\e}\cap \{ z=-\e\, \}, \,\, E_{\e}:=D_{\e}\cap \{ 0\le x\le \e\}.
\end{align*}
Note that $J_{\e}$ and $K_{\e}$
 (resp. $L_{\e}$) 
are segments of length $\pi\e$ (resp. length $\e$). 
Since there is an isometry 
$\varphi:\bar U_{\e}\times [-\pi\e/2, \pi\e/2] \to S^1_{\ell}\times E_{\e}$, 
we have an obvious gluing to obtain three-dimensional Riemannian manifold $M_{\e}$
with totally geodesic boundary:
\[
       M_{\e} = N\times  [-\pi\e/2, \pi\e/2] \amalg_{\varphi} S^1_{\ell}\times D_{\e}.
\]
Note that after slight smoothing of $M_{\e}$, we may assume that $M_{\e}\in \mathcal M(3,\kappa,0,d)$
for some $\kappa$, $d$ and it inradius collapses to $N$ as $\e\to 0$.
Note that $M_{\e}$ is homeomorphic to $N\times I\cup \pa N\times D^2$ as in Theorem \ref{thm:codim1-detail}.

Note that any finite cover $\hat M_{\e}$ of $M_{\e}$ is not homeomorphic to $W\times [0,1]$,
for any closed surface.
Otherwise  a finite cover $\hat M_{\e}$ of $M_{\e}$ is homeomorphic to $W\times [0,1]$ as above.
Since $M_{\e}$ has the same homotopy type as $N$, $\pi_1(M_{\e})$ is a free group generated by two elements. It turns out that  $\pi_1(\hat M_{\e})=\pi_1(W)$ is a free group, which is a contradiction.
\end{ex}

%%%%%%%%%%%%%%%%%%%%%%%%%%%%%%%%%%%%%%%%%%%%%%%%%%%%%%%%%%%%%%
\section{Remark on locally convex manifolds}

In the argument so far, the assumption  $|{\rm II}_{\pa M}|\le \lambda^2$ was used to have 
lower sectional curvature bound $K_{\pa M} \ge c(\kappa,\lambda)$.
It is a challenging problem to study the case where only lower bound 
${\rm II}_{\pa M} \ge -\kappa$ is assumed.

In the case of locally convex boundary in the sense that  ${\rm II}_{\pa M} \ge 0$, 
the Gauss equation implies  $K_{\pa M} \ge \kappa$ as long as  $K_M \ge \kappa$.
Therefore taking $\phi(t)=1$ as the warping function, we can extend $M$ to 
$\tilde M = M\cup \pa M\times [0,t_0]$, and  proceed by the same argument as in 
the previous sections, to obtain the results corresponding to Theorems \ref{thm:codim1},
\ref{thm:RMBcap} and \ref{thm:two}.

\end{document}